\documentclass[11pt]{article}

\usepackage{amsmath, amsfonts, amsthm, mathrsfs, tikz, url}
\usepackage[british]{babel}

\makeatletter
\@addtoreset{equation}{section}
\makeatother

\theoremstyle{plain}
\newtheorem{proposition}{Proposition}[section]
\newtheorem{theorem}[proposition]{Theorem}

\newtheorem{lemma}[proposition]{Lemma}
\newtheorem{corollary}[proposition]{Corollary}

\newtheorem{remark}[proposition]{Remark}

\theoremstyle{definition}
\newtheorem*{remark*}{Remark}

\newcommand{\N}{\mathbb{N}}

\newcommand{\R}{\mathbb{R}}
\newcommand{\Z}{\mathbb{Z}}
\newcommand{\E}{\mathbb{E}}

\newcommand{\1}{1}

\newcommand{\tr}{{\rm tr}}
\newcommand{\cR}{{\mathcal R}}

\renewcommand{\Pr}{\mathbb{P}}

\newcommand{\WM}{\mathscr{W}}
\newcommand{\block}{\mathsf{sblock}}
\newcommand{\wlim}{\mathop{\text{\rm w-lim}}}

\begin{document}

\title{A quenched large deviation principle in a continuous scenario} 

\date{9th~December~2013}

\author{
\renewcommand{\thefootnote}{\arabic{footnote}}
M.\ Birkner
\footnotemark[1]
\\
\renewcommand{\thefootnote}{\arabic{footnote}}
F.\ den Hollander
\footnotemark[2]
}

\footnotetext[1]{
Institut f\"{u}r Mathematik, Johannes-Gutenberg-Universit\"at,
Staudingerweg 9, 55099 Mainz, Germany,\\ 
{\sl birkner@mathematik.uni-mainz.de}
}
\footnotetext[2]{
Mathematical Institute, Leiden University, P.O.\ Box 9512,
2300 RA Leiden, The Netherlands,\\
{\sl denholla@math.leidenuniv.nl}
}

\maketitle

\begin{abstract} 
We prove the analogue for continuous space-time of the quenched LDP
derived in Birkner, Greven and den Hollander~\cite{BiGrdHo10} for
discrete space-time. In particular, we consider a random environment
given by Brownian increments, cut into pieces according to an
independent continuous-time renewal process. We look at the empirical
process obtained by recording both the length of and the increments in 
the successive pieces. For the case where the renewal time distribution 
has a Lebesgue density with a polynomial tail, we derive the quenched 
LDP for the empirical process, i.e., the LDP conditional on a typical 
environment. The rate function is a sum of two specific relative entropies,
one for the pieces and one for the concatenation of the pieces. 
We also obtain a quenched LDP when the tail decays faster than algebraic. 
The proof uses coarse-graining and truncation arguments, 
involving various approximations of specific relative entropies that are
not quite standard. 

In a companion paper we show how the quenched LDP 
and the techniques developed in the present paper can be applied to obtain 
a variational characterisation of the free energy and the phase transition line 
for the Brownian copolymer near a selective interface.

\emph{Acknowledgment:} The research in this paper is supported by ERC Advanced Grant
267356 VARIS of FdH. MB is grateful for hospitality at the Mathematical Institute in Leiden 
during a sabbatical leave from September 2012 until February 2013, supported by ERC.
\end{abstract}


\section{Introduction and main result}
\label{intro}

When we cut an i.i.d.\ sequence of letters into words according to an independent
integer-valued renewal process, we obtain an i.i.d.\ sequence of words. In the 
\emph{annealed} LDP for the empirical process of words, the rate function is the 
specific relative entropy of the observed law of words w.r.t.\ the reference law 
of words. Birkner, Greven and den Hollander~\cite{BiGrdHo10} considered the 
\emph{quenched} LDP, i.e., conditional on a typical letter sequence. The rate 
function of the quenched LDP turned out to be a sum of two terms, one being the 
annealed rate function, the other being proportional to the specific relative 
entropy of the observed law of letters w.r.t.\ the reference law of letters, with 
the former being obtained by concatenating the words and randomising the location 
of the origin. The proportionality constant equals the tail exponent of the renewal 
time distribution.
     
The goal of the present paper is to derive the analogue of the quenched LDP for
the case where the i.i.d.\ sequence of letters is replaced by the process of
Brownian increments, and the renewal process has a length distribution with a 
Lebesgue density that has a polynomial tail. 

In Section~\ref{setting} we define the continuous space-time setting, in 
Section~\ref{LDPs} we state both the annealed and the quenched LDP, 
while in Section~\ref{disc} we discuss these LDPs and indicate some further 
extensions. In Section~\ref{proof} we prove the quenched LDP subject to 
three propositions. In Sections~\ref{props}--\ref{removeass} we give the 
proof of these propositions. In Section~\ref{proofalpha1infty} we prove the
extensions. Appendix~\ref{metrics} recalls a few basic facts about metrics 
on path space, while Appendices~\ref{entropy}--\ref{contrelentr} prove a 
few basic facts about specific relative entropy that are needed in the proof 
and that are not quite standard.
 

\subsection{Continuous space-time}
\label{setting}

Let $X=(X_t)_{t \geq 0}$ be the standard one-dimensional Brownian motion starting 
from $X_0=0$. Let $\WM$ denote its law on path space: the Wiener measure on 
$C([0,\infty))$, equipped with the $\sigma$-algebra generated by the coordinate 
projections. Let $T=(T_i)_{i \in \N_0}$ ($T_0=0$) be an independent continuous-time 
renewal process, with interarrival times $\tau_i=T_i-T_{i-1}$, $i\in\N$, whose 
common law $\rho=\mathscr{L}(\tau_1)$ is absolutely continuous with respect to 
the Lebesgue measure on $(0,\infty)$, with density $\bar{\rho}$ satisfying 
\begin{equation} 
\label{ass:rhodensdecay} 
\lim_{x\to\infty} \frac{\log\bar{\rho}(x)}{\log x} = - \alpha,
\qquad \alpha \in (1,\infty). 
\end{equation} 
In addition, assume that 
\begin{equation} 
\label{ass:rhobar.reg0} 
\begin{minipage}{0.85\textwidth}
$\mathrm{supp}(\rho) = [s_*,\infty)$ with $0 \leq s_* < \infty$, and $\bar{\rho}$ is 
continuous and strictly positive on $(s_*,\infty)$, and varies regularly near $s_*$.
\end{minipage}
\end{equation}

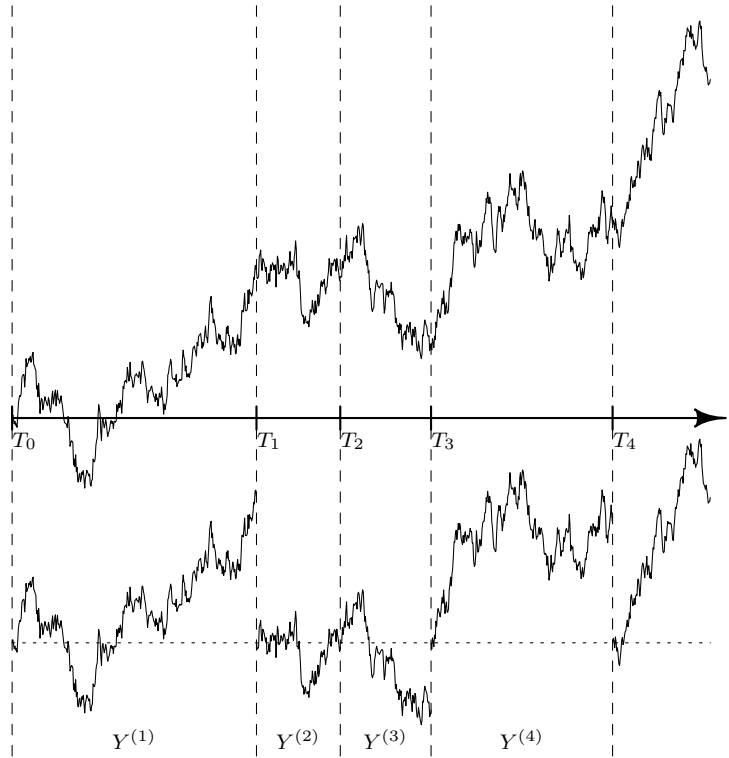
\begin{figure}[htbp]
\begin{center}
\begin{tikzpicture}[x=0.8pt,y=0.8pt]
\definecolor[named]{fillColor}{rgb}{1.00,1.00,1.00}
\path[use as bounding box,fill=fillColor,fill opacity=0.00] (0,0) rectangle (361.35,361.35);
\begin{scope}
\path[clip] (  2.40,  2.40) rectangle (358.95,358.95);
\definecolor[named]{drawColor}{rgb}{0.00,0.00,0.00}

\path[draw=drawColor,line width= 0.8pt,line join=round,line cap=round] ( 15.61,162.33) -- (345.74,162.33);
\definecolor[named]{fillColor}{rgb}{0.00,0.00,0.00}

\path[draw=drawColor,line width= 0.8pt,line join=round,line cap=round,fill=fillColor] (338.42,166.36) --
	(341.56,164.65) --
	(345.14,163.38) --
	(349.03,162.60) --
	(353.07,162.33) --
	(353.07,162.33) --
	(349.03,162.07) --
	(345.14,161.29) --
	(341.56,160.02) --
	(338.42,158.31) --
	(338.42,158.31) --
	(338.76,158.41) --
	(339.08,158.71) --
	(339.37,159.19) --
	(339.61,159.82) --
	(339.79,160.59) --
	(339.90,161.44) --
	(339.94,162.33) --
	(339.90,163.23) --
	(339.79,164.08) --
	(339.61,164.84) --
	(339.37,165.48) --
	(339.08,165.96) --
	(338.76,166.26) --
	(338.42,166.36) --
	cycle;

\path[draw=drawColor,line width= 0.4pt,line join=round,line cap=round] ( 15.61,162.33) --
	( 15.94,163.78) --
	( 16.27,161.59) --
	( 16.60,159.54) --
	( 16.93,160.84) --
	( 17.26,158.90) --
	( 17.59,159.60) --
	( 17.92,157.91) --
	( 18.25,162.99) --
	( 18.58,169.37) --
	( 18.91,171.53) --
	( 19.24,172.96) --
	( 19.57,172.49) --
	( 19.90,174.03) --
	( 20.23,181.44) --
	( 20.56,182.39) --
	( 20.89,182.86) --
	( 21.22,182.87) --
	( 21.55,176.96) --
	( 21.88,181.97) --
	( 22.21,190.43) --
	( 22.54,186.60) --
	( 22.87,187.90) --
	( 23.20,188.38) --
	( 23.53,190.46) --
	( 23.86,189.69) --
	( 24.19,191.41) --
	( 24.52,188.72) --
	( 24.85,189.63) --
	( 25.18,192.42) --
	( 25.51,193.43) --
	( 25.84,187.48) --
	( 26.17,187.08) --
	( 26.50,185.68) --
	( 26.83,184.93) --
	( 27.16,176.03) --
	( 27.49,177.69) --
	( 27.82,172.43) --
	( 28.15,178.64) --
	( 28.48,173.22) --
	( 28.81,164.52) --
	( 29.14,169.59) --
	( 29.47,171.64) --
	( 29.80,171.10) --
	( 30.13,166.26) --
	( 30.46,168.12) --
	( 30.79,170.35) --
	( 31.12,172.40) --
	( 31.45,170.64) --
	( 31.78,170.16) --
	( 32.11,168.50) --
	( 32.44,168.21) --
	( 32.77,171.97) --
	( 33.10,170.81) --
	( 33.43,167.06) --
	( 33.76,165.03) --
	( 34.09,167.83) --
	( 34.42,172.37) --
	( 34.75,175.09) --
	( 35.08,168.99) --
	( 35.41,171.08) --
	( 35.74,173.81) --
	( 36.07,175.59) --
	( 36.40,168.07) --
	( 36.73,172.84) --
	( 37.06,175.36) --
	( 37.39,171.05) --
	( 37.72,171.36) --
	( 38.05,169.27) --
	( 38.39,170.48) --
	( 38.72,172.81) --
	( 39.05,170.21) --
	( 39.38,168.26) --
	( 39.71,165.13) --
	( 40.04,162.23) --
	( 40.37,160.18) --
	( 40.70,159.97) --
	( 41.03,159.69) --
	( 41.36,159.26) --
	( 41.69,157.81) --
	( 42.02,163.04) --
	( 42.35,161.43) --
	( 42.68,156.62) --
	( 43.01,148.97) --
	( 43.34,150.66) --
	( 43.67,151.58) --
	( 44.00,151.35) --
	( 44.33,150.49) --
	( 44.66,140.15) --
	( 44.99,142.98) --
	( 45.32,139.32) --
	( 45.65,144.90) --
	( 45.98,136.91) --
	( 46.31,138.99) --
	( 46.64,132.30) --
	( 46.97,132.08) --
	( 47.30,135.59) --
	( 47.63,136.14) --
	( 47.96,132.33) --
	( 48.29,133.80) --
	( 48.62,136.85) --
	( 48.95,135.64) --
	( 49.28,137.39) --
	( 49.61,133.43) --
	( 49.94,129.17) --
	( 50.27,133.93) --
	( 50.60,135.33) --
	( 50.93,137.62) --
	( 51.26,136.47) --
	( 51.59,136.07) --
	( 51.92,134.90) --
	( 52.25,135.79) --
	( 52.58,130.82) --
	( 52.91,130.10) --
	( 53.24,132.73) --
	( 53.57,133.36) --
	( 53.90,142.93) --
	( 54.23,141.95) --
	( 54.56,141.64) --
	( 54.89,143.25) --
	( 55.22,149.41) --
	( 55.55,151.74) --
	( 55.88,153.06) --
	( 56.21,162.66) --
	( 56.54,168.64) --
	( 56.87,166.72) --
	( 57.20,164.37) --
	( 57.53,155.91) --
	( 57.86,151.62) --
	( 58.19,160.67) --
	( 58.52,162.00) --
	( 58.85,159.70) --
	( 59.18,157.72) --
	( 59.51,159.15) --
	( 59.84,161.56) --
	( 60.17,159.91) --
	( 60.50,156.12) --
	( 60.83,153.61) --
	( 61.16,155.32) --
	( 61.49,156.12) --
	( 61.82,158.02) --
	( 62.16,160.85) --
	( 62.49,161.11) --
	( 62.82,159.73) --
	( 63.15,160.63) --
	( 63.48,157.34) --
	( 63.81,160.60) --
	( 64.14,160.85) --
	( 64.47,160.34) --
	( 64.80,160.90) --
	( 65.13,168.70) --
	( 65.46,165.38) --
	( 65.79,166.83) --
	( 66.12,169.33) --
	( 66.45,167.26) --
	( 66.78,170.76) --
	( 67.11,168.82) --
	( 67.44,164.58) --
	( 67.77,170.12) --
	( 68.10,177.37) --
	( 68.43,176.15) --
	( 68.76,179.62) --
	( 69.09,179.71) --
	( 69.42,180.19) --
	( 69.75,178.91) --
	( 70.08,186.32) --
	( 70.41,183.84) --
	( 70.74,179.87) --
	( 71.07,180.00) --
	( 71.40,188.72) --
	( 71.73,182.11) --
	( 72.06,181.57) --
	( 72.39,180.12) --
	( 72.72,176.85) --
	( 73.05,177.30) --
	( 73.38,179.62) --
	( 73.71,182.55) --
	( 74.04,182.70) --
	( 74.37,183.30) --
	( 74.70,182.27) --
	( 75.03,182.22) --
	( 75.36,181.50) --
	( 75.69,184.22) --
	( 76.02,183.65) --
	( 76.35,178.58) --
	( 76.68,182.69) --
	( 77.01,188.94) --
	( 77.34,189.21) --
	( 77.67,186.79) --
	( 78.00,185.20) --
	( 78.33,182.63) --
	( 78.66,180.80) --
	( 78.99,177.16) --
	( 79.32,175.06) --
	( 79.65,169.62) --
	( 79.98,172.75) --
	( 80.31,175.36) --
	( 80.64,175.16) --
	( 80.97,175.56) --
	( 81.30,177.05) --
	( 81.63,177.49) --
	( 81.96,175.25) --
	( 82.29,174.08) --
	( 82.62,172.44) --
	( 82.95,172.27) --
	( 83.28,170.18) --
	( 83.61,173.79) --
	( 83.94,170.97) --
	( 84.27,171.84) --
	( 84.60,168.01) --
	( 84.93,171.34) --
	( 85.26,174.13) --
	( 85.59,172.13) --
	( 85.93,172.15) --
	( 86.26,177.04) --
	( 86.59,172.69) --
	( 86.92,170.13) --
	( 87.25,164.17) --
	( 87.58,168.98) --
	( 87.91,166.36) --
	( 88.24,168.31) --
	( 88.57,170.04) --
	( 88.90,172.97) --
	( 89.23,182.76) --
	( 89.56,183.62) --
	( 89.89,186.49) --
	( 90.22,187.63) --
	( 90.55,190.17) --
	( 90.88,188.71) --
	( 91.21,182.24) --
	( 91.54,180.87) --
	( 91.87,180.77) --
	( 92.20,180.59) --
	( 92.53,181.03) --
	( 92.86,179.60) --
	( 93.19,182.92) --
	( 93.52,177.21) --
	( 93.85,178.44) --
	( 94.18,179.21) --
	( 94.51,179.97) --
	( 94.84,180.61) --
	( 95.17,180.12) --
	( 95.50,180.53) --
	( 95.83,180.34) --
	( 96.16,190.58) --
	( 96.49,194.48) --
	( 96.82,191.22) --
	( 97.15,190.34) --
	( 97.48,184.60) --
	( 97.81,186.95) --
	( 98.14,181.80) --
	( 98.47,183.21) --
	( 98.80,181.83) --
	( 99.13,183.31) --
	( 99.46,181.57) --
	( 99.79,183.81) --
	(100.12,183.91) --
	(100.45,187.30) --
	(100.78,191.76) --
	(101.11,195.42) --
	(101.44,195.63) --
	(101.77,196.44) --
	(102.10,194.73) --
	(102.43,193.23) --
	(102.76,191.60) --
	(103.09,198.12) --
	(103.42,196.52) --
	(103.75,200.69) --
	(104.08,201.54) --
	(104.41,201.90) --
	(104.74,202.27) --
	(105.07,199.36) --
	(105.40,195.38) --
	(105.73,198.02) --
	(106.06,196.32) --
	(106.39,198.78) --
	(106.72,202.74) --
	(107.05,205.89) --
	(107.38,200.59) --
	(107.71,200.39) --
	(108.04,197.96) --
	(108.37,203.79) --
	(108.70,209.82) --
	(109.03,216.18) --
	(109.36,218.39) --
	(109.70,219.87) --
	(110.03,212.15) --
	(110.36,214.22) --
	(110.69,207.46) --
	(111.02,208.96) --
	(111.35,207.20) --
	(111.68,205.55) --
	(112.01,201.79) --
	(112.34,197.57) --
	(112.67,197.28) --
	(113.00,199.80) --
	(113.33,202.19) --
	(113.66,199.62) --
	(113.99,201.48) --
	(114.32,198.95) --
	(114.65,198.88) --
	(114.98,195.00) --
	(115.31,198.39) --
	(115.64,194.21) --
	(115.97,196.46) --
	(116.30,195.87) --
	(116.63,198.92) --
	(116.96,205.32) --
	(117.29,202.39) --
	(117.62,205.91) --
	(117.95,201.42) --
	(118.28,201.44) --
	(118.61,197.01) --
	(118.94,192.33) --
	(119.27,196.29) --
	(119.60,197.02) --
	(119.93,200.61) --
	(120.26,199.05) --
	(120.59,195.86) --
	(120.92,197.09) --
	(121.25,192.65) --
	(121.58,193.14) --
	(121.91,199.03) --
	(122.24,201.72) --
	(122.57,196.94) --
	(122.90,199.67) --
	(123.23,196.29) --
	(123.56,196.94) --
	(123.89,204.30) --
	(124.22,213.13) --
	(124.55,213.25) --
	(124.88,212.46) --
	(125.21,216.19) --
	(125.54,221.37) --
	(125.87,213.17) --
	(126.20,213.06) --
	(126.53,213.21) --
	(126.86,214.20) --
	(127.19,222.84) --
	(127.52,216.51) --
	(127.85,223.10) --
	(128.18,221.10) --
	(128.51,221.54) --
	(128.84,220.01) --
	(129.17,221.11) --
	(129.50,223.96) --
	(129.83,228.91) --
	(130.16,227.96) --
	(130.49,234.23) --
	(130.82,234.61) --
	(131.15,231.28) --
	(131.48,228.32) --
	(131.81,228.57) --
	(132.14,229.23) --
	(132.47,239.02) --
	(132.80,239.74) --
	(133.13,241.87) --
	(133.47,234.42) --
	(133.80,232.14) --
	(134.13,237.75) --
	(134.46,237.07) --
	(134.79,235.75) --
	(135.12,233.85) --
	(135.45,231.74) --
	(135.78,233.66) --
	(136.11,225.93) --
	(136.44,226.73) --
	(136.77,227.52) --
	(137.10,230.76) --
	(137.43,227.00) --
	(137.76,234.79) --
	(138.09,232.12) --
	(138.42,234.66) --
	(138.75,235.80) --
	(139.08,234.71) --
	(139.41,232.44) --
	(139.74,234.44) --
	(140.07,234.81) --
	(140.40,224.35) --
	(140.73,232.53) --
	(141.06,238.96) --
	(141.39,237.36) --
	(141.72,230.46) --
	(142.05,230.84) --
	(142.38,231.28) --
	(142.71,234.70) --
	(143.04,234.30) --
	(143.37,232.08) --
	(143.70,232.70) --
	(144.03,235.58) --
	(144.36,235.07) --
	(144.69,233.43) --
	(145.02,228.12) --
	(145.35,228.37) --
	(145.68,232.98) --
	(146.01,230.77) --
	(146.34,233.58) --
	(146.67,236.81) --
	(147.00,241.11) --
	(147.33,234.45) --
	(147.66,235.57) --
	(147.99,236.56) --
	(148.32,235.01) --
	(148.65,232.37) --
	(148.98,237.98) --
	(149.31,241.02) --
	(149.64,246.15) --
	(149.97,242.29) --
	(150.30,229.52) --
	(150.63,229.61) --
	(150.96,228.39) --
	(151.29,232.02) --
	(151.62,227.27) --
	(151.95,221.02) --
	(152.28,213.36) --
	(152.61,211.49) --
	(152.94,210.32) --
	(153.27,210.94) --
	(153.60,206.97) --
	(153.93,209.02) --
	(154.26,210.39) --
	(154.59,207.60) --
	(154.92,207.49) --
	(155.25,207.10) --
	(155.58,208.32) --
	(155.91,205.47) --
	(156.24,208.75) --
	(156.57,213.58) --
	(156.90,212.61) --
	(157.24,212.47) --
	(157.57,217.98) --
	(157.90,214.83) --
	(158.23,212.02) --
	(158.56,211.23) --
	(158.89,217.78) --
	(159.22,211.94) --
	(159.55,214.47) --
	(159.88,216.76) --
	(160.21,220.74) --
	(160.54,215.35) --
	(160.87,217.00) --
	(161.20,220.68) --
	(161.53,226.07) --
	(161.86,227.25) --
	(162.19,224.58) --
	(162.52,226.61) --
	(162.85,225.82) --
	(163.18,227.91) --
	(163.51,224.23) --
	(163.84,221.02) --
	(164.17,223.97) --
	(164.50,224.40) --
	(164.83,228.81) --
	(165.16,231.43) --
	(165.49,237.76) --
	(165.82,235.92) --
	(166.15,230.46) --
	(166.48,229.64) --
	(166.81,235.66) --
	(167.14,234.52) --
	(167.47,234.84) --
	(167.80,236.71) --
	(168.13,237.52) --
	(168.46,236.74) --
	(168.79,231.86) --
	(169.12,237.77) --
	(169.45,235.00) --
	(169.78,233.38) --
	(170.11,230.54) --
	(170.44,227.22) --
	(170.77,229.26) --
	(171.10,229.10) --
	(171.43,229.16) --
	(171.76,236.46) --
	(172.09,233.20) --
	(172.42,236.14) --
	(172.75,239.23) --
	(173.08,236.66) --
	(173.41,244.17) --
	(173.74,248.64) --
	(174.07,239.55) --
	(174.40,237.44) --
	(174.73,238.78) --
	(175.06,236.14) --
	(175.39,240.74) --
	(175.72,240.68) --
	(176.05,241.19) --
	(176.38,237.89) --
	(176.71,242.98) --
	(177.04,247.72) --
	(177.37,246.98) --
	(177.70,247.28) --
	(178.03,248.02) --
	(178.36,248.41) --
	(178.69,245.31) --
	(179.02,252.07) --
	(179.35,253.23) --
	(179.68,252.12) --
	(180.01,248.55) --
	(180.34,244.15) --
	(180.67,246.85) --
	(181.01,251.95) --
	(181.34,254.35) --
	(181.67,253.16) --
	(182.00,245.95) --
	(182.33,239.46) --
	(182.66,241.70) --
	(182.99,236.71) --
	(183.32,234.62) --
	(183.65,236.87) --
	(183.98,230.99) --
	(184.31,232.04) --
	(184.64,223.22) --
	(184.97,223.69) --
	(185.30,219.03) --
	(185.63,214.00) --
	(185.96,212.67) --
	(186.29,213.98) --
	(186.62,214.49) --
	(186.95,221.15) --
	(187.28,222.66) --
	(187.61,225.91) --
	(187.94,227.62) --
	(188.27,226.36) --
	(188.60,224.53) --
	(188.93,224.52) --
	(189.26,223.79) --
	(189.59,224.46) --
	(189.92,225.43) --
	(190.25,225.91) --
	(190.58,226.08) --
	(190.91,225.43) --
	(191.24,218.70) --
	(191.57,211.95) --
	(191.90,213.77) --
	(192.23,218.90) --
	(192.56,223.07) --
	(192.89,220.63) --
	(193.22,222.21) --
	(193.55,219.58) --
	(193.88,222.59) --
	(194.21,221.76) --
	(194.54,221.55) --
	(194.87,226.57) --
	(195.20,220.80) --
	(195.53,222.11) --
	(195.86,221.04) --
	(196.19,219.93) --
	(196.52,215.40) --
	(196.85,207.54) --
	(197.18,209.61) --
	(197.51,206.69) --
	(197.84,205.55) --
	(198.17,206.15) --
	(198.50,202.37) --
	(198.83,207.16) --
	(199.16,204.30) --
	(199.49,206.52) --
	(199.82,199.35) --
	(200.15,194.79) --
	(200.48,200.97) --
	(200.81,202.15) --
	(201.14,201.79) --
	(201.47,203.54) --
	(201.80,199.50) --
	(202.13,205.55) --
	(202.46,206.47) --
	(202.79,205.18) --
	(203.12,202.96) --
	(203.45,206.30) --
	(203.78,207.63) --
	(204.11,202.62) --
	(204.44,200.70) --
	(204.78,196.79) --
	(205.11,204.41) --
	(205.44,201.80) --
	(205.77,200.74) --
	(206.10,205.37) --
	(206.43,204.00) --
	(206.76,197.62) --
	(207.09,194.13) --
	(207.42,194.06) --
	(207.75,198.72) --
	(208.08,196.40) --
	(208.41,193.98) --
	(208.74,192.00) --
	(209.07,190.51) --
	(209.40,192.45) --
	(209.73,201.69) --
	(210.06,201.17) --
	(210.39,205.64) --
	(210.72,208.75) --
	(211.05,207.28) --
	(211.38,206.20) --
	(211.71,205.38) --
	(212.04,201.88) --
	(212.37,200.89) --
	(212.70,201.29) --
	(213.03,194.23) --
	(213.36,195.33) --
	(213.69,197.48) --
	(214.02,194.52) --
	(214.35,199.85) --
	(214.68,198.18) --
	(215.01,195.67) --
	(215.34,201.15) --
	(215.67,199.51) --
	(216.00,207.69) --
	(216.33,205.32) --
	(216.66,210.90) --
	(216.99,213.68) --
	(217.32,212.47) --
	(217.65,211.42) --
	(217.98,209.79) --
	(218.31,211.95) --
	(218.64,216.32) --
	(218.97,218.96) --
	(219.30,216.24) --
	(219.63,222.50) --
	(219.96,217.14) --
	(220.29,211.39) --
	(220.62,214.94) --
	(220.95,216.84) --
	(221.28,211.99) --
	(221.61,212.00) --
	(221.94,213.57) --
	(222.27,222.54) --
	(222.60,224.81) --
	(222.93,228.58) --
	(223.26,224.40) --
	(223.59,230.01) --
	(223.92,232.62) --
	(224.25,238.72) --
	(224.58,245.31) --
	(224.91,245.21) --
	(225.24,249.23) --
	(225.57,250.78) --
	(225.90,246.38) --
	(226.23,249.82) --
	(226.56,253.98) --
	(226.89,251.08) --
	(227.22,253.96) --
	(227.55,252.86) --
	(227.88,247.13) --
	(228.21,245.41) --
	(228.55,243.95) --
	(228.88,248.27) --
	(229.21,244.80) --
	(229.54,248.41) --
	(229.87,250.08) --
	(230.20,246.72) --
	(230.53,245.55) --
	(230.86,248.37) --
	(231.19,248.10) --
	(231.52,247.01) --
	(231.85,248.99) --
	(232.18,251.64) --
	(232.51,250.12) --
	(232.84,250.16) --
	(233.17,244.70) --
	(233.50,244.72) --
	(233.83,236.17) --
	(234.16,233.66) --
	(234.49,234.23) --
	(234.82,239.14) --
	(235.15,250.48) --
	(235.48,245.34) --
	(235.81,240.36) --
	(236.14,240.69) --
	(236.47,243.92) --
	(236.80,246.40) --
	(237.13,243.28) --
	(237.46,244.01) --
	(237.79,251.20) --
	(238.12,254.25) --
	(238.45,255.04) --
	(238.78,255.00) --
	(239.11,259.04) --
	(239.44,259.14) --
	(239.77,266.13) --
	(240.10,264.21) --
	(240.43,269.21) --
	(240.76,266.62) --
	(241.09,263.77) --
	(241.42,266.82) --
	(241.75,264.14) --
	(242.08,264.15) --
	(242.41,254.93) --
	(242.74,251.71) --
	(243.07,246.16) --
	(243.40,242.13) --
	(243.73,241.97) --
	(244.06,240.62) --
	(244.39,241.02) --
	(244.72,249.78) --
	(245.05,256.52) --
	(245.38,258.86) --
	(245.71,262.44) --
	(246.04,260.04) --
	(246.37,262.31) --
	(246.70,257.04) --
	(247.03,256.15) --
	(247.36,250.51) --
	(247.69,253.00) --
	(248.02,258.47) --
	(248.35,255.29) --
	(248.68,259.00) --
	(249.01,261.94) --
	(249.34,263.23) --
	(249.67,262.69) --
	(250.00,267.16) --
	(250.33,267.94) --
	(250.66,271.27) --
	(250.99,278.01) --
	(251.32,273.25) --
	(251.65,270.81) --
	(251.98,270.29) --
	(252.32,270.85) --
	(252.65,269.61) --
	(252.98,267.66) --
	(253.31,271.13) --
	(253.64,269.60) --
	(253.97,263.28) --
	(254.30,263.28) --
	(254.63,260.89) --
	(254.96,266.78) --
	(255.29,265.33) --
	(255.62,275.53) --
	(255.95,277.96) --
	(256.28,272.13) --
	(256.61,275.83) --
	(256.94,279.30) --
	(257.27,278.67) --
	(257.60,275.19) --
	(257.93,270.61) --
	(258.26,273.76) --
	(258.59,272.35) --
	(258.92,265.35) --
	(259.25,261.65) --
	(259.58,258.02) --
	(259.91,262.44) --
	(260.24,254.63) --
	(260.57,252.54) --
	(260.90,250.93) --
	(261.23,251.18) --
	(261.56,250.71) --
	(261.89,246.91) --
	(262.22,251.58) --
	(262.55,254.17) --
	(262.88,256.25) --
	(263.21,249.34) --
	(263.54,246.51) --
	(263.87,247.81) --
	(264.20,247.92) --
	(264.53,251.01) --
	(264.86,252.13) --
	(265.19,248.78) --
	(265.52,251.45) --
	(265.85,251.79) --
	(266.18,249.59) --
	(266.51,248.59) --
	(266.84,244.71) --
	(267.17,240.02) --
	(267.50,239.30) --
	(267.83,239.88) --
	(268.16,239.07) --
	(268.49,234.20) --
	(268.82,238.51) --
	(269.15,233.07) --
	(269.48,227.14) --
	(269.81,229.75) --
	(270.14,232.43) --
	(270.47,238.94) --
	(270.80,233.41) --
	(271.13,231.07) --
	(271.46,235.33) --
	(271.79,234.62) --
	(272.12,234.92) --
	(272.45,240.07) --
	(272.78,242.76) --
	(273.11,244.46) --
	(273.44,251.92) --
	(273.77,249.04) --
	(274.10,250.86) --
	(274.43,246.15) --
	(274.76,244.27) --
	(275.09,238.48) --
	(275.42,240.93) --
	(275.75,241.16) --
	(276.09,240.75) --
	(276.42,247.29) --
	(276.75,245.28) --
	(277.08,247.95) --
	(277.41,251.00) --
	(277.74,251.67) --
	(278.07,250.16) --
	(278.40,249.70) --
	(278.73,250.12) --
	(279.06,258.25) --
	(279.39,252.94) --
	(279.72,240.96) --
	(280.05,243.59) --
	(280.38,242.00) --
	(280.71,242.91) --
	(281.04,237.02) --
	(281.37,235.68) --
	(281.70,235.73) --
	(282.03,237.13) --
	(282.36,234.99) --
	(282.69,231.86) --
	(283.02,234.34) --
	(283.35,233.88) --
	(283.68,235.64) --
	(284.01,235.89) --
	(284.34,231.60) --
	(284.67,227.84) --
	(285.00,230.97) --
	(285.33,233.05) --
	(285.66,228.66) --
	(285.99,230.28) --
	(286.32,233.51) --
	(286.65,237.46) --
	(286.98,236.76) --
	(287.31,240.35) --
	(287.64,240.32) --
	(287.97,247.11) --
	(288.30,250.31) --
	(288.63,249.51) --
	(288.96,250.87) --
	(289.29,252.69) --
	(289.62,249.74) --
	(289.95,248.08) --
	(290.28,253.85) --
	(290.61,252.54) --
	(290.94,250.68) --
	(291.27,252.45) --
	(291.60,247.94) --
	(291.93,250.95) --
	(292.26,251.98) --
	(292.59,249.74) --
	(292.92,257.42) --
	(293.25,251.20) --
	(293.58,257.37) --
	(293.91,260.03) --
	(294.24,265.07) --
	(294.57,269.85) --
	(294.90,267.27) --
	(295.23,264.21) --
	(295.56,266.82) --
	(295.89,262.48) --
	(296.22,257.93) --
	(296.55,247.13) --
	(296.88,248.21) --
	(297.21,244.84) --
	(297.54,250.42) --
	(297.87,251.20) --
	(298.20,254.59) --
	(298.53,258.18) --
	(298.86,262.72) --
	(299.19,257.03) --
	(299.52,253.77) --
	(299.86,255.90) --
	(300.19,254.51) --
	(300.52,251.41) --
	(300.85,256.59) --
	(301.18,254.11) --
	(301.51,249.85) --
	(301.84,251.42) --
	(302.17,247.50) --
	(302.50,243.68) --
	(302.83,243.08) --
	(303.16,246.73) --
	(303.49,250.11) --
	(303.82,255.08) --
	(304.15,256.28) --
	(304.48,252.91) --
	(304.81,254.17) --
	(305.14,257.11) --
	(305.47,259.03) --
	(305.80,259.48) --
	(306.13,262.14) --
	(306.46,264.52) --
	(306.79,261.40) --
	(307.12,263.11) --
	(307.45,267.93) --
	(307.78,265.37) --
	(308.11,275.41) --
	(308.44,276.54) --
	(308.77,273.85) --
	(309.10,271.35) --
	(309.43,272.37) --
	(309.76,274.79) --
	(310.09,271.90) --
	(310.42,275.56) --
	(310.75,276.66) --
	(311.08,276.26) --
	(311.41,283.76) --
	(311.74,283.90) --
	(312.07,280.50) --
	(312.40,275.85) --
	(312.73,279.26) --
	(313.06,285.42) --
	(313.39,289.36) --
	(313.72,291.51) --
	(314.05,291.02) --
	(314.38,287.60) --
	(314.71,284.52) --
	(315.04,288.90) --
	(315.37,289.94) --
	(315.70,294.69) --
	(316.03,286.10) --
	(316.36,290.07) --
	(316.69,290.24) --
	(317.02,284.59) --
	(317.35,286.91) --
	(317.68,284.17) --
	(318.01,288.19) --
	(318.34,294.16) --
	(318.67,298.83) --
	(319.00,298.05) --
	(319.33,303.31) --
	(319.66,304.55) --
	(319.99,304.07) --
	(320.32,309.41) --
	(320.65,313.13) --
	(320.98,313.67) --
	(321.31,316.29) --
	(321.64,317.33) --
	(321.97,314.70) --
	(322.30,316.53) --
	(322.63,308.27) --
	(322.96,302.30) --
	(323.29,297.26) --
	(323.63,296.35) --
	(323.96,297.05) --
	(324.29,304.16) --
	(324.62,302.97) --
	(324.95,302.08) --
	(325.28,308.07) --
	(325.61,311.21) --
	(325.94,310.19) --
	(326.27,310.64) --
	(326.60,309.98) --
	(326.93,310.49) --
	(327.26,309.83) --
	(327.59,307.09) --
	(327.92,302.60) --
	(328.25,302.22) --
	(328.58,308.29) --
	(328.91,312.81) --
	(329.24,317.89) --
	(329.57,317.26) --
	(329.90,319.15) --
	(330.23,320.71) --
	(330.56,326.20) --
	(330.89,323.12) --
	(331.22,326.01) --
	(331.55,330.71) --
	(331.88,331.80) --
	(332.21,331.79) --
	(332.54,332.55) --
	(332.87,329.31) --
	(333.20,336.26) --
	(333.53,336.09) --
	(333.86,336.15) --
	(334.19,343.27) --
	(334.52,342.85) --
	(334.85,346.95) --
	(335.18,346.22) --
	(335.51,345.34) --
	(335.84,345.01) --
	(336.17,342.60) --
	(336.50,347.99) --
	(336.83,341.90) --
	(337.16,344.08) --
	(337.49,343.88) --
	(337.82,338.46) --
	(338.15,340.81) --
	(338.48,338.70) --
	(338.81,336.82) --
	(339.14,338.28) --
	(339.47,338.43) --
	(339.80,341.46) --
	(340.13,346.47) --
	(340.46,349.57) --
	(340.79,350.21) --
	(341.12,346.49) --
	(341.45,347.69) --
	(341.78,337.64) --
	(342.11,333.52) --
	(342.44,329.43) --
	(342.77,328.51) --
	(343.10,327.76) --
	(343.43,326.43) --
	(343.76,328.38) --
	(344.09,325.85) --
	(344.42,323.55) --
	(344.75,319.87) --
	(345.08,320.35) --
	(345.41,320.71) --
	(345.74,322.50);

\path[draw=drawColor,line width= 0.8pt,line join=round,line cap=round] ( 15.61,156.83) --
	( 15.61,167.84);

\node[text=drawColor,anchor=base,inner sep=0pt, outer sep=0pt, scale=  1.00] at ( 21.55,148.83) {$\scriptstyle T_0$};

\path[draw=drawColor,line width= 0.8pt,line join=round,line cap=round] (131.15,156.83) --
	(131.15,167.84);

\node[text=drawColor,anchor=base,inner sep=0pt, outer sep=0pt, scale=  1.00] at (137.10,148.83) {$\scriptstyle T_1$};

\path[draw=drawColor,line width= 0.8pt,line join=round,line cap=round] (170.77,156.83) --
	(170.77,167.84);

\node[text=drawColor,anchor=base,inner sep=0pt, outer sep=0pt, scale=  1.00] at (176.71,148.83) {$\scriptstyle T_2$};

\path[draw=drawColor,line width= 0.8pt,line join=round,line cap=round] (213.69,156.83) --
	(213.69,167.84);

\node[text=drawColor,anchor=base,inner sep=0pt, outer sep=0pt, scale=  1.00] at (219.63,148.83) {$\scriptstyle T_3$};

\path[draw=drawColor,line width= 0.8pt,line join=round,line cap=round] (299.52,156.83) --
	(299.52,167.84);

\node[text=drawColor,anchor=base,inner sep=0pt, outer sep=0pt, scale=  1.00] at (305.47,148.83) {$\scriptstyle T_4$};

\path[draw=drawColor,line width= 0.4pt,dash pattern=on 4pt off 4pt ,line join=round,line cap=round] ( 15.61,  2.40) -- ( 15.61,358.95);

\path[draw=drawColor,line width= 0.4pt,dash pattern=on 4pt off 4pt ,line join=round,line cap=round] (131.15,  2.40) -- (131.15,358.95);

\path[draw=drawColor,line width= 0.4pt,dash pattern=on 4pt off 4pt ,line join=round,line cap=round] (170.77,  2.40) -- (170.77,358.95);

\path[draw=drawColor,line width= 0.4pt,dash pattern=on 4pt off 4pt ,line join=round,line cap=round] (213.69,  2.40) -- (213.69,358.95);

\path[draw=drawColor,line width= 0.4pt,dash pattern=on 4pt off 4pt ,line join=round,line cap=round] (299.52,  2.40) -- (299.52,358.95);

\path[draw=drawColor,line width= 0.4pt,line join=round,line cap=round] ( 15.61, 55.96) --
	( 15.94, 57.40) --
	( 16.27, 55.21) --
	( 16.60, 53.16) --
	( 16.93, 54.46) --
	( 17.26, 52.52) --
	( 17.59, 53.22) --
	( 17.92, 51.53) --
	( 18.25, 56.61) --
	( 18.58, 62.99) --
	( 18.91, 65.15) --
	( 19.24, 66.58) --
	( 19.57, 66.11) --
	( 19.90, 67.65) --
	( 20.23, 75.06) --
	( 20.56, 76.02) --
	( 20.89, 76.48) --
	( 21.22, 76.49) --
	( 21.55, 70.58) --
	( 21.88, 75.59) --
	( 22.21, 84.06) --
	( 22.54, 80.22) --
	( 22.87, 81.52) --
	( 23.20, 82.00) --
	( 23.53, 84.09) --
	( 23.86, 83.31) --
	( 24.19, 85.04) --
	( 24.52, 82.34) --
	( 24.85, 83.26) --
	( 25.18, 86.04) --
	( 25.51, 87.06) --
	( 25.84, 81.11) --
	( 26.17, 80.70) --
	( 26.50, 79.30) --
	( 26.83, 78.55) --
	( 27.16, 69.65) --
	( 27.49, 71.31) --
	( 27.82, 66.05) --
	( 28.15, 72.26) --
	( 28.48, 66.84) --
	( 28.81, 58.15) --
	( 29.14, 63.21) --
	( 29.47, 65.26) --
	( 29.80, 64.72) --
	( 30.13, 59.88) --
	( 30.46, 61.74) --
	( 30.79, 63.97) --
	( 31.12, 66.02) --
	( 31.45, 64.26) --
	( 31.78, 63.78) --
	( 32.11, 62.12) --
	( 32.44, 61.83) --
	( 32.77, 65.59) --
	( 33.10, 64.43) --
	( 33.43, 60.69) --
	( 33.76, 58.66) --
	( 34.09, 61.45) --
	( 34.42, 65.99) --
	( 34.75, 68.72) --
	( 35.08, 62.61) --
	( 35.41, 64.70) --
	( 35.74, 67.43) --
	( 36.07, 69.22) --
	( 36.40, 61.70) --
	( 36.73, 66.47) --
	( 37.06, 68.99) --
	( 37.39, 64.67) --
	( 37.72, 64.98) --
	( 38.05, 62.89) --
	( 38.39, 64.10) --
	( 38.72, 66.43) --
	( 39.05, 63.83) --
	( 39.38, 61.89) --
	( 39.71, 58.76) --
	( 40.04, 55.85) --
	( 40.37, 53.80) --
	( 40.70, 53.60) --
	( 41.03, 53.31) --
	( 41.36, 52.88) --
	( 41.69, 51.43) --
	( 42.02, 56.66) --
	( 42.35, 55.05) --
	( 42.68, 50.24) --
	( 43.01, 42.59) --
	( 43.34, 44.28) --
	( 43.67, 45.20) --
	( 44.00, 44.98) --
	( 44.33, 44.12) --
	( 44.66, 33.77) --
	( 44.99, 36.60) --
	( 45.32, 32.95) --
	( 45.65, 38.52) --
	( 45.98, 30.53) --
	( 46.31, 32.61) --
	( 46.64, 25.92) --
	( 46.97, 25.70) --
	( 47.30, 29.21) --
	( 47.63, 29.76) --
	( 47.96, 25.95) --
	( 48.29, 27.42) --
	( 48.62, 30.47) --
	( 48.95, 29.26) --
	( 49.28, 31.01) --
	( 49.61, 27.05) --
	( 49.94, 22.79) --
	( 50.27, 27.55) --
	( 50.60, 28.95) --
	( 50.93, 31.24) --
	( 51.26, 30.10) --
	( 51.59, 29.70) --
	( 51.92, 28.52) --
	( 52.25, 29.41) --
	( 52.58, 24.45) --
	( 52.91, 23.73) --
	( 53.24, 26.35) --
	( 53.57, 26.98) --
	( 53.90, 36.55) --
	( 54.23, 35.57) --
	( 54.56, 35.26) --
	( 54.89, 36.87) --
	( 55.22, 43.03) --
	( 55.55, 45.37) --
	( 55.88, 46.68) --
	( 56.21, 56.28) --
	( 56.54, 62.26) --
	( 56.87, 60.34) --
	( 57.20, 58.00) --
	( 57.53, 49.53) --
	( 57.86, 45.25) --
	( 58.19, 54.29) --
	( 58.52, 55.62) --
	( 58.85, 53.32) --
	( 59.18, 51.34) --
	( 59.51, 52.77) --
	( 59.84, 55.18) --
	( 60.17, 53.53) --
	( 60.50, 49.74) --
	( 60.83, 47.23) --
	( 61.16, 48.94) --
	( 61.49, 49.74) --
	( 61.82, 51.65) --
	( 62.16, 54.47) --
	( 62.49, 54.73) --
	( 62.82, 53.35) --
	( 63.15, 54.26) --
	( 63.48, 50.97) --
	( 63.81, 54.22) --
	( 64.14, 54.47) --
	( 64.47, 53.97) --
	( 64.80, 54.52) --
	( 65.13, 62.32) --
	( 65.46, 59.00) --
	( 65.79, 60.46) --
	( 66.12, 62.95) --
	( 66.45, 60.88) --
	( 66.78, 64.38) --
	( 67.11, 62.44) --
	( 67.44, 58.20) --
	( 67.77, 63.75) --
	( 68.10, 70.99) --
	( 68.43, 69.77) --
	( 68.76, 73.24) --
	( 69.09, 73.34) --
	( 69.42, 73.81) --
	( 69.75, 72.54) --
	( 70.08, 79.94) --
	( 70.41, 77.46) --
	( 70.74, 73.50) --
	( 71.07, 73.62) --
	( 71.40, 82.34) --
	( 71.73, 75.73) --
	( 72.06, 75.19) --
	( 72.39, 73.74) --
	( 72.72, 70.48) --
	( 73.05, 70.92) --
	( 73.38, 73.24) --
	( 73.71, 76.17) --
	( 74.04, 76.33) --
	( 74.37, 76.93) --
	( 74.70, 75.89) --
	( 75.03, 75.84) --
	( 75.36, 75.12) --
	( 75.69, 77.84) --
	( 76.02, 77.27) --
	( 76.35, 72.20) --
	( 76.68, 76.31) --
	( 77.01, 82.56) --
	( 77.34, 82.84) --
	( 77.67, 80.41) --
	( 78.00, 78.82) --
	( 78.33, 76.25) --
	( 78.66, 74.42) --
	( 78.99, 70.78) --
	( 79.32, 68.68) --
	( 79.65, 63.24) --
	( 79.98, 66.37) --
	( 80.31, 68.98) --
	( 80.64, 68.78) --
	( 80.97, 69.18) --
	( 81.30, 70.67) --
	( 81.63, 71.11) --
	( 81.96, 68.87) --
	( 82.29, 67.70) --
	( 82.62, 66.06) --
	( 82.95, 65.90) --
	( 83.28, 63.80) --
	( 83.61, 67.42) --
	( 83.94, 64.59) --
	( 84.27, 65.46) --
	( 84.60, 61.64) --
	( 84.93, 64.97) --
	( 85.26, 67.75) --
	( 85.59, 65.75) --
	( 85.93, 65.77) --
	( 86.26, 70.67) --
	( 86.59, 66.31) --
	( 86.92, 63.75) --
	( 87.25, 57.79) --
	( 87.58, 62.60) --
	( 87.91, 59.99) --
	( 88.24, 61.93) --
	( 88.57, 63.66) --
	( 88.90, 66.59) --
	( 89.23, 76.39) --
	( 89.56, 77.24) --
	( 89.89, 80.11) --
	( 90.22, 81.25) --
	( 90.55, 83.79) --
	( 90.88, 82.33) --
	( 91.21, 75.87) --
	( 91.54, 74.49) --
	( 91.87, 74.39) --
	( 92.20, 74.22) --
	( 92.53, 74.65) --
	( 92.86, 73.22) --
	( 93.19, 76.54) --
	( 93.52, 70.83) --
	( 93.85, 72.07) --
	( 94.18, 72.83) --
	( 94.51, 73.59) --
	( 94.84, 74.23) --
	( 95.17, 73.74) --
	( 95.50, 74.15) --
	( 95.83, 73.96) --
	( 96.16, 84.20) --
	( 96.49, 88.10) --
	( 96.82, 84.84) --
	( 97.15, 83.96) --
	( 97.48, 78.22) --
	( 97.81, 80.58) --
	( 98.14, 75.42) --
	( 98.47, 76.84) --
	( 98.80, 75.45) --
	( 99.13, 76.93) --
	( 99.46, 75.20) --
	( 99.79, 77.43) --
	(100.12, 77.53) --
	(100.45, 80.92) --
	(100.78, 85.38) --
	(101.11, 89.04) --
	(101.44, 89.26) --
	(101.77, 90.06) --
	(102.10, 88.35) --
	(102.43, 86.86) --
	(102.76, 85.22) --
	(103.09, 91.74) --
	(103.42, 90.14) --
	(103.75, 94.31) --
	(104.08, 95.16) --
	(104.41, 95.52) --
	(104.74, 95.89) --
	(105.07, 92.98) --
	(105.40, 89.00) --
	(105.73, 91.64) --
	(106.06, 89.94) --
	(106.39, 92.40) --
	(106.72, 96.36) --
	(107.05, 99.52) --
	(107.38, 94.21) --
	(107.71, 94.02) --
	(108.04, 91.58) --
	(108.37, 97.41) --
	(108.70,103.44) --
	(109.03,109.80) --
	(109.36,112.02) --
	(109.70,113.49) --
	(110.03,105.77) --
	(110.36,107.84) --
	(110.69,101.08) --
	(111.02,102.59) --
	(111.35,100.82) --
	(111.68, 99.17) --
	(112.01, 95.41) --
	(112.34, 91.19) --
	(112.67, 90.90) --
	(113.00, 93.42) --
	(113.33, 95.81) --
	(113.66, 93.24) --
	(113.99, 95.10) --
	(114.32, 92.57) --
	(114.65, 92.50) --
	(114.98, 88.63) --
	(115.31, 92.01) --
	(115.64, 87.83) --
	(115.97, 90.08) --
	(116.30, 89.49) --
	(116.63, 92.54) --
	(116.96, 98.94) --
	(117.29, 96.01) --
	(117.62, 99.53) --
	(117.95, 95.04) --
	(118.28, 95.06) --
	(118.61, 90.63) --
	(118.94, 85.95) --
	(119.27, 89.91) --
	(119.60, 90.64) --
	(119.93, 94.23) --
	(120.26, 92.67) --
	(120.59, 89.48) --
	(120.92, 90.71) --
	(121.25, 86.27) --
	(121.58, 86.76) --
	(121.91, 92.65) --
	(122.24, 95.34) --
	(122.57, 90.56) --
	(122.90, 93.29) --
	(123.23, 89.91) --
	(123.56, 90.56) --
	(123.89, 97.92) --
	(124.22,106.75) --
	(124.55,106.87) --
	(124.88,106.09) --
	(125.21,109.81) --
	(125.54,114.99) --
	(125.87,106.79) --
	(126.20,106.68) --
	(126.53,106.83) --
	(126.86,107.82) --
	(127.19,116.46) --
	(127.52,110.13) --
	(127.85,116.72) --
	(128.18,114.72) --
	(128.51,115.16) --
	(128.84,113.63) --
	(129.17,114.73) --
	(129.50,117.59) --
	(129.83,122.53) --
	(130.16,121.58) --
	(130.49,127.85) --
	(130.82,128.23) --
	(131.15,124.90);

\node[text=drawColor,anchor=base,inner sep=0pt, outer sep=0pt, scale=  1.00] at ( 73.38,  5.77) {$\scriptstyle Y^{(1)}$};

\path[draw=drawColor,line width= 0.4pt,line join=round,line cap=round] (131.15, 55.96) --
	(131.48, 53.00) --
	(131.81, 53.25) --
	(132.14, 53.91) --
	(132.47, 63.70) --
	(132.80, 64.42) --
	(133.13, 66.55) --
	(133.47, 59.10) --
	(133.80, 56.82) --
	(134.13, 62.43) --
	(134.46, 61.75) --
	(134.79, 60.43) --
	(135.12, 58.53) --
	(135.45, 56.42) --
	(135.78, 58.34) --
	(136.11, 50.61) --
	(136.44, 51.41) --
	(136.77, 52.20) --
	(137.10, 55.44) --
	(137.43, 51.68) --
	(137.76, 59.47) --
	(138.09, 56.80) --
	(138.42, 59.34) --
	(138.75, 60.48) --
	(139.08, 59.39) --
	(139.41, 57.12) --
	(139.74, 59.12) --
	(140.07, 59.49) --
	(140.40, 49.03) --
	(140.73, 57.21) --
	(141.06, 63.64) --
	(141.39, 62.04) --
	(141.72, 55.14) --
	(142.05, 55.52) --
	(142.38, 55.96) --
	(142.71, 59.38) --
	(143.04, 58.98) --
	(143.37, 56.76) --
	(143.70, 57.38) --
	(144.03, 60.26) --
	(144.36, 59.75) --
	(144.69, 58.11) --
	(145.02, 52.80) --
	(145.35, 53.05) --
	(145.68, 57.66) --
	(146.01, 55.45) --
	(146.34, 58.26) --
	(146.67, 61.49) --
	(147.00, 65.79) --
	(147.33, 59.13) --
	(147.66, 60.25) --
	(147.99, 61.24) --
	(148.32, 59.69) --
	(148.65, 57.05) --
	(148.98, 62.66) --
	(149.31, 65.70) --
	(149.64, 70.83) --
	(149.97, 66.97) --
	(150.30, 54.20) --
	(150.63, 54.29) --
	(150.96, 53.07) --
	(151.29, 56.70) --
	(151.62, 51.95) --
	(151.95, 45.70) --
	(152.28, 38.04) --
	(152.61, 36.17) --
	(152.94, 35.00) --
	(153.27, 35.62) --
	(153.60, 31.65) --
	(153.93, 33.70) --
	(154.26, 35.07) --
	(154.59, 32.28) --
	(154.92, 32.17) --
	(155.25, 31.78) --
	(155.58, 33.00) --
	(155.91, 30.15) --
	(156.24, 33.43) --
	(156.57, 38.26) --
	(156.90, 37.29) --
	(157.24, 37.15) --
	(157.57, 42.66) --
	(157.90, 39.51) --
	(158.23, 36.70) --
	(158.56, 35.91) --
	(158.89, 42.46) --
	(159.22, 36.62) --
	(159.55, 39.15) --
	(159.88, 41.44) --
	(160.21, 45.42) --
	(160.54, 40.03) --
	(160.87, 41.68) --
	(161.20, 45.36) --
	(161.53, 50.75) --
	(161.86, 51.93) --
	(162.19, 49.26) --
	(162.52, 51.29) --
	(162.85, 50.50) --
	(163.18, 52.59) --
	(163.51, 48.91) --
	(163.84, 45.70) --
	(164.17, 48.65) --
	(164.50, 49.08) --
	(164.83, 53.49) --
	(165.16, 56.11) --
	(165.49, 62.44) --
	(165.82, 60.60) --
	(166.15, 55.14) --
	(166.48, 54.32) --
	(166.81, 60.34) --
	(167.14, 59.20) --
	(167.47, 59.52) --
	(167.80, 61.39) --
	(168.13, 62.20) --
	(168.46, 61.42) --
	(168.79, 56.54) --
	(169.12, 62.45) --
	(169.45, 59.68) --
	(169.78, 58.06) --
	(170.11, 55.22) --
	(170.44, 51.90) --
	(170.77, 53.94);

\node[text=drawColor,anchor=base,inner sep=0pt, outer sep=0pt, scale=  1.00] at (150.96,  5.77) {$\scriptstyle Y^{(2)}$};

\path[draw=drawColor,line width= 0.4pt,line join=round,line cap=round] (170.77, 55.96) --
	(171.10, 55.79) --
	(171.43, 55.86) --
	(171.76, 63.15) --
	(172.09, 59.89) --
	(172.42, 62.83) --
	(172.75, 65.93) --
	(173.08, 63.35) --
	(173.41, 70.87) --
	(173.74, 75.34) --
	(174.07, 66.24) --
	(174.40, 64.14) --
	(174.73, 65.47) --
	(175.06, 62.83) --
	(175.39, 67.44) --
	(175.72, 67.37) --
	(176.05, 67.88) --
	(176.38, 64.59) --
	(176.71, 69.68) --
	(177.04, 74.41) --
	(177.37, 73.68) --
	(177.70, 73.97) --
	(178.03, 74.72) --
	(178.36, 75.11) --
	(178.69, 72.00) --
	(179.02, 78.76) --
	(179.35, 79.93) --
	(179.68, 78.81) --
	(180.01, 75.24) --
	(180.34, 70.84) --
	(180.67, 73.54) --
	(181.01, 78.64) --
	(181.34, 81.05) --
	(181.67, 79.85) --
	(182.00, 72.64) --
	(182.33, 66.16) --
	(182.66, 68.39) --
	(182.99, 63.40) --
	(183.32, 61.31) --
	(183.65, 63.56) --
	(183.98, 57.68) --
	(184.31, 58.74) --
	(184.64, 49.91) --
	(184.97, 50.38) --
	(185.30, 45.72) --
	(185.63, 40.69) --
	(185.96, 39.36) --
	(186.29, 40.67) --
	(186.62, 41.19) --
	(186.95, 47.84) --
	(187.28, 49.35) --
	(187.61, 52.60) --
	(187.94, 54.32) --
	(188.27, 53.05) --
	(188.60, 51.22) --
	(188.93, 51.22) --
	(189.26, 50.49) --
	(189.59, 51.15) --
	(189.92, 52.13) --
	(190.25, 52.60) --
	(190.58, 52.78) --
	(190.91, 52.13) --
	(191.24, 45.39) --
	(191.57, 38.64) --
	(191.90, 40.47) --
	(192.23, 45.59) --
	(192.56, 49.76) --
	(192.89, 47.32) --
	(193.22, 48.90) --
	(193.55, 46.27) --
	(193.88, 49.28) --
	(194.21, 48.46) --
	(194.54, 48.25) --
	(194.87, 53.27) --
	(195.20, 47.49) --
	(195.53, 48.80) --
	(195.86, 47.74) --
	(196.19, 46.62) --
	(196.52, 42.10) --
	(196.85, 34.23) --
	(197.18, 36.31) --
	(197.51, 33.39) --
	(197.84, 32.25) --
	(198.17, 32.84) --
	(198.50, 29.07) --
	(198.83, 33.86) --
	(199.16, 30.99) --
	(199.49, 33.21) --
	(199.82, 26.04) --
	(200.15, 21.48) --
	(200.48, 27.66) --
	(200.81, 28.85) --
	(201.14, 28.49) --
	(201.47, 30.24) --
	(201.80, 26.19) --
	(202.13, 32.25) --
	(202.46, 33.17) --
	(202.79, 31.87) --
	(203.12, 29.65) --
	(203.45, 32.99) --
	(203.78, 34.32) --
	(204.11, 29.31) --
	(204.44, 27.40) --
	(204.78, 23.48) --
	(205.11, 31.10) --
	(205.44, 28.49) --
	(205.77, 27.43) --
	(206.10, 32.06) --
	(206.43, 30.70) --
	(206.76, 24.32) --
	(207.09, 20.83) --
	(207.42, 20.76) --
	(207.75, 25.42) --
	(208.08, 23.10) --
	(208.41, 20.68) --
	(208.74, 18.69) --
	(209.07, 17.20) --
	(209.40, 19.14) --
	(209.73, 28.38) --
	(210.06, 27.86) --
	(210.39, 32.34) --
	(210.72, 35.44) --
	(211.05, 33.98) --
	(211.38, 32.89) --
	(211.71, 32.07) --
	(212.04, 28.57) --
	(212.37, 27.59) --
	(212.70, 27.98) --
	(213.03, 20.93) --
	(213.36, 22.02) --
	(213.69, 24.18);

\node[text=drawColor,anchor=base,inner sep=0pt, outer sep=0pt, scale=  1.00] at (192.23,  5.77) {$\scriptstyle Y^{(3)}$};

\path[draw=drawColor,line width= 0.4pt,line join=round,line cap=round] (213.69, 55.96) --
	(214.02, 53.00) --
	(214.35, 58.32) --
	(214.68, 56.65) --
	(215.01, 54.14) --
	(215.34, 59.62) --
	(215.67, 57.99) --
	(216.00, 66.17) --
	(216.33, 63.80) --
	(216.66, 69.37) --
	(216.99, 72.16) --
	(217.32, 70.94) --
	(217.65, 69.90) --
	(217.98, 68.27) --
	(218.31, 70.42) --
	(218.64, 74.79) --
	(218.97, 77.43) --
	(219.30, 74.71) --
	(219.63, 80.97) --
	(219.96, 75.62) --
	(220.29, 69.86) --
	(220.62, 73.41) --
	(220.95, 75.31) --
	(221.28, 70.46) --
	(221.61, 70.47) --
	(221.94, 72.04) --
	(222.27, 81.01) --
	(222.60, 83.28) --
	(222.93, 87.06) --
	(223.26, 82.88) --
	(223.59, 88.48) --
	(223.92, 91.10) --
	(224.25, 97.19) --
	(224.58,103.79) --
	(224.91,103.68) --
	(225.24,107.70) --
	(225.57,109.25) --
	(225.90,104.85) --
	(226.23,108.30) --
	(226.56,112.46) --
	(226.89,109.55) --
	(227.22,112.43) --
	(227.55,111.33) --
	(227.88,105.61) --
	(228.21,103.88) --
	(228.55,102.42) --
	(228.88,106.74) --
	(229.21,103.27) --
	(229.54,106.88) --
	(229.87,108.55) --
	(230.20,105.19) --
	(230.53,104.02) --
	(230.86,106.85) --
	(231.19,106.57) --
	(231.52,105.48) --
	(231.85,107.46) --
	(232.18,110.11) --
	(232.51,108.60) --
	(232.84,108.63) --
	(233.17,103.17) --
	(233.50,103.20) --
	(233.83, 94.64) --
	(234.16, 92.13) --
	(234.49, 92.71) --
	(234.82, 97.61) --
	(235.15,108.95) --
	(235.48,103.81) --
	(235.81, 98.83) --
	(236.14, 99.17) --
	(236.47,102.39) --
	(236.80,104.87) --
	(237.13,101.75) --
	(237.46,102.49) --
	(237.79,109.67) --
	(238.12,112.72) --
	(238.45,113.51) --
	(238.78,113.48) --
	(239.11,117.52) --
	(239.44,117.61) --
	(239.77,124.61) --
	(240.10,122.68) --
	(240.43,127.69) --
	(240.76,125.10) --
	(241.09,122.24) --
	(241.42,125.29) --
	(241.75,122.61) --
	(242.08,122.62) --
	(242.41,113.40) --
	(242.74,110.18) --
	(243.07,104.63) --
	(243.40,100.60) --
	(243.73,100.44) --
	(244.06, 99.09) --
	(244.39, 99.50) --
	(244.72,108.25) --
	(245.05,114.99) --
	(245.38,117.33) --
	(245.71,120.91) --
	(246.04,118.51) --
	(246.37,120.79) --
	(246.70,115.51) --
	(247.03,114.62) --
	(247.36,108.98) --
	(247.69,111.48) --
	(248.02,116.94) --
	(248.35,113.77) --
	(248.68,117.47) --
	(249.01,120.41) --
	(249.34,121.71) --
	(249.67,121.16) --
	(250.00,125.63) --
	(250.33,126.42) --
	(250.66,129.75) --
	(250.99,136.48) --
	(251.32,131.72) --
	(251.65,129.28) --
	(251.98,128.77) --
	(252.32,129.32) --
	(252.65,128.09) --
	(252.98,126.13) --
	(253.31,129.60) --
	(253.64,128.07) --
	(253.97,121.75) --
	(254.30,121.75) --
	(254.63,119.36) --
	(254.96,125.26) --
	(255.29,123.80) --
	(255.62,134.00) --
	(255.95,136.43) --
	(256.28,130.60) --
	(256.61,134.30) --
	(256.94,137.77) --
	(257.27,137.14) --
	(257.60,133.66) --
	(257.93,129.08) --
	(258.26,132.23) --
	(258.59,130.82) --
	(258.92,123.82) --
	(259.25,120.12) --
	(259.58,116.50) --
	(259.91,120.91) --
	(260.24,113.10) --
	(260.57,111.01) --
	(260.90,109.40) --
	(261.23,109.65) --
	(261.56,109.18) --
	(261.89,105.38) --
	(262.22,110.05) --
	(262.55,112.64) --
	(262.88,114.72) --
	(263.21,107.81) --
	(263.54,104.98) --
	(263.87,106.29) --
	(264.20,106.39) --
	(264.53,109.48) --
	(264.86,110.61) --
	(265.19,107.25) --
	(265.52,109.92) --
	(265.85,110.26) --
	(266.18,108.07) --
	(266.51,107.06) --
	(266.84,103.18) --
	(267.17, 98.50) --
	(267.50, 97.78) --
	(267.83, 98.35) --
	(268.16, 97.54) --
	(268.49, 92.68) --
	(268.82, 96.98) --
	(269.15, 91.54) --
	(269.48, 85.62) --
	(269.81, 88.22) --
	(270.14, 90.90) --
	(270.47, 97.41) --
	(270.80, 91.88) --
	(271.13, 89.54) --
	(271.46, 93.80) --
	(271.79, 93.09) --
	(272.12, 93.40) --
	(272.45, 98.54) --
	(272.78,101.24) --
	(273.11,102.93) --
	(273.44,110.39) --
	(273.77,107.51) --
	(274.10,109.34) --
	(274.43,104.62) --
	(274.76,102.74) --
	(275.09, 96.96) --
	(275.42, 99.40) --
	(275.75, 99.63) --
	(276.09, 99.23) --
	(276.42,105.76) --
	(276.75,103.76) --
	(277.08,106.42) --
	(277.41,109.47) --
	(277.74,110.14) --
	(278.07,108.63) --
	(278.40,108.18) --
	(278.73,108.60) --
	(279.06,116.72) --
	(279.39,111.41) --
	(279.72, 99.43) --
	(280.05,102.07) --
	(280.38,100.48) --
	(280.71,101.38) --
	(281.04, 95.49) --
	(281.37, 94.15) --
	(281.70, 94.20) --
	(282.03, 95.61) --
	(282.36, 93.47) --
	(282.69, 90.33) --
	(283.02, 92.81) --
	(283.35, 92.35) --
	(283.68, 94.12) --
	(284.01, 94.36) --
	(284.34, 90.07) --
	(284.67, 86.31) --
	(285.00, 89.44) --
	(285.33, 91.52) --
	(285.66, 87.13) --
	(285.99, 88.75) --
	(286.32, 91.98) --
	(286.65, 95.93) --
	(286.98, 95.23) --
	(287.31, 98.82) --
	(287.64, 98.80) --
	(287.97,105.58) --
	(288.30,108.78) --
	(288.63,107.98) --
	(288.96,109.35) --
	(289.29,111.17) --
	(289.62,108.21) --
	(289.95,106.55) --
	(290.28,112.32) --
	(290.61,111.02) --
	(290.94,109.15) --
	(291.27,110.93) --
	(291.60,106.41) --
	(291.93,109.42) --
	(292.26,110.46) --
	(292.59,108.21) --
	(292.92,115.89) --
	(293.25,109.67) --
	(293.58,115.84) --
	(293.91,118.50) --
	(294.24,123.54) --
	(294.57,128.32) --
	(294.90,125.74) --
	(295.23,122.69) --
	(295.56,125.30) --
	(295.89,120.96) --
	(296.22,116.41) --
	(296.55,105.60) --
	(296.88,106.68) --
	(297.21,103.31) --
	(297.54,108.89) --
	(297.87,109.68) --
	(298.20,113.07) --
	(298.53,116.66) --
	(298.86,121.19) --
	(299.19,115.50) --
	(299.52,112.25);

\node[text=drawColor,anchor=base,inner sep=0pt, outer sep=0pt, scale=  1.00] at (256.61,  5.77) {$\scriptstyle Y^{(4)}$};

\path[draw=drawColor,line width= 0.4pt,dash pattern=on 1pt off 3pt ,line join=round,line cap=round] ( 15.61, 55.96) --
	(345.74, 55.96);

\path[draw=drawColor,line width= 0.4pt,line join=round,line cap=round] (299.52, 55.96) --
	(299.86, 58.09) --
	(300.19, 56.70) --
	(300.52, 53.59) --
	(300.85, 58.77) --
	(301.18, 56.29) --
	(301.51, 52.03) --
	(301.84, 53.60) --
	(302.17, 49.68) --
	(302.50, 45.86) --
	(302.83, 45.26) --
	(303.16, 48.91) --
	(303.49, 52.29) --
	(303.82, 57.26) --
	(304.15, 58.47) --
	(304.48, 55.09) --
	(304.81, 56.36) --
	(305.14, 59.29) --
	(305.47, 61.21) --
	(305.80, 61.66) --
	(306.13, 64.32) --
	(306.46, 66.71) --
	(306.79, 63.58) --
	(307.12, 65.29) --
	(307.45, 70.12) --
	(307.78, 67.55) --
	(308.11, 77.59) --
	(308.44, 78.73) --
	(308.77, 76.03) --
	(309.10, 73.53) --
	(309.43, 74.56) --
	(309.76, 76.97) --
	(310.09, 74.08) --
	(310.42, 77.75) --
	(310.75, 78.84) --
	(311.08, 78.44) --
	(311.41, 85.95) --
	(311.74, 86.09) --
	(312.07, 82.68) --
	(312.40, 78.03) --
	(312.73, 81.44) --
	(313.06, 87.60) --
	(313.39, 91.54) --
	(313.72, 93.69) --
	(314.05, 93.20) --
	(314.38, 89.78) --
	(314.71, 86.70) --
	(315.04, 91.09) --
	(315.37, 92.12) --
	(315.70, 96.87) --
	(316.03, 88.28) --
	(316.36, 92.25) --
	(316.69, 92.42) --
	(317.02, 86.78) --
	(317.35, 89.09) --
	(317.68, 86.35) --
	(318.01, 90.37) --
	(318.34, 96.35) --
	(318.67,101.01) --
	(319.00,100.23) --
	(319.33,105.49) --
	(319.66,106.73) --
	(319.99,106.25) --
	(320.32,111.59) --
	(320.65,115.31) --
	(320.98,115.85) --
	(321.31,118.47) --
	(321.64,119.51) --
	(321.97,116.89) --
	(322.30,118.72) --
	(322.63,110.46) --
	(322.96,104.48) --
	(323.29, 99.45) --
	(323.63, 98.53) --
	(323.96, 99.24) --
	(324.29,106.34) --
	(324.62,105.15) --
	(324.95,104.26) --
	(325.28,110.25) --
	(325.61,113.39) --
	(325.94,112.37) --
	(326.27,112.82) --
	(326.60,112.17) --
	(326.93,112.67) --
	(327.26,112.01) --
	(327.59,109.27) --
	(327.92,104.78) --
	(328.25,104.40) --
	(328.58,110.47) --
	(328.91,114.99) --
	(329.24,120.07) --
	(329.57,119.44) --
	(329.90,121.34) --
	(330.23,122.89) --
	(330.56,128.39) --
	(330.89,125.31) --
	(331.22,128.19) --
	(331.55,132.89) --
	(331.88,133.98) --
	(332.21,133.97) --
	(332.54,134.73) --
	(332.87,131.49) --
	(333.20,138.44) --
	(333.53,138.27) --
	(333.86,138.33) --
	(334.19,145.45) --
	(334.52,145.03) --
	(334.85,149.13) --
	(335.18,148.40) --
	(335.51,147.53) --
	(335.84,147.20) --
	(336.17,144.78) --
	(336.50,150.17) --
	(336.83,144.08) --
	(337.16,146.26) --
	(337.49,146.06) --
	(337.82,140.65) --
	(338.15,142.99) --
	(338.48,140.88) --
	(338.81,139.00) --
	(339.14,140.47) --
	(339.47,140.61) --
	(339.80,143.65) --
	(340.13,148.65) --
	(340.46,151.75) --
	(340.79,152.39) --
	(341.12,148.67) --
	(341.45,149.87) --
	(341.78,139.82) --
	(342.11,135.70) --
	(342.44,131.61) --
	(342.77,130.69) --
	(343.10,129.94) --
	(343.43,128.61) --
	(343.76,130.56) --
	(344.09,128.03) --
	(344.42,125.73) --
	(344.75,122.06) --
	(345.08,122.53) --
	(345.41,122.89) --
	(345.74,124.68);
\end{scope}
\end{tikzpicture}

\end{center}
\caption{\small The word sequence $Y$. Upper part: Brownian path $X$ and renewal times $T$. 
Lower part: increments of the path between the renewal times (which are elements of $F$).}
\label{fig-wordsequence}
\end{figure}

Define the \emph{word sequence} $Y = (Y^{(i)})_{i\in\N}$ by putting (see 
Fig.~\ref{fig-wordsequence})
\begin{equation} 
\qquad Y^{(i)} = \Big(T_i-T_{i-1}, \big(X_{(s+T_{i-1}) \wedge T_i}
-X_{T_{i-1}}\big)_{s\geq 0}\Big), 
\end{equation}
which takes values in the \emph{word space}
\begin{equation} 
\label{eq:defF}
F = \bigcup_{t>0} \Big(\{t\} \times \big\{ f \in C([0,\infty))\colon\, 
f(0)=0, f(s)=f(t) \; \text{for} \; s > t \big\}\Big)
\end{equation}
equipped with a Skorohod-type metric (see Appendix~\ref{metrics}). Let 
\begin{equation}
Y^{N\text{-}\mathrm{per}} = \big(\,\underbrace{Y^{(1)},Y^{(2)},\dots,Y^{(N)}},\,
\underbrace{Y^{(1)},Y^{(2)},\dots,Y^{(N)}},\,\dots\big)
\end{equation} 
denote the $N$-periodisation of $Y$, and let   
\begin{equation} 
\label{RNdef}
R_N = \frac1N \sum_{i=0}^{N-1} 
\delta_{\widetilde{\theta}^i Y^{N\text{-}\mathrm{per}}}
\end{equation}
be the \emph{empirical process of words}, where $\widetilde{\theta}$ is the left-shift 
acting on $F^\N$. Note that $R_N$ takes values in $\mathcal{P}^{\mathrm{inv}}(F^\N)$, 
the set of shift-invariant probability measures on $F^\N$. Endow $F^\N$ with the product 
topology and $\mathcal{P}^{\mathrm{inv}}(F^\N)$ with the corresponding weak topology. 
When averaged over $X$ and $T$, the law of $Y$ is ($\mathscr{L}$ denotes law)
\begin{equation}
\label{qrhoWdef} 
Q_{\rho,\WM}=(q_{\rho,\WM})^{\otimes\N} \quad \text{with} \quad
q_{\rho,\WM} = \int_{(0,\infty)} \rho(dt)\,
\mathscr{L}\big((t, (X_{s \wedge t})_{s \geq 0})\big).
\end{equation}
By the ergodic theorem, $\wlim_{N\to\infty} R_N = Q_{\rho,\WM}$ a.s., where 
$\wlim$ denotes the weak limit. 


\subsection{Large deviation principles}
\label{LDPs}

For definitions and properties of specific relative entropy, we refer the reader to 
Appendix~\ref{entropy}.

The following theorem is standard (see e.g.\ Dembo and Zeitouni~\cite[Section 6.5.3]{DeZe98}).

\begin{theorem} 
\label{thm0:contaLDP} 
{\rm {\bf [Annealed LDP]}}\\
The family $\mathscr{L}(R_N)$, $N\in\N$, satisfies the LDP on $\mathcal{P}^{\mathrm{inv}}
(F^\N)$ with rate $N$ and with rate function
\begin{equation}
\label{eq:Iann}
I^{\mathrm{ann}}(Q)= H(Q \mid Q_{\rho,\WM}),
\end{equation}
the specific relative entropy of $Q$ w.r.t.\ $Q_{\rho,\WM}$. 
This rate function is lower 
semi-continuous, has compact level sets, is affine, and has a unique zero at 
$Q=Q_{\rho,\WM}$.
\end{theorem}

To state the quenched LDP, we need to look at the reverse of cutting out words, namely, 
glueing words together. Let ${y}=(y^{(i)})_{i\in\N}=((t_i,f_i))_{i\in\N} 
\in F^\N$. Then the \emph{concatenation} of ${y}$, 
written $\kappa({y}) \in C([0,\infty))$, is defined by
\begin{equation}
\begin{aligned}
&\kappa({y})(s) = f_1(t_1)+\dots+f_{i-1}(t_{i-1})
+f_i\big(s-(t_1+\cdots+t_{i-1})\big),\\ 
&t_1+\cdots+t_{i-1} \leq s < t_1+\cdots+t_{i}, \;\; i \in \N.
\end{aligned}
\end{equation}
Write $\tau_i({y})=t_i$ to denote the length of the $i$-th word. For $Q 
\in \mathcal{P}^{\mathrm{inv}}(F^\N)$ with finite mean word length $m_Q = 
\E_Q[\tau_1] =\E_Q[\tau_1(Y)]$, put
\begin{equation} 
\label{eq:PsiQcont}
\Psi_Q(A) = 
\frac{1}{m_Q} \E_Q\left[ \int_0^{\tau_1} \1_A(\theta^s \kappa(Y)) \, ds\right], 
\quad A \subset C([0,\infty))\;\;\mbox{measurable}, 
\end{equation}
where $\theta^s$ is the shift acting on $f \in C([0,\infty))$ as $\theta^s f(t) 
= f(s+t)-f(s)$, $t \geq 0$. Note that $\Psi_Q$ is a probability measure on 
$C([0,\infty))$ with stationary increments, i.e., $\Psi_Q = \Psi_Q \circ (\theta^s)^{-1}$ 
for all $s \ge 0$. We can think of $\Psi_Q$ as the ``stationarised'' version of 
$\kappa(Q)$. In fact, if $m_Q<\infty$, then 
\begin{equation}
\label{eq:PsiQ}
\Psi_Q = \wlim_{T\to\infty} \frac{1}{T} 
\int_0^T \kappa(Q) \circ (\theta^s)^{-1}\,ds,
\end{equation} 
and $\kappa(Q)$ is asymptotically mean stationary (AMS) with stationary mean $\Psi_Q$. 
In fact, the convergence in \eqref{eq:PsiQ} also holds in total variation norm (see 
Lemma~\ref{lemma:PsiQ:TVlim} in Appendix~\ref{entropy}). Note that $\Psi_{Q_{\rho,\WM}}=\WM$.

To state the quenched LDP, we also need to define word \emph{truncation}. For 
$(t,f) \in F$ and $\tr > 0$, let 
\begin{equation}
[(t,f)]_\tr = \big(t \wedge \tr, (f(s \wedge \tr)_{s\ge 0}\big)
\end{equation} 
be the word $(t,f)$ truncated at length $\tr$. Analogously, for ${y}=(y^{(i)})_{i\in\N} 
\in F^\N$ set $[{y}]_\tr=([y^{(i)}]_{\tr})_{i\in\N} \in F^\N$, and denote by 
$[Q]_\tr \in \mathcal{P}^{\mathrm{inv}}(F_{0,\tr}^\N) \subset \mathcal{P}^{\mathrm{inv}}
(F^\N)$ with $F_{0,\tr}=[F]_\tr$ the image measure of $Q \in \mathcal{P}^{\mathrm{inv}}
(F^\N)$ under the map ${y} \mapsto [{y}]_\tr$.

\begin{theorem} 
\label{thm0:contqLDP}
{\rm {\bf [Quenched LDP]}}\\
Suppose that $\rho$ satisfies {\rm (\ref{ass:rhodensdecay}--\ref{ass:rhobar.reg0})}. Then, 
for $\WM$ a.e.\ $X$, the family $\mathscr{L}(R_N \mid X)$, $N\in\N$, satisfies the LDP 
on $\mathcal{P}^{\mathrm{inv}}(F^\N)$ with rate $N$ and with deterministic rate function 
$I^{\mathrm{que}}(Q)$ given by 
\begin{equation} 
\label{eq:Iquelimitform}
I^{\mathrm{que}}(Q) = \lim_{\tr\to\infty} I^{\mathrm{que}}_\tr([Q]_\tr), 
\end{equation}
where
\begin{equation}
\label{def:Ique.tr}
I^{\mathrm{que}}_\tr([Q]_\tr) = H\big([Q]_\tr \mid [Q_{\rho,\WM}]_\tr\big) 
+ (\alpha-1) m_{[Q]_\tr} H\big(\Psi_{[Q]_\tr} \mid \WM\big).
\end{equation}
This rate function is lower semi-continuous, has compact level sets, is affine, and 
has a unique zero at $Q=Q_{\rho,\WM}$.
\end{theorem}

Theorem~\ref{thm0:contqLDP} is proved in Sections~\ref{proof}--\ref{removeass}.
Let $\mathcal{P}^{\mathrm{inv,fin}}(F^\N)=\{Q \in \mathcal{P}^{\mathrm{inv}}(F^\N)
\colon\,m_Q<\infty\}$. We will show that the limit in \eqref{eq:Iquelimitform} exists 
for all $Q \in \mathcal{P}^{\mathrm{inv}}(F^\N)$, and that
\begin{equation} 
\label{eq:Ique}
I^{\mathrm{que}}(Q) = H(Q \mid Q_{\rho,\WM}) + (\alpha-1) m_Q H(\Psi_Q \mid \WM), 
\qquad Q \in \mathcal{P}^{\mathrm{inv,fin}}(F^\N). 
\end{equation} 
We will also see that $I^{\mathrm{que}}(Q)$ is the lower semi-continuous extension to 
$\mathcal{P}^{\mathrm{inv}}(F^\N)$ of its restriction to $\mathcal{P}^{\mathrm{inv,fin}}
(F^\N)$.


\subsection{Discussion}
\label{disc}

{\bf 0.} A \emph{heuristic} behind Theorem~\ref{thm0:contqLDP} is as follows. Let 
\begin{equation}
\label{RNdefind}
R^N_{t_1,\dots,t_N}(X), \qquad 0<t_1<\dots<t_N<\infty,
\end{equation}
denote the empirical process of $N$-tuples of words when $X$ is cut at the points $t_1,
\dots,t_N$ (i.e., when $T_i=t_i$ for $i=1,\dots,N$). Fix $Q \in \mathcal{P}^{\mathrm{inv,fin}}
(F^\N)$ and suppose that $Q$ is shift-ergodic. The probability $\Pr(R_N \approx Q \mid X)$ 
is an integral over all $N$-tuples $t_1,\dots,t_N$ such that $R^N_{t_1,\dots,t_N}(X) 
\approx Q$, weighted by $\prod_{i=1}^N \bar{\rho}(t_i-t_{i-1})$ (with $t_0=0$). The fact 
that $R^N_{t_1,\dots,t_N}(X) \approx Q$ has three consequences: 
\begin{enumerate}
\item[(1)] 
The $t_1,\dots,t_N$ must cut $\approx N$ substrings out of $X$ of total length 
$\approx N m_Q$ that look like the concatenation of words that are $Q$-typical, 
i.e., that look as if generated by $\Psi_Q$ (possibly with gaps in between). This 
means that most of the cut-points must hit atypical pieces of $X$. We expect 
to have to shift $X$ by $\approx\exp[N m_Q H(\Psi_Q \mid \WM)]$ in order to 
find the first contiguous substring of length $N m_Q$ whose empirical shifts lie 
in a small neighbourhood of $\Psi_Q$. By (\ref{ass:rhodensdecay}), the probability 
for the single increment $t_1-t_0$ to have the size of this shift is $\approx 
\exp[-N\alpha\,m_Q H(\Psi_Q \mid \WM)]$. 
\item[(2)] 
The ``number of local perturbations'' of $t_1,\dots,t_N$ preserving the property 
$R^N_{t_1,\dots,t_N}(X)\approx Q$ is $\approx \exp[NH_{\tau|K}(Q)]$, where 
$H_{\tau|K}$ stands for the \emph{conditional specific entropy (density) 
of word lengths under the law $Q$}.
\item[(3)] 
The statistics of the increments $t_1-t_0,\dots,t_N-t_{N-1}$ must be close to the 
distribution of word lengths under $Q$. Hence, the weight factor $\prod_{i=1}^N 
\bar{\rho}(t_i-t_{i-1})$ must be $\approx \exp[N \E_Q[\log\bar{\rho}(\tau_1)]]$ (at 
least, for $Q$-typical pieces).
\end{enumerate}
Since
\begin{equation}
\label{eqnsre1}
m_Q H(\Psi_Q \mid \WM) - H_{\tau|K}(Q) - \E_Q[\log\bar{\rho}(\tau_1)]
= H(Q \mid q_{\rho,\WM}),
\end{equation}
the observations made in (1)--(3) combine to explain the shape of the quenched 
rate function in \eqref{eq:Ique}. 
For further details, see \cite[Section~1.5]{BiGrdHo10}. 
\smallskip

\noindent
\emph{Note:} We have not defined $H_{\tau|K}(Q)$ rigorously here, nor do we 
prove \eqref{eqnsre1}. 
Our proof of Theorem~\ref{thm0:contqLDP} uses the above heuristic only very 
implicitly. Rather, it starts from the discrete-time quenched LDP derived in 
\cite{BiGrdHo10} and draws out Theorem~\ref{thm0:contqLDP} via 
control of 
exponential functionals through a coarse-graining 
approximation. 

\medskip\noindent
{\bf 1.}
We can include the cases $\alpha=1$ and $\alpha=\infty$ in \eqref{ass:rhodensdecay}. 

\begin{theorem} 
\label{mainthmboundarycases} 
Suppose that $\rho$ satisfies {\rm (\ref{ass:rhodensdecay}--\ref{ass:rhobar.reg0})}.\\ 
{\rm (a)} If $\alpha=1$, then the quenched LDP holds with $I^\mathrm{que}=I^\mathrm{ann}$
given by \eqref{eq:Iann}.\\
{\rm (b)} If $\alpha=\infty$, then the quenched LDP holds with rate function 
\begin{equation}
\label{eq:ratefctalphainfty} 
I^\mathrm{que}(Q) = 
\begin{cases} 
H(Q \mid Q_{\rho,\WM}) & \mbox{if} \;\; 
\lim\limits_{\tr\to\infty} m_{[Q]_\tr} H( \Psi_{[Q]_\tr} \mid \WM) = 0, \\
\infty & \mbox{otherwise}. 
\end{cases}
\end{equation}
\end{theorem}
 
\noindent
Theorem~\ref{mainthmboundarycases} is the continuous analogue of Birkner, Greven and den 
Hollander~\cite[Theorem 1.4]{BiGrdHo10} and is proved in Section~\ref{proofalpha1infty}. 

\medskip\noindent
{\bf 2.}
We can also include the case where $\bar{\rho}$ has an exponentially bounded tail: 
\begin{equation}
\label{ass:rhoexp}
\bar{\rho}(t) \leq e^{-\lambda t} \mbox{ for some } \lambda >0 \mbox{ and }
t \mbox{ large enough}.
\end{equation}

\begin{theorem}
\label{thmexp}
Suppose that $\rho$ satisfies {\rm (\ref{ass:rhodensdecay}--\ref{ass:rhobar.reg0})} and 
\eqref{ass:rhoexp}. Then, for $\WM$ a.e.\ $X$, the family $\mathscr{L}(R_N \mid X)$, 
$N\in\N$, satisfies the LDP on $\mathcal{P}^{\mathrm{inv}}(F^\N)$ with rate $N$ and 
with deterministic rate function $I^{\mathrm{que}}(Q)$ given by 
\begin{equation}
\label{eq:ratefctexptail}
I^\mathrm{que}(Q) = \left\{\begin{array}{ll}
H(Q \mid Q_{\rho,\WM})  &\mbox{if } Q \in \cR_\WM,\\
\infty &\mbox{otherwise},
\end{array}
\right.
\end{equation}
where
\begin{equation}
\cR_\WM = \left\{Q \in \mathcal{P}^{\mathrm{inv}}(F^\N)\colon\,
\wlim_{T\to\infty} \frac{1}{T} \int_0^T \delta_{\kappa(Y)} 
\circ (\theta^s)^{-1}\,ds = \WM \;\; \text{for $Q$-a.e.\ $Y$}\right\}.
\end{equation}
\end{theorem}

\noindent
Theorem~\ref{thmexp} is the continuous analogue of Birkner~\cite[Theorem 1]{Bi08} and is 
proved in Section~\ref{proofalpha1infty}. On the set $\mathcal{P}^{\mathrm{inv,fin}}(F^\N)$ 
the following holds:
\begin{equation}
\label{eq:calRchar}
\Psi_Q = \WM \quad \mbox{ if and only if } \quad Q \in \cR_\WM. 
\end{equation}
The equivalence in \eqref{eq:calRchar} is the continuous analogue of \cite[Lemma~2]{Bi08} 
(and can be proved analogously). 

\medskip\noindent
{\bf 3.}
By applying the contraction principle we obtain the quenched LDP for single words.
Let $\pi_1\colon\,F^\N \to F$ be the projection onto the first word, and let $\pi_1R_N
= R_N \circ (\pi_1)^{-1}$. 

\begin{corollary}
\label{cor:marginal}
Suppose that $\rho$ satisfies {\rm (\ref{ass:rhodensdecay}--\ref{ass:rhobar.reg0})}.
For $\WM$-a.e.\ $X$, the family $\mathscr{L}(\pi_1 R_N \mid X)$, $N\in\N$, satisfies the 
LDP on $\mathcal{P}(F)$ with rate $N$ and with deterministic rate function $I^\mathrm{que}_1$
given by
\begin{equation}
\label{eq:contractedratefct} 
I^\mathrm{que}_1(q) 
= \inf\big\{ I^\mathrm{que}(Q)\colon\, 
Q \in \mathcal{P}^{\mathrm{inv}}(F^\N),\,\pi_1 Q = q \big\}. 
\end{equation}
This rate function is lower semi-continuous, has compact levels sets, is convex, and has 
a unique zero at $q=q_{\rho,\WM}$.
\end{corollary}

\noindent
For general $q$ it is not possible to evaluate the infimum in (\ref{eq:contractedratefct}) explicitly. 
For $q$ with $m_q=\E_q[\tau]=\E_{q^{\otimes\N}}[\tau_1]=m_{q^{\otimes\N}}<\infty$ and 
$\Psi_{q^{\otimes\N}}=\WM$, we have $I^\mathrm{que}_1(q)=h(q \mid q_{\rho,\WM})$, the 
relative entropy of $q$ w.r.t.\ 
$q_{\rho,\WM}$. 

\medskip\noindent
{\bf 4.}
We expect assumption \eqref{ass:rhobar.reg0} to be redundant. In any case, it can be relaxed to
(see Section~\ref{prop1}):
\begin{equation} 
\label{ass:rhobar.reg-mg} 
\begin{minipage}{0.85\textwidth} 
$\mathrm{supp}(\rho)= \cup_{i=1}^M [a_i,b_i] \cup [a_{M+1},\infty)$ with $M \in \N$ 
and $0 \leq a_1 < b_1 \leq a_2 < \cdots < b_M \leq a_{M+1}<\infty$, and $\bar{\rho}$ 
is continuous and strictly positive on $\cup_{i=1}^M (a_i,b_i) \cup (a_{M+1},\infty)$ 
and varies regularly near each of the finite endpoints of these intervals. 
\end{minipage}
\end{equation}

\medskip\noindent
{\bf 5.}
It is possible to extend Theorem~\ref{thm0:contqLDP} to other classes of random environments,
as stated in the following theorem whose proof will not be spelled out in the present paper.  

\begin{theorem} 
Theorems~{\rm \ref{thm0:contqLDP}--\ref{thmexp}} and Corollary~{\rm \ref{cor:marginal}} 
carry over verbatim when the Brownian motion $X$ is replaced by a $d$-dimensional L\'evy 
process $\bar{X}$ with the property that $\E[e^{\langle \lambda, \bar{X}_1 \rangle}] < \infty$ for all 
$\lambda \in \R^d$ (where $\langle\cdot\rangle$ denotes the standard inner product), $\WM$ is 
replaced by the law of $\bar{X}$, and in the definition of $F$ in \eqref{eq:defF} continuous paths 
are replaced by c\`adl\`ag paths. 
\end{theorem}
 
\medskip\noindent
{\bf 6.} 
In the companion paper \cite{BidHo13b} we apply Theorem~\ref{thm0:contqLDP} and the techniques
developed in the present paper to the Brownian copolymer. In this model a c\`adl\`ag path, representing 
the configuration of the polymer, is rewarded or penalised for staying above or below a linear interface, 
separating oil from water, according to Brownian increments representing the degrees of hydrophobicity 
or hydrophilicity along the polymer. The reference measure for the path can be either the Wiener 
measure or the law of a more general L\'evy process. We derive a variational formula for the quenched free energy, from 
which we deduce a variational formula for the slope of the quenched critical line. This critical line 
separates a \emph{localized phase} (where the copolymer stays close to the interface) from a 
\emph{delocalized phase} (where the copolymer wanders away from the interface). This slope has 
been the object of much debate in recent years. The Brownian copolymer is the unique attractor in the 
limit of weak interaction for a whole universality class of discrete copolymer models. See Bolthausen 
and den Hollander~\cite{BodHo97}, Caravenna and Giacomin~\cite{CaGi10}, Caravenna, Giacomin 
and Toninelli~\cite{CaGiTo12} for details.


\section{Proof of Theorem~\ref{thm0:contqLDP}}
\label{proof}

The proof proceeds via a \emph{coarse-graining} and \emph{truncation} argument. In 
Section~\ref{cogrtrun} we set up the coarse-graining and the truncation, and state 
a quenched LDP for this setting that follows from the quenched LDP in \cite{BiGrdHo10} 
and serves as the starting point of our analysis (Proposition~\ref{thm00:contqLDP} and
Corollary~\ref{prop:qLDPhtr} below). In Section~\ref{3prop} we state three propositions
(Propositions~\ref{prop:LambdaPhilimit1tr}--\ref{prop:Ique.tr.cont} below), involving 
expectations of exponential functionals of the coarse-grained truncated empirical process
as well as approximation properties of the associated rate function, and we use these 
propositions to complete the proof of Theorem~\ref{thm0:contqLDP} with the help of 
Bryc's inverse of Varadhan's lemma. In Section~\ref{3lem} we state and prove two lemmas 
that are used in Section~\ref{3prop}, involving approximation estimates under the 
coarse-graining. The proof of the three propositions is deferred to 
Sections~\ref{props}--\ref{removeass}.


\subsection{Preparation: coarse-graining and truncation}
\label{cogrtrun}


\subsubsection{Coarse-graining}
\label{cogr}

Suppose that, instead of the absolutely continuous $\rho$ introduced in 
Section~\ref{setting}, we are given a discrete 
$\hat{\rho}$ with $\mathrm{supp}(\hat\rho) \subset h \N$ for 
some $h>0$. Let 
\begin{equation} 
\label{def:Eh}
E_h = \{ f \in C([0,h])\colon\,f(0)=0 \}.
\end{equation} 
Path pieces of length $h$ in a continuous-time scenario can act as ``letters'' in a 
discrete-time scenario, and therefore we can use the results from \cite{BiGrdHo10}. 
Note that $(E_h)^\N$ as a metric space is isomorphic to $\{f \in C([0,\infty))\colon\,f(0)=0\}$ 
via the obvious glueing together of path pieces into a single path, provided the latter is 
given a suitable metric that metrises locally uniform convergence. Similarly, we can identify 
$\mathcal{P}^{\mathrm{inv}}(E_h^\N)$ with 
\begin{equation}
\mathcal{P}^{h\text{-}\mathrm{inv}}(C([0,\infty))) 
= \big\{ Q \in \mathcal{P}(C([0,\infty))) \colon\, Q = Q \circ (\theta^{h})^{-1} \big\}, 
\end{equation}
which is the set of laws on continuous paths that are invariant under a time shift by 
$h$. Note that the set 
\begin{equation} 
\label{def:Fh}
F_h = \bigcup_{t \in h\N} \Big(\{t\} \times 
\big\{ f \in C([0,\infty))\colon\,f(0)=0, f(s)=f(t) \; \text{for} \; s > t \big\}\Big)
\end{equation}
is isomorphic to $\widetilde{E_h} = \cup_{n \in \N} \left(E_h\right)^n$ via the map
$\iota_h\colon\, F_h \to \widetilde{E_h}$ defined by 
\begin{equation}
\label{iotahdef} 
\iota_h\big( (nh, f)\big) = \Big( \big(f\big((\,\cdot+(i-1)h) \wedge ih\big)
-f((i-1)h)\big)\Big)_{i=1,\dots,n}, \qquad (nh, f) \in F_h. 
\end{equation}

For $Q \in \mathcal{P}^{\mathrm{inv, fin}}(F_h^\N)$, define 
\begin{equation}
\label{eq:definitionPsiQh}
\Psi_{Q,h}(A) = \frac1{m_Q} \E_Q\left[ \sum_{i=0}^{\tau_1-1} 
\1_A\big(\theta^i \iota_h \kappa(Y)\big)\right] 
= \frac1{h \, m_Q} \E_Q\left[ \int_0^{h \tau_1} 
\1_A\big( \kappa(Y)(h \lfloor u/h\rfloor +s))_ {s \geq 0} \big) \, du\right] 
\end{equation}
for $A \subset C([0,\infty))$ measurable, where $\tau_1$ is the length of the first word 
(counted in letters, so that the length of the first word viewed as an element of 
$F_h$ is $h\tau_1$) 
and $\theta$ is the left-shift acting on $(E_h)^\N$. The right-most 
expression in \eqref{eq:definitionPsiQh} can be viewed as a coarse-grained version of 
(\ref{eq:PsiQcont}). The following coarse-grained version of the quenched LDP serves 
as our starting point.

\begin{proposition}
\label{thm00:contqLDP}
Fix $h>0$. Suppose that $\mathrm{supp}(\hat\rho) \subset h\N$ and $\lim_{n\to\infty} 
\log \hat\rho(\{nh\})/\log n = -\alpha$ with $\alpha \in (1,\infty)$. Then, for $\WM$ a.e.\ 
$X$, the family $\mathscr{L}(R_N \mid X)$, $N\in\N$, satisfies the LDP on 
$\mathcal{P}^{\mathrm{inv}}((\widetilde{E_h})^\N)$ with rate $N$ and with 
deterministic rate function given by
\begin{equation} 
\label{eq:ratefctfixedh}
I^{\mathrm{que}}_h(Q) = H(Q \mid Q_{\hat\rho,\WM}) 
+ (\alpha-1) m_Q H(\Psi_{Q,h} \mid \WM), 
\qquad Q \in \mathcal{P}^{\mathrm{inv,fin}}((\widetilde{E_h})^\N),
\end{equation}
and 
\begin{equation}
I^{\mathrm{que}}_h(Q) = \lim_{\tr\to\infty} I^{\mathrm{que}}_h([Q]_\tr),
\qquad Q \notin \mathcal{P}^{\mathrm{inv,fin}}((\widetilde{E_h})^\N),
\end{equation} 
where $Q_{\hat\rho,\WM}=(q_{\hat\rho,\WM})^{\otimes\N}$ with $q_{\hat\rho,\WM}$ 
defined as in \eqref{qrhoWdef}, and $\Psi_{Q,h}$ defined via \eqref{eq:definitionPsiQh}.
\end{proposition}

\begin{proof}
The claim follows from \cite[Corollary 1.6]{BiGrdHo10}  by using $E_h$ as letter 
space and observing that $\widetilde{E_h}=\iota_h(F_h)$. Note that $F_h^\N$ 
is a closed subspace of $F^\N$. Since $\mathrm{supp}(\hat\rho) \subset h\N$ by assumption, 
we have $I^{\mathrm{que}}_h(Q) \geq  H(Q \mid Q_{\hat\rho,\WM}) = \infty$ for any 
$Q \in \mathcal{P}^{\mathrm{inv}}(F^\N)$ with $Q\big(F^\N \setminus F_h^\N\big)>0$.
Therefore we can consider the random variable $R_N$ as taking values in 
$\mathcal{P}^{\mathrm{inv}}((\widetilde{E_h})^\N)$, $\mathcal{P}^{\mathrm{inv}}
(F_h^\N)$ or $\mathcal{P}^{\mathrm{inv}}(F^\N)$, without changing the statement 
of Proposition~\ref{thm00:contqLDP}. Note that $I^{\mathrm{que}}_h$ is finite only 
on $\mathcal{P}^{\mathrm{inv}}(F_h^\N)\subset\mathcal{P}^{\mathrm{inv}}(F^\N)$.
\end{proof}

We want to pass to the limit $h \downarrow 0$ and deduce Theorem~\ref{thm0:contqLDP} 
from Proposition~\ref{thm00:contqLDP}. However, an immediate application of a projective 
limit at the level of letters appears to be impossible. Indeed, when we replace 
$h$ by $h/2$, each ``$h$-letter'' turns into two ``$(h/2)$-letters'', so the word 
length changes, and even diverges as $h \downarrow 0$. This does not fit well with 
the way the projective limit was set up in \cite[Section 8]{BiGrdHo10}, where the 
internal structure of the letters was allowed to become increasingly richer, but 
the word length had to remain the same. In some sense, the problem is that we have 
finite words but only infinitesimal letters (i.e., there is no fixed letter space). 
To remedy this, we proceed as follows. For fixed discretisation length $h>0$ we 
have a fixed letter space, and so Proposition~\ref{thm00:contqLDP} applies. We will 
handle the limit $h \downarrow 0$ via Bryc's inverse of Varadhan's lemma. This will
require several intermediate steps.


\subsubsection{Truncation}
\label{trun}

It will be expedient to work with a \emph{truncated} version of Proposition~\ref{thm00:contqLDP}.  
For $h>0$, let $\lceil t \rceil_h =h\lceil t/h \rceil$ for $t \in (0,\infty)$ and put $\lceil\rho\rceil_h
=\rho\,\circ (\lceil\cdot\rceil_h)^{-1}$, i.e., 
\begin{equation} 
\label{def:rho.h.trunc}
\lceil \rho \rceil_h = \sum_{i\in\N} w_{h,i} \delta_{ih} \, \in \mathcal{P}(h\N) 
\subset \mathcal{P}((0,\infty)),
\end{equation}
where 
\begin{equation}
w_{h,i} = \rho\big(((i-1)h,ih]\big) = \int_{(i-1)h}^{ih} \bar{\rho}(x)dx 
\end{equation}
is the coarse-grained version of $\rho$ from Section~\ref{setting}. It is easily checked that 
\eqref{ass:rhodensdecay} implies
\begin{equation} 
\lim_{n\to\infty} \frac{\log \lceil \rho\rceil_h(\{nh\})}{\log n} = -\alpha.
\end{equation}
Write$ \mathscr{L}_{\lceil \rho \rceil_h}([R_N]_\tr \mid X)$ for the law of the
truncated empirical process $[R_N]_\tr$ conditional on $X$ when the $\tau_i$'s are 
drawn according to $\lceil \rho \rceil_h$.

\begin{corollary}
\label{prop:qLDPhtr}
For $\WM$-a.e.\ $X$, the family $\mathscr{L}_{\lceil \rho \rceil_h}([R_N]_\tr \mid X)$, 
$N\in\N$, satisfies the LDP on $\mathcal{P}^\mathrm{inv}(F_h^\N)$ with 
rate $N$ and with deterministic rate function given by
\begin{equation} 
\label{eq:Ique.h.tr}
I^{\mathrm{que}}_{h,\tr}(Q) 
=  H\bigl(Q \mid Q_{\lceil\rho\rceil_h,\WM,\tr}\bigr) 
+ (\alpha-1) m_Q H(\Psi_{Q,h} \mid \WM)
\end{equation} 
with $Q_{\lceil\rho\rceil_h,\WM,\tr} = ([q_{\lceil \rho \rceil_h, \WM}]_\tr)^{\otimes \N}$.
\end{corollary}

\begin{proof}
This follows from Proposition~\ref{thm00:contqLDP} and the contraction principle. 
Alternatively, it follows from the proofs of \cite[Theorem 1.2 and Corollary 1.6]{BiGrdHo10}.
\end{proof}

\noindent
Note that $I^{\mathrm{que}}_{h,\tr}(Q)=\infty$ when under $Q$ the word lengths are 
not supported on $h\N \cap (0,\tr]$.


\subsection{Application of Bryc's inverse of Varadhan's lemma}
\label{3prop}

In this section we state three propositions 
(Propositions~\ref{prop:LambdaPhilimit1tr}--\ref{prop:Ique.tr.cont} below)
and show that these imply Theorem~\ref{thm0:contqLDP}. The proof of
these propositions is deferred to Sections~\ref{props}--\ref{removeass}.


\subsubsection{Notations}
\label{subsect:notations}

In what follows we obtain the quenched LDP for the truncated empirical process 
$[R_N]_{\tr}$ by letting $h \downarrow 0$ in the coarse-grained and truncated 
empirical process $[R_{N,h}]_\tr$ with $\tr \in \N$ fixed (for a precise definition,
see \eqref{def:RNh} in Section~\ref{prop1}) and afterwards letting $\tr\to\infty$. 
(We assume that $\tr \in \N$ and $h=2^{-M}$ for some $M \in \N$, in particular, 
$\tr$ is an integer multiple of $h$.) 

In the coarse-graining procedure, it may happen that a very short continuous 
word $y=(t,f) \in F$ disappears, namely, when $0<t<h$. We remedy this by formally 
allowing ``empty'' words, i.e., by using 
\begin{align} 
\label{def:Fhat}
\widehat{F} = F \cup \big\{ (0,0) \big\} 
= \bigcup_{t \geq 0} \Big(\{t\} \times 
\big\{ f \in C([0,\infty))\colon\,f(0)=0, f(s)=f(t) \; \text{for} \; s > t \big\}\Big)
\end{align}
as word space instead of $F$. The metric on $F$ defined in Appendix~\ref{metrics} 
extends in the obvious way to $\widehat{F}$.

Before we proceed, we impose \emph{additional regularity assumptions} on $\bar{\rho}$ 
that will be required in the proof of Proposition~\ref{prop:LambdaPhilimit1tr}. Recall from
\eqref{ass:rhobar.reg0} that $\mathrm{supp}(\rho) = [s_*,\infty)$. Let
\begin{align}
\label{eq:Vbarrhodef} 
V_{\bar{\rho}}(t,h) = 
\sup_{v \in (0,2h)} \left| \log 
\frac{\int_t^{t+h} \bar{\rho}(s)\,ds}{\int_{t+v}^{t+h+v} \bar{\rho}(s)\,ds} \right|,
\qquad t,h>0.
\end{align}
We assume that there exist monotone sequences $(\eta_n)_{n \in \N}$ and $(A_n)_{n\in\N}$,
with $\eta_n \in (0,1)$ and $A_n \subset (s_*,\infty)$ satisfying $\lim_{n\to\infty} \eta_n = 0$ 
and $\lim_{n\to\infty} A_n = (s_*,\infty)$, such that $(s_*,\infty) \setminus A_n$ is a (possibly 
empty) union of finitely many bounded intervals whose endpoints lie in $2^{-n} \N_0$, and
\begin{align} 
\label{ass:rhobar.reg2}
\sup_{t \in A_n} V_{\bar{\rho}}(t,2^{-n}) \leq \eta_n \qquad \forall\,n \in\N. 
\end{align}
In addition, we assume that there exists an $\eta_0 < \infty$ such that
\begin{align} 
\label{ass:rhobar.reg1}
\sup_{n \in \N} \sup_{t \in (s_*,\infty)} V_{\bar{\rho}}(t,2^{-n}) \leq \eta_0. 
\end{align}
These assumptions will be removed only in Section~\ref{removeass}. Note that 
\eqref{ass:rhobar.reg2}--\eqref{ass:rhobar.reg1} are satisfied when $\bar{\rho}$ 
is continuous and strictly positive on $(s_*,\infty)$ and varies regularly near
$s_*$ and at $\infty$.


\subsubsection{Proof of Theorem~\ref{thm0:contqLDP} subject to 
(\ref{ass:rhobar.reg2}--\ref{ass:rhobar.reg1}) and three propositions}

\begin{proof}
A function $g$ on $\widehat{F}^\ell$ is Lipschitz when it satisfies 
\begin{align} 
\label{eq:g_Lipschitz}
\big| g(y^{(1)},\dots,y^{(\ell)}) - g(y^{(1)}{}',\dots,y^{(\ell)}{}') \big| 
\leq C_g \sum_{j=1}^\ell d_F(y^{(j)},y^{(j)}{}')
\quad \mbox{ for some } C_g < \infty. 
\end{align}
Consider the class $\mathscr{C}$ of functions $\Phi\colon\,\mathcal{P}(\widehat{F}^\N) \to \R$ 
of the form 
\begin{equation} 
\label{eq:Phiform1}
\Phi(Q) = \int_{\widehat{F}^{\ell_1}} g_1 \, d\pi_{\ell_1} Q \wedge \cdots \wedge 
\int_{\widehat{F}^{\ell_m}} g_m \, d\pi_{\ell_m}Q, 
\quad Q \in \mathcal{P}^{\mathrm{inv}}(\widehat{F}^\N), 
\end{equation}
where $m \in N$, $\ell_1,\dots, \ell_m \in \N$, and $g_i$ is a bounded Lipschitz function 
on $\widehat{F}^{\ell_i}$ for $i=1,\dots,m$. This class is well-separating and thus is sufficient 
for the application of Bryc's lemma (see Dembo and Zeitouni~\cite[Section 4.4]{DeZe98}. 

Our first proposition identifies the exponential moments of $[R_N]_\tr$.  
\begin{proposition} 
\label{prop:LambdaPhilimit1tr}  
The families $\mathscr{L}(R_N \mid X)$, $N\in\N$, and 
$\mathscr{L}([R_N]_{\tr} \mid X)$, 
$\tr \in \N$, are exponentially tight $X$-a.s. Moreover, for $\Phi \in \mathscr{C}$,
\begin{equation} 
\label{eq:LambdaPhilimit1tr}
\Lambda_{0,\tr}(\Phi) 
= \lim_{N\to\infty} \frac1N \log \E\Big[ \exp\big( N \Phi([R_N]_{\tr})\big) ~\Big|~ X \Big] 
= \lim_{h \downarrow 0} \Lambda_{h,\tr}(\Phi) \quad 
\text{exists $X$-a.s.},
\end{equation}
where $\Lambda_{h,\tr}$ is the generalised convex transform of $I^{\mathrm{que}}_{h,\tr}$
given by
\begin{equation}
\label{eq:Phihtrform} 
\Lambda_{h,\tr}(\Phi) 
= \sup_{Q \in \mathcal{P}^{\mathrm{inv, fin}}((\widetilde{E_h})^\N)} 
\big\{ \Phi(Q) - I^{\mathrm{que}}_{h,\tr}(Q) \big\}. 
\end{equation}
Furthermore, for $\Phi\in \mathscr{C}$, 
\begin{equation} 
\label{eq:LambdaPhilimit3}
\Lambda(\Phi) 
= \lim_{N\to\infty} \frac1N \log \E\Big[ \exp\big( N \Phi(R_N)\big) ~\Big|~ X \Big] 
= \lim_{\tr\to\infty} \Lambda_{0,\tr}(\Phi) \quad 
\text{exists $X$-a.s.}
\end{equation}
\end{proposition} 
\smallskip

Our second proposition identifies the limit in \eqref{eq:LambdaPhilimit1tr} as the 
generalised convex transform of $I^{\mathrm{que}}_{\tr}$ defined in \eqref{def:Ique.tr}, 
\begin{equation} 
I^{\mathrm{que}}_{\tr}(Q) 
= \begin{cases} H\bigl( Q \mid Q_{\rho,\WM,\tr} \bigr) 
+ (\alpha-1) m_Q H\left( \Psi_Q \mid \WM \right) & 
\text{if} \; Q \in \mathcal{P}^{\mathrm{inv}}(F_{0,\tr}^\N), \\[1ex]
\infty &  \text{otherwise}, 
\end{cases}
\end{equation}
and implies that the latter is the rate function for the truncated empirical process 
$[R_N]_\tr$. 

\begin{proposition} 
\label{prop:qLDPtrunc1}
For $\Phi\in \mathscr{C}$, 
\begin{equation}
\label{eq:Lambda0tr}
\Lambda_{0,\tr}(\Phi) = 
\sup_{Q \in \mathcal{P}^{\mathrm{inv}}(F_{0,\tr}^\N)}
\big\{ \Phi(Q) - I^{\mathrm{que}}_{\tr}(Q) \big\}.
\end{equation}
Furthermore, for $\WM$-a.e.\ $X$, the family $\mathscr{L}([R_N]_{\tr} \mid X)$, $N\in\N$, 
satisfies the LDP on $\mathcal{P}^{\mathrm{inv}}(F_{0,\tr}^\N)$ with deterministic rate 
function $I^{\mathrm{que}}_{\tr}$.
\end{proposition} 

Note that the family of truncation operators $[\cdot]_{\tr}$ forms a projective system as the 
truncation level $\tr$ increases. Hence we immediately get from Proposition~\ref{prop:qLDPtrunc1} 
and the Dawson-G\"artner projective limit LDP (see \cite[Theorem 4.6.1]{DeZe98}) that the 
family $\mathscr{L}(R_N \mid X)$, $N\in\N$, satisfies the LDP with rate function $Q \mapsto 
\sup_{\tr \in \N} I^{\mathrm{que}}_{\tr}([Q]_{\tr})$. Furthermore, since the projection can start 
at any initial level of truncation, we also know that the rate function is given by $Q \mapsto 
\sup_{\tr \geq n} I^{\mathrm{que}}_{\tr}([Q]_{\tr})$ for any $n\in\N$. Thus, Proposition~\ref{prop:qLDPtrunc1} 
in fact implies that the rate function is given by 
\begin{align} 
\label{eq:Ique-DGform}
\tilde{I}^{\mathrm{que}}(Q) = \limsup_{\tr\to\infty} I^{\mathrm{que}}_{\tr}([Q]_{\tr}).
\end{align}
At this point, it remains to prove that $\tilde{I}^{\mathrm{que}}$ from \eqref{eq:Ique-DGform} 
actually equals $I^{\mathrm{que}}$ from \eqref{eq:Iquelimitform} and has the form claimed 
in \eqref{eq:Ique}. 

This is achieved via the following proposition, 
note that \eqref{eq:lemma:Ique.tr.cont1} is the continuous analogue of 
\cite[Lemma A.1]{BiGrdHo10}.

\begin{proposition} 
\label{prop:Ique.tr.cont}
{\rm (1)} For $Q \in \mathcal{P}^{\mathrm{inv,fin}}(F^{\N})$, 
\begin{align}
\label{eq:lemma:Ique.tr.cont1}
\lim_{\tr\to\infty} I^{\mathrm{que}}_\tr([Q]_\tr) 
= H(Q \mid Q_{\rho,\WM}) + (\alpha-1) m_Q H(\Psi_Q \mid \WM).
\end{align}
{\rm (2)} For $Q \in \mathcal{P}^{\mathrm{inv}}(F^{\N})$ with $m_Q = \infty$ 
and $H(Q \mid Q_{\rho, \WM})<\infty$ there exists a sequence 
$(\widetilde{Q}_\tr)_{\tr \in \N}$ in $\mathcal{P}^{\mathrm{inv,fin}}(F^{\N})$ 
such that $\wlim_{\tr\to\infty} \widetilde{Q}_\tr = Q$ and 
\begin{align} 
\label{eq:lemma:Ique.tr.approx2}
\tilde{I}^{\mathrm{que}}(\widetilde{Q}_\tr) \leq I^{\mathrm{que}}_\tr([Q]_\tr) 
+ o(1), \qquad \tr \to \infty.
\end{align}
\end{proposition} 

\noindent
Proposition~\ref{prop:Ique.tr.cont} (1) implies that for $Q \in \mathcal{P}^{\mathrm{inv,fin}}
(F^{\N})$ the $\limsup$ in \eqref{eq:Ique-DGform} is a limit, i.e., it implies \eqref{eq:Iquelimitform} 
on $\mathcal{P}^{\mathrm{inv,fin}}(F^{\N})$ and also \eqref{eq:Ique}.

To prove \eqref{eq:Iquelimitform} for $Q \in \mathcal{P}^{\mathrm{inv}}(F^{\N})$ with 
$m_Q = \infty$ and $H(Q \mid Q_{\rho, \WM})<\infty$, consider $\widetilde{Q}_\tr$ as 
in Proposition~\ref{prop:Ique.tr.cont} (2). Then
\begin{align} 
\tilde{I}^{\mathrm{que}}(Q) \leq \liminf_{\tr\to\infty} \tilde{I}^{\mathrm{que}}(\widetilde{Q}_\tr) 
\leq \liminf_{\tr\to\infty} I^{\mathrm{que}}_\tr([Q]_\tr),
\end{align}
where the first inequality uses that $\tilde{I}^{\mathrm{que}}$ is lower semi-continuous (being
a rate function by the Dawson-G\"artner projective limit LDP), and the second inequality is a 
consequence of \eqref{eq:lemma:Ique.tr.approx2}. For $Q \in \mathcal{P}^{\mathrm{inv}}(F^{\N})$ 
with $H(Q \mid Q_{\rho, \WM})=\infty$ we have 
\begin{align} 
\liminf_{\tr\to\infty} I^{\mathrm{que}}_\tr([Q]_\tr) 
\geq \liminf_{\tr\to\infty} H([Q]_\tr \mid [Q_{\rho,\WM}]_\tr) 
= H(Q \mid Q_{\rho, \WM}) = \infty,
\end{align} 
i.e., also in this case the $\limsup$ in \eqref{eq:Ique-DGform} 
is a limit and \eqref{eq:Iquelimitform} holds. 
\smallskip

It remains to prove the properties of $I^{\mathrm{que}}$ claimed in Theorem~\ref{thm0:contqLDP}: 
lower semi-continuity of $I^{\mathrm{que}}=\tilde{I}^{\mathrm{que}}$ follows from the representation 
via the Dawson-G\"artner projective limit LDP in \eqref{eq:Ique-DGform}; compactness of the level 
sets of $I^{\mathrm{que}}$ and the fact that $Q_{\rho,\WM}$ is the unique zero of $Q \mapsto 
I^{\mathrm{que}}(Q)$ are inherited from the corresponding properties of $I^{\mathrm{ann}}$ 
because $I^{\mathrm{que}} \leq I^{\mathrm{ann}}$; affineness of $Q \mapsto I^{\mathrm{que}}(Q)$ 
can be checked as in \cite[Proof of Theorem 1.3]{BiGrdHo10}.
\end{proof}

\medskip\noindent
{\bf Remark.}\  
Theorem~\ref{thm0:contqLDP} together with Varadhan's lemma implies that
\begin{align} 
\label{eq:LambdaPhilimit2}
\Lambda(\Phi) 
= \sup_{Q \in \mathcal{P}^{\mathrm{inv, fin}}(F^\N)} 
\big\{ \Phi(Q) - I^{\mathrm{que}}(Q) \big\}, \qquad \Phi\in \mathscr{C}, 
\end{align}
and identifies $I^{\mathrm{que}}(Q)$ as the generalised convex transform
\begin{equation}
\label{eq:Iquetrafo} 
I^{\mathrm{que}}(Q) = \sup_{\Phi \in \mathscr{C}} 
\big\{ \Phi(Q) -  \Lambda(\Phi)\big\}, \qquad Q \in \mathcal{P}^{\mathrm{inv}}(F^\N)
\end{equation}
(see \cite[Theorems 4.4.2 and 4.4.10]{DeZe98}). The supremum in \eqref{eq:Iquetrafo} can 
also be taken over all continuous bounded functions on $\mathcal{P}^{\mathrm{inv}}(F^\N)$.



\subsection{Continuity of the empirical process under coarse-graining}
\label{3lem}

Before embarking on the proof of Propositions~\ref{prop:LambdaPhilimit1tr}--\ref{prop:Ique.tr.cont}
in Section~\ref{props}, we state and prove two approximation lemmas  
(Lemmas~\ref{obs:dSclose1}--\ref{obs:Rdiscdiff} below) that will be needed along the way.

For $N \in \N$, $0=t_0 < t_1 < \cdots < t_N$ and $\varphi \in C([0,\infty))$, let 
$y_\varphi = (y_\varphi^{(i)})_{i\in\N}$ with
\begin{equation}
\label{def:yphii}
y_\varphi^{(i)} = \Big(t_i-t_{i-1}, \big(\varphi((t_{i-1}+s) \wedge t_i)
-\varphi(t_{i-1})\big)_{s\geq 0}\Big) \in F, \qquad i=1,\dots,N,
\end{equation}
and define  
\begin{equation}
\label{eq:defRNphi}
R_{N;t_1,\dots,t_N}(\varphi) 
= \frac1N \sum_{i=0}^{N-1} 
\delta_{\widetilde{\theta}^i y_\varphi^{N\text{-}\mathrm{per}}}
\in \mathcal{P}^{\mathrm{inv}}(F^\N).
\end{equation}
We need a Skorohod-type distance $d_S$ on paths, which is defined in 
Appendix~\ref{metrics}.

\begin{lemma} 
\label{obs:dSclose1}
Let $i, j \in \N$, $i \leq j$, and $t, t' \in (0,\infty)$, $t<t'$, be such that 
$(i-1)h < t \leq ih$, $(j-1)h < t' \leq jh$. Then, for any $\varphi \in C([0,\infty))$ 
and $k \in \N$, 
\begin{align} 
\label{eq:dS_wishful}
& d_S\big( \varphi((ih+\cdot) \wedge jh), \varphi((t+\cdot) \wedge t')\big) \notag \\
& \leq \log\tfrac{k+1}{k} + 2 \sup_{(i-1)h \leq s \leq (i+k)h} 
|\varphi(s)-\varphi((i-1)h)| + 2 \sup_{(j-1)h \leq s \leq jh} 
|\varphi(s)-\varphi((j-1)h)|.
\end{align}
The same bound holds for $d_S([\varphi((ih+\cdot) \wedge jh)]_\tr, 
[\varphi((t+\cdot) \wedge t')]_\tr)$ for any truncation length $\tr > 0$.
\end{lemma}

\begin{proof}
Without loss of generality we may assume that $j \geq i+k$ (otherwise, employ the
trivial time transform $\lambda(s)=s$ and estimate the left-hand side of 
\eqref{eq:dS_wishful} by the second term in the right-hand side of \eqref{eq:dS_wishful}), 
and use the time transformation
\begin{align} 
\lambda(s) = 
\begin{cases} 
s \, \frac{(i+k)h-t}{kh} & \text{if} \; s < kh, \\
s + ih -t & \text{if} \; s \geq kh.
\end{cases}
\end{align} 
In that case $\lambda(s)+t=s+ih$ for $s \geq kh$ and $\gamma(\lambda)=|\log[((i+k)h-t)/kh]|
\leq \log\frac{k+1}{k}$. The same argument applies to the truncated paths 
$[\varphi((ih+\cdot) \wedge jh)]_\tr$ and $[\varphi((t+\cdot) \wedge t')]_\tr$
(in fact, we can drop the third term in the right-hand side of \eqref{eq:dS_wishful} 
when $(j-1)h>\tr$).
\end{proof} 

\begin{lemma} 
\label{obs:Rdiscdiff} 
Let $\varphi \in C([0,\infty))$, $h>0$, $N\in\N$ and $t_0=0<t_1<\cdots<t_N$. Let $\ell \in \N$, 
and let $g\colon\,\widehat{F}^\ell \to \R$ be bounded Lipschitz with Lipschitz constant $C_g$. 
Then, for $k \in \N$ with $k \geq \ell$, 
\begin{equation}
\begin{aligned} 
&N \Big| \int_{\widehat{F}^{\ell}} g\, d\pi_\ell R_{N;t_1,\dots,t_N}(\varphi) 
- \int_{\widehat{F}^{\ell}} g\, d\pi_\ell R_{N;\lceil t_1 \rceil_h,\dots,\lceil t_N \rceil_h}(\varphi) \Big|\\
&\qquad \leq 
4\ell \|g\|_\infty  + C_g \ell N \bigl(2h + \log{\textstyle\frac{k+1}{k}}\bigr) 
+ 4 C_g \ell \sum_{i=1}^N \sup_{\lceil t_i \rceil_h-h \leq s \leq \lceil t_i \rceil_h+kh}
\big| \varphi(s) -\varphi({\lceil t_i \rceil_h-h}) \big|,
\end{aligned}
\end{equation}
where $\pi_\ell\colon\,\widehat{F}^\N \to \widehat{F}^\ell$ denotes the projection onto the first 
$\ell$ coordinates. The same bound holds for the truncated versions $[R_{N;t_1,\dots,t_N}
(\varphi)]_\tr$ and $[R_{N;\lceil t_1 \rceil_h,\dots,\lceil t_N \rceil_h}(\varphi) ]_\tr$ for any 
truncation length $\tr>0$.  
\end{lemma}

\begin{proof}
For $i=1,\dots,N$, 
recall $y_\varphi^{(i)}$ from \eqref{def:yphii}, i.e., $y_\varphi^{(i)}$ is the $i$-th word obtained 
by cutting the continuous path $\varphi$ along the time points $t_1,\ldots,t_n$, and let
\begin{align} 
\tilde{y}_\varphi^{(i,h)} 
& = \Bigl(\lceil t_i \rceil_h-\lceil t_{i-1} \rceil_h, 
\bigl(\varphi((\lceil t_{i-1} \rceil_h+s) \wedge \lceil t_i \rceil_h)
-\varphi(\lceil t_{i-1} \rceil_h)\bigr)_{s\geq 0}\Bigr), 
\end{align} 
be the analogous quantity when the $h$-discretised time points $\lceil t_1 \rceil_h,\ldots,\lceil t_N \rceil$ 
are used. By Lemma~\ref{obs:dSclose1} we have 
\begin{equation}
\begin{aligned} 
d_F\bigl(y_\varphi^{(i)}, \tilde{y}_\varphi^{(i,h)}\bigr) 
&\leq \bigl(2h + \log{\textstyle\frac{k+1}{k}}\bigr) 
+ 2 \sup_{\lceil t_{i-1} \rceil_h-h \leq s \leq \lceil t_{i-1} \rceil_h+kh}
\big| \varphi(s) -\varphi({\lceil t_{i-1} \rceil_h}-h) \big|\\ 
&\qquad\qquad + 2 \sup_{\lceil t_{i} \rceil_h-h \leq s \leq \lceil t_{i} \rceil_h} 
\big| \varphi(s) -\varphi({\lceil t_{i} \rceil_h}-h) \big|.
\end{aligned}
\end{equation} 
Writing $\tilde{y}^{(h)}=(\tilde{y}^{(i,h)})_{i\in\N}$ and putting, similarly as in
\eqref{eq:defRNphi},
\begin{equation}
\label{eq:defRNphi_hdisc}
R_{N;\lceil t_1 \rceil_h,\dots,\lceil t_N \rceil_h}(\varphi) 
= \frac1N \sum_{i=0}^{N-1} 
\delta_{\widetilde{\theta}^i(\tilde{y}^{(h)})^{N\text{-}\mathrm{per}}},
\end{equation}
we see that the claim follows from \eqref{eq:g_Lipschitz} in combination with 
Lemma~\ref{obs:dSclose1}. Note that possible boundary effects due to the periodisation 
are estimated by the term $4\ell\|g\|_\infty$. The observation about the truncated 
versions of the empirical process follow analogously from Lemma~\ref{obs:dSclose1}.
\end{proof}


\section{Proof of Propositions~\ref{prop:LambdaPhilimit1tr}--\ref{prop:Ique.tr.cont}}
\label{props}


\subsection{Proof of Proposition~\ref{prop:LambdaPhilimit1tr}}
\label{prop1}

\begin{proof}
The proof comes in 3 Steps.

\paragraph{Step 1.} 
A.s.\ exponential tightness of the family $\mathscr{L}(R_N \mid X)$, $N\in\N$, 
is standard, because the family of unconditional distributions $\mathscr{L}(R_N)$ 
satisfies the LDP with a rate function that has compact level sets. Indeed, let 
$M > 0$, and pick a compact set $K \subset \mathcal{P}^{\mathrm{inv}}(F^\N)$ such 
that $\limsup_{N\to\infty} \tfrac1N \log \Pr(R_N \not\in K) \leq -2M$. Then 
$\Pr(\Pr(R_N \not\in K \mid X) > e^{-MN}) \leq e^{MN} \E[\Pr(R_N \not\in K\mid X)] 
\leq \exp(MN -2MN +o(N))$, which is summable in $N$. Hence we have $\limsup_{N\to\infty} 
\tfrac1N \log \Pr(R_N \not\in K \mid X) \leq -M$ a.s.\ by the Borel-Cantelli lemma. 
The same argument applies to $[R_N]_{\tr}$ (alternatively, use the fact that 
$[\cdot]_\tr$ is a continuous map). 

\paragraph{Step 2a.} 
We next verify that the limits in \eqref{eq:LambdaPhilimit1tr} exist. In Step 2a we consider 
the case $\mathrm{supp}(\rho)=[0,\infty)$, in Step 2b the case $\mathrm{supp}(\rho)=[s_*,\infty)$
with $s_*>0$.

Let $\tr \in \N$ and $h = 2^{-n}$. Let $Y^{(i,h)}=(\lceil T_i \rceil_h -\lceil T_{i-1} \rceil_h, 
(X_{(s+\lceil T_{i-1} \rceil_h) \wedge \lceil T_i \rceil_h}-X_{\lceil T_{i-1} \rceil_h})_{s\geq 0}) \in 
\widehat{F}$ be the $h$-discretised $i$-th word, and let
\begin{equation} 
\label{def:RNh}
R_{N,h} = \frac1N \sum_{i=0}^{N-1} 
\delta_{\widetilde{\theta}^i (Y^{(h)})^{N\text{-}\mathrm{per}}}
\end{equation}
be the $h$-discretised empirical process, where $Y^{(h)}=(Y^{(i,h)})_{i\in\N}$. Put 
$\ell = \ell_1 \vee \cdots \vee \ell_m$, $C_g=C_{g_1} \vee \cdots \vee C_{g_m}$. 
Let 
\begin{equation}
\label{eq:defDjh}
D_{j,h} = \sup_{(j-1)h \leq s \leq j h} |X_s-X_{j h}|,
\qquad  
A_{\varepsilon,k,h}(N) = \left\{ \sum_{i=1}^N \sum_{j=0}^k 
D_{\lceil T_i/h \rceil+j,h} \leq N \varepsilon \right\}. 
\end{equation}
By Lemma~\ref{obs:Rdiscdiff}, on the event $A_{\varepsilon,k,h}(N)$ we have 
\begin{equation} 
N \big| \Phi([R_N]_\tr) - \Phi([R_{N,h}]_\tr) \big| 
\leq 4\ell \|\Phi\|_\infty + N C_g \ell m \Big(2h+ \log{\textstyle\frac{k+1}{k}} 
+ 4\varepsilon \Big),  
\end{equation}
and hence
\begin{align} 
\label{eq:EeNPhiRNtr.ub1}
\E\big[ e^{N\Phi([R_N]_\tr)} \big| X \big] \leq & \, 
\exp\big[N C_g \ell m \big(2h+ \log{\textstyle\frac{k+1}{k}} 
+ 4\varepsilon \big) + 4\ell \|\Phi\|_\infty \big]\, 
\E\big[ e^{N\Phi([R_{N,h}]_\tr)} \big| X \big] \notag \\
& \, {} + e^{N \|\Phi\|_\infty} \Pr\big(A_{\varepsilon,k,h}(N)^c \mid X\big), 
\end{align}
For $\lambda >0$, estimate
\begin{align} 
\Pr([A_{\varepsilon,k,h}(N)]^c | X) 
\leq e^{-N \lambda \varepsilon} \, 
\E\left[ \exp\Big[ \lambda \sum_{i=1}^N \sum_{m=0}^k 
D_{\lceil T_i/h \rceil+m,h} \Big] \,\, \Big| \, X \right],
\end{align}
so that, by Lemma~\ref{lem:expmomentsDsum} in Step 4 below, 
\begin{align} 
\label{eq:limsuplogPrAN}
\limsup_{N\to\infty} \frac1N \log \Pr\big([A_{\varepsilon,k,h}(N)]^c \mid X\big) 
\leq -\varepsilon \lambda + \frac12 \log \chi\big(2 k \lambda \sqrt{h}\big). 
\end{align}
Since $\lim_{u\downarrow 0} \chi(u)= 1$, we have, for all $\varepsilon > 0$ and $k\in\N$, 
\begin{equation} 
\label{eq_Aepskh.unlikely}
\limsup_{h \downarrow 0} \limsup_{N\to\infty} \frac1N 
\log \Pr\big([A_{\varepsilon,k,h}(N)]^c \mid X\big) = - \infty \quad \text{a.s.}
\end{equation}
(pick $\lambda=\lambda(h)$ in (\ref{eq:limsuplogPrAN}) in such a way that 
$\lambda\to\infty$ and $\lambda\sqrt{h}\to 0$).
\smallskip

Next, observe that
\begin{align}
\E\big[ e^{N\Phi([R_N]_\tr)} \mid X \big] 
&= \int\cdots\int_{0<t_1<\cdots<t_N} \bar{\rho}(t_1) dt_1\, \bar{\rho}(t_2-t_1) dt_2  
\times\cdots\times \bar{\rho}(t_N-t_{N-1}) dt_N \notag \\
&\qquad \times \exp\big[ {N\Phi\big([R_{N;t_1,\dots,t_N}(X)]_\tr\big)} \big],\\[0.5ex] 
\label{eq:expPhiRNhtr}
\E\big[ e^{N\Phi([R_{N,h}]_\tr)} \mid X \big] 
&= \sum_{1 \leq j_1 \leq \cdots \leq j_N} w_h(j_1,\dots,j_N) 
\exp\big[ {N\Phi\big([R_{N;h j_1,\dots,h j_N}(X)]_\tr\big)} \big], 
\end{align} 
where 
\begin{equation}
\begin{aligned} 
\label{eq:wh.weights}
w_h(j_1,\dots,j_N) 
&= \int\cdots\int_{0<t_1<\cdots<t_N} \bar{\rho}(t_1) dt_1\, \bar{\rho}(t_2-t_1) dt_2  
\times\cdots\times \bar{\rho}(t_N-t_{N-1}) dt_N\\
&\qquad\qquad \times \prod_{k=1}^N \1_{(h(j_k-1), h j_k]}(t_k).
\end{aligned}
\end{equation} 
The idea is to replace the right-hand side of \eqref{eq:wh.weights} by $\prod_{k=1}^N 
\lceil \rho \rceil_h(h(j_k-j_{k-1}))$, which is the corresponding weight for a discrete-time 
renewal process with waiting time distribution $\lceil \rho \rceil_h$. The rigorous implementation 
of this idea requires some care, since the coarse graining can produce ``empty'' words. 

For $\underline{j}=(j_1,\dots,j_N)$ appearing in the sum in \eqref{eq:expPhiRNhtr}, let 
$R(\underline{j}) = \# \{ 1 \leq i \leq N \colon j_i = j_{i-1}\}$ be the total number of repeated 
values and $\underline{\hat{\jmath}}=(\hat \jmath_1,\dots,\hat \jmath_M)$ with $M=M(\underline{j})
=N-R(\underline{j})$, $1 \leq \hat \jmath_1 < \cdots < \hat \jmath_M$, the unique elements 
of $\underline{j}$. Note that any given $\underline{\hat{\jmath}}$ with $M=\lceil (1-\varepsilon)
N \rceil$ can be obtained in this way from at most ${N \choose \lceil \varepsilon N \rceil}$ 
different $\underline{j}$'s. 

In the following, we write $\eta(h)=\eta_n$ and $A(h) = A_n$ with $\eta_n$ and $A_n$ from 
\eqref{ass:rhobar.reg2} when $h=2^{-n}$. Let us parse through the right-hand side of 
\eqref{eq:wh.weights} successively for $k=N,N-1,\dots,1$. When $j_k=j_{k-1}$, we
integrate $t_k$ out over $(h(j_k-1), h j_k]$ and estimate the (multiplicative) contribution 
of this integral from above by $1$. When $j_k>j_{k-1}$, we replace $\bar{\rho}(t_k-t_{k-1})$ 
by $\bar{\rho}(t_k- h j_{k-1})$ and integrate $t_k$ out over $(h(j_k-1), h j_k]$. For $h(j_k-j_{k-1}) 
\in A(h)$ we can estimate the contribution of this integral from above by $e^{\eta(h)} \lceil \rho 
\rceil_h(h(j_k-j_{k-1}))$) by using \eqref{ass:rhobar.reg2}, while for $h(j_k-j_{k-1}) \not \in A(h)$ 
we can estimate it by $e^{\eta_0} \lceil \rho \rceil_h(h(j_k-j_{k-1}))$ by using \eqref{ass:rhobar.reg1} 
with $s_*=0$. Thus, for $\underline{j}$ with $R(\underline{j}) \leq \varepsilon N$ and 
$\#\{ 1 \leq i < N \colon h (j_i - j_{i-1}) \not\in A(h) \} \leq \varepsilon N$, we have
\begin{align} 
\label{eq:estwhbyrhohgr}
w_h(\underline{j}) \leq e^{\varepsilon \eta_0 N} e^{\eta(h) N}
\prod_{i=1}^{M} \lceil \rho \rceil_h \big(h(\hat \jmath_i- \hat \jmath_{i-1})\big) 
\end{align}
with $M=N-R(\underline{j})$.
Furthermore, 
\begin{align} 
\label{eq:estlast}
\Big| N\Phi\big([R_{N;h j_1,\dots,h j_N}(X)]_\tr\big) 
 - M\Phi\big([R_{M;h \hat \jmath_1,\dots,h \hat \jmath_M}(X)]_\tr\big) \Big| 
\le (N-M) \ell \| \Phi \|_\infty \le \varepsilon N \ell \| \Phi \|_\infty.
\end{align}
Combining (\ref{eq:expPhiRNhtr}--\ref{eq:estlast}), we find 
\begin{align} 
\label{eq:EeNPhiRNhtr.ub2}
&\E\big[e^{N\Phi([R_{N,h}]_\tr)} \mid X \big] \notag \\
&\leq e^{N \| \Phi \|_\infty} \Big\{ \Pr\Big( R(\lceil T_1 \rceil_h,\dots, \lceil T_N \rceil_h) 
\geq \varepsilon N \, \Big| \, X \Big) \notag \\ 
&\hspace{5em} {} + \Pr\Big( \#\big\{ 1 \leq i < N \colon \lceil T_i \rceil_h - \lceil T_{i-1} \rceil_h 
\not\in A(h) \big\} \geq \varepsilon N \, \Big| \, X \Big) \Big\} \notag \\
&\quad + e^{[\varepsilon \eta_0 + \eta(h)]N} {N \choose \varepsilon N} 
\sum_{M=\lceil (1-\varepsilon) N \rceil}^N \sum_{1 \leq \hat{\jmath}_1 < \cdots < \hat{\jmath}_M} 
e^{M \Phi\big([R_{M;h \hat{\jmath}_1,\dots,h \hat{\jmath}_M}(X)]_\tr\big)}
\prod_{k=1}^M \lceil \rho \rceil_h(h(\hat{\jmath}_k-\hat{\jmath}_{k-1})) . 
\end{align}
But
\begin{align} 
\sum_{1 \leq \hat{\jmath}_1 < \cdots < \hat{\jmath}_M} 
e^{M \Phi\big([R_{M;h \hat{\jmath}_1,\dots,h \hat{\jmath}_M}(X)]_\tr\big)}
\prod_{k=1}^M \lceil \rho \rceil_h(h(\hat{\jmath}_k-\hat{\jmath}_{k-1})) 
= \E_{\lceil \rho \rceil_h}\big[ & e^{M\Phi([R_M]_\tr)} \mid X \big],
\end{align}
where $\E_{\lceil \rho \rceil_h}$ denotes expectation w.r.t.\ the reference measure 
$Q_{\lceil \rho \rceil_h, \WM}$, and so we can apply Corollary~\ref{prop:qLDPhtr} 
and Varadhan's lemma to obtain 
\begin{align} 
\label{eq:limEeMPhi}
\lim_{M\to\infty} \frac1M \log 
\E_{\lceil \rho \rceil_h}\big[ & e^{M\Phi([R_M]_\tr)} \mid X \big]
= \sup_{Q \in \mathcal{P}^{\mathrm{inv, fin}}(\widetilde{E_h}^\N)} 
\big\{ \Phi(Q) - I^{\mathrm{que}}_{h,\tr}(Q) \big\}.
\end{align}

By elementary large deviation estimates for binomials we have, for any $\varepsilon>0$,
\begin{align} 
\label{eq:toomanywrongloops1}
&\limsup_{h \downarrow 0} \limsup_{N\to\infty} \frac1N \log 
\Pr\big( R(\lceil T_1 \rceil_h,\dots, \lceil T_N \rceil_h) 
\geq \varepsilon N \, \big| \, X \big) = - \infty,\\
\label{eq:toomanywrongloops2}
&\limsup_{h \downarrow 0} \limsup_{N\to\infty} \frac1N \log 
\Pr\Big( \#\big\{ 1 \leq i < N \colon\, \lceil T_i \rceil_h - \lceil T_{i-1} \rceil_h 
\not\in A(h) \big\} \geq \varepsilon N \, \Big| \, X \Big) = - \infty. 
\end{align}
(Note that the events in (\ref{eq:toomanywrongloops1}--\ref{eq:toomanywrongloops2}) 
are independent of $X$.) Combining \eqref{eq:EeNPhiRNtr.ub1}, \eqref{eq:EeNPhiRNhtr.ub2} 
and \eqref{eq:limEeMPhi}, and noting that $\lim_{N\to\infty} \frac1N \log {N \choose \varepsilon N}
= -\varepsilon\log\varepsilon - (1-\varepsilon)\log(1-\varepsilon)$, we find
\begin{equation} 
\label{eq:eNPhiRNasympt_upper0}
\begin{aligned} 
&\limsup_{N\to\infty}  \frac1N \log \E\big[ e^{N\Phi([R_N]_\tr)} \mid X \big] \\
&\qquad \leq \bigg\{ \sup_{Q \in \mathcal{P}^{\mathrm{inv, fin}}((\widetilde{E_h})^\N)} 
\big\{ \Phi(Q) - I^{\mathrm{que}}_{h,\tr}(Q) \big\} \\
&\qquad \qquad + C_g \ell m \big(2h+ \log{\textstyle\frac{k+1}{k}} 
+ 4\varepsilon \big) + \varepsilon \eta_0 + \eta(h) 
+ \varepsilon\log\tfrac {1}{\varepsilon} + (1-\varepsilon)\log\tfrac{1}{1-\varepsilon} \bigg\}\\ 
&\qquad \vee \bigg( \|\Phi\|_\infty 
+ \limsup_{N\to\infty} \frac1N \log \Pr\big(A_{\varepsilon,k,h}(N)^c \mid X\big)\bigg\} \\
&\qquad \vee \bigg\{ \|\Phi\|_\infty 
+  \limsup_{N\to\infty} \frac1N \log 
\Pr\Big( R(\lceil T_1 \rceil_h,\dots, \lceil T_N \rceil_h) 
\geq \varepsilon N \, \Big| \, X \Big) \bigg\} \\
&\qquad \vee \bigg\{ \|\Phi\|_\infty 
+  \limsup_{N\to\infty} \frac1N \log 
\Pr\Big( \#\{ 1 \leq i < N \colon \lceil T_i \rceil_h - \lceil T_{i-1} \rceil_h 
\not\in A(h) \} \geq \varepsilon N \, \Big| \, X \Big) \bigg\},
\end{aligned}
\end{equation}
and hence 
\begin{align} 
\label{eq:eNPhiRNasympt_upper}
\limsup_{N\to\infty} \frac1N \log \E\big[ e^{N\Phi([R_N]_\tr)} \big| X \big] 
\leq \liminf_{h\downarrow 0} 
\sup_{Q \in \mathcal{P}^{\mathrm{inv, fin}}(\widetilde{E_h}^\N)} 
\big\{ \Phi(Q) - I^{\mathrm{que}}_{h,\tr}(Q) \big\}
\end{align}
(let $h\downarrow 0$ along a suitable subsequence, followed by $\varepsilon
\downarrow 0$ and $k\to\infty$, and use \eqref{eq_Aepskh.unlikely} and
(\ref{eq:toomanywrongloops1}--\ref{eq:toomanywrongloops2})). 

Analogous arguments yield 
\begin{align} 
\label{eq:eNPhiRNasympt_lower}
\liminf_{N\to\infty} \frac1N \log \E\big[ e^{N\Phi([R_N]_\tr)} \mid X \big] 
\geq \limsup_{h\downarrow 0} 
\sup_{Q \in \mathcal{P}^{\mathrm{inv, fin}}((\widetilde{E_h})^\N)} 
\big\{ \Phi(Q) - I^{\mathrm{que}}_{h,\tr}(Q) \big\}.
\end{align}
Indeed, we can simply restrict the sum in \eqref{eq:expPhiRNhtr} to $\underline{j}$'s 
with $j_1 < \cdots < j_N$, so that the approximation argument is in fact a little easier 
because we need not pass to the $\underline{\hat\jmath}$'s. 

Finally, combine (\ref{eq:eNPhiRNasympt_upper}--\ref{eq:eNPhiRNasympt_lower}) to 
obtain \eqref{eq:LambdaPhilimit1tr}. 

\paragraph{Step 2b.} 
Next we consider the case $\mathrm{supp}(\rho)=[s_*,\infty)$ with $s_*>0$ and indicate 
the changes compared to Step 2a. To some extent this case is easier than the case $s_*=0$, 
since for coarse-graining level $h<s_*$ no ``empty'' word can appear in the coarse-graining 
scheme. On the other hand, when implementing a replacement similar to \eqref{eq:estwhbyrhohgr}, 
it can happen that an integral $\int \bar{\rho}(t_k-t_{k-1}) \1_{(h(j_k-1),h j_k]}(t) \,dt_k$ gets 
mapped to $\lceil \rho \rceil_h(h(j_k-j_{k-1}))=0$ even though the true contribution of that 
integral to \eqref{eq:expPhiRNhtr} is strictly positive (namely, when $h (j_k-j_{k-1}) \leq s_* \leq h(j_k-j_{k-1}+1)$).
The idea to remedy this problem is to replace $\lceil \rho \rceil_h(h(j_k-j_{k-1}))$ by a sum of 
``neighbouring'' weights of $\lceil \rho \rceil_h$ and to suitably control the overcounting incurred 
by this replacement. The details are as follows. 

Fix $h>0$ and $s_{*,h} = \lceil s_* \rceil_h$. For $N\in N$, consider $\underline{j}=(j_1,\dots,j_N)$ 
as appearing in the sum in \eqref{eq:expPhiRNhtr}. We say that $k \in \{1,\dots,N\}$ is ``problematic''
when $h(j_k-j_{k-1}) \in \{ s_{*,h}-1, s_{*,h}, s_{*,h}+1\}$, and ``relaxable'' when $j_k-j_{k-1} \geq 2$ and 
\begin{equation} 
\max_{m=-1,0,1} 
\left| \log \frac{\lceil \rho \rceil_h(h(j_k-j_{k-1}+m))}{\lceil \rho \rceil_h(h(j_k-j_{k-1}))} \right| \leq 2.
\end{equation}
Write $K_{\text{pro}}(\underline{j}) = \{ 1 \leq k \leq N \colon\, k\; \text{problematic}\}$ and $K_{\text{rel}}
(\underline{j}) = \{ 1 \leq k \leq N \colon\,k\; \text{relaxable}\}$. Try to construct an injection $f_{\text{rel},
\underline{j}}\colon\,K_{\text{pro}} \to K_{\text{rel}}$ with the property $f_{\text{rel},\underline{j}}(k) > k$ 
as follows: 
\begin{itemize}
\item[]
Start with an empty ``stack'' ${\sf s}$. For $k=1,\dots,N$ successively: when $k$ is problematic, 
push $k$ on ${\sf s}$; when $k$ is relaxable and ${\sf s}$ is not empty, pop the top element, say 
$k'$, from ${\sf s}$ and put $f_{\text{rel},\underline{j}}(k')=k$; when $k$ is neither problematic nor 
relaxable, proceed with the next $k$. 
\end{itemize}
We say that $\underline{j}$ is ``good'' when the above procedure terminates with an empty stack 
(in particular, $f_{\text{rel},\underline{j}}(k')$ is defined for all $k' \in K_{\text{pro}}$) and 
\begin{equation} 
\sum_{k \in K_{\text{pro}}} \big( f_{\text{rel},\underline{j}}(k) - k \big) \leq \varepsilon N
\end{equation}
(in particular, $\# K_{\text{pro}}(\underline{j}) \leq \varepsilon N$), and also $\# \{ 1 \leq k \leq N 
\colon\, j_k-j_{k-1} \not\in A(h)\} \leq \varepsilon N$.  For a given good $\underline{j}$, consider 
the set of all $\underline{\tilde\jmath}=(\tilde\jmath_1,\dots,\tilde\jmath_N)$ obtainable by setting 
\begin{align} 
\tilde\jmath_k=j_k+\Delta_k, \; \tilde\jmath_{f_{\text{rel},\underline{j}}(k)}
= j_{f_{\text{rel},\underline{j}}(k)}-\Delta_k \quad 
\text{with}\;\: \Delta_k \in \{-1,0,1\} \quad \text{for}\;\:k \in K_{\text{pro}},
\end{align}
and $\tilde\jmath_k=j_k$ for $k \not\in(K_{\text{pro}} \cup f_{\text{rel},\underline{j}}(K_{\text{pro}}))$. 
Note that a given good $\underline{j}$ corresponds to at most $3^{\varepsilon N}$ different 
$\underline{\tilde\jmath}$'s and that, for any such $\underline{\tilde\jmath}$,
\begin{align} 
&\Big| N\Phi\big([R_{N;h j_1,\dots,h j_N}(X)]_\tr\big) 
- N\Phi\big([R_{N;h \tilde \jmath_1,\dots,h \tilde \jmath_M}(X)]_\tr\big) \Big| \notag \\
&\qquad \leq \ell \| \Phi \|_\infty 
\sum_{k \in K_{\text{pro}}} \big( f_{\text{rel},\underline{j}}(k) - k \big) 
\leq \varepsilon N \ell \| \Phi \|_\infty .
\end{align}

With $w_h(j_1,\dots,j_N)$ defined in \eqref{eq:wh.weights}, we now see that (analogously to the 
argument prior to \eqref{eq:estwhbyrhohgr}) for any good $\underline{j}$, 
\begin{align} 
\label{eq:estwhbyrhohgr.c2}
w_h(\underline{j}) \leq e^{\varepsilon \eta_0 N} e^{\eta(h) N} 
2^{\varepsilon N} \sum_{\underline{\tilde\jmath} \; \text{corresp.\ to}\; \underline{j}} 
\; \prod_{i=1}^{N} \lceil \rho \rceil_h \big(h(\tilde \jmath_i - 
\tilde \jmath_{i-1})\big). 
\end{align}
Moreover, we have 
\begin{align} 
\label{eq:toomanywrongloops3}
\limsup_{h \downarrow 0} \limsup_{N\to\infty} \frac1N \log 
\Pr\big( (\lceil T_1 \rceil_h, \dots, \lceil T_N \rceil_h) 
\; \text{not good} \, \big| \, X \big) = - \infty. 
\end{align}
To check \eqref{eq:toomanywrongloops3}, let $S_k$ be the size of the stack ${\sf s}$ in 
the $k$-th step of the above construction when we use $j_k=\lceil T_k \rceil_h$, and note 
that $(\lceil T_1 \rceil_h, \dots, \lceil T_N \rceil_h)$ is good when $\sum_{k=1}^N S_k 
< \varepsilon N$. A comparison of $(S_k)_{k\in\N}$ with a (reflected) random walk on 
$\N_0$ that draws its steps from $\{0, \pm 1 \}$,  where $(+1)$-steps have a very small 
probability ($\leq \int_{s_*}^{s_*+2h} \bar{\rho}(t)\, dt$) and $(-1)$-steps have a very large 
probability ($\rho(A_h)$) when not from $0$, shows that $\limsup_{h \downarrow h} \frac1N 
\log \Pr(\sum_{k=1}^N S_k \geq \varepsilon N) = -\infty$ for every $\varepsilon >0$. We can 
then estimate similarly as in \eqref{eq:eNPhiRNasympt_upper0}, to obtain 
\eqref{eq:eNPhiRNasympt_upper} for the case $s_*>0$ as well. 

Analogous arguments also yield the lower bound in \eqref{eq:eNPhiRNasympt_lower}.

\paragraph{Step 3.} 
We next verify that the limits in \eqref{eq:LambdaPhilimit3} exist. Note that 
\begin{equation}
| \Phi(R_N) - \Phi([R_N]_\tr)| \leq \|\Phi\|_\infty 
\frac1N \#\big\{ \text{loops among the first $N$ loops that are longer than $\tr$} \big\},
\end{equation} 
which can be made arbitrarily small (also on the exponential scale, via a suitable 
annealing argument that uses that loop lengths are i.i.d.). A similar estimate holds for 
$| \Phi([R_N]_{\tr}) - \Phi([R_N]_{\tr'})|$ with $\tr < \tr'$. This shows that 
$\Lambda_{0,\tr}(\Phi)$ forms a Cauchy sequence as $\tr\to\infty$. 
\end{proof}

\begin{remark}
The arguments in Steps {\rm 2a} and {\rm 2b} can be combined to yield the same results when 
assumption \eqref{ass:rhobar.reg0} is relaxed to assumption \eqref{ass:rhobar.reg-mg}. 
Indeed, for a given coarse-graining level $h$, \eqref{ass:rhobar.reg-mg} gives rise to finitely 
many types of ``problematic points'' that can be handled similarly as in Step {\rm 2b} (combined 
with arguments from Step {\rm 2a} when $a_1=0$). 
\end{remark}

\paragraph{Step 4.} 
We close by deriving the estimate on Brownian increments over randomly
drawn short time intervals that was used in \eqref{eq:limsuplogPrAN}
in Step 2. The intuitive idea is that even though there are
arbitrarily large increments over short time intervals somewhere on
the Brownian path, it is extremely unlikely to hit these when sampling
along an independent renewal process. The proof employs a suitable
annealing argument.

Recall $D_{j,h}$ from (\ref{eq:defDjh}). For $h>0$ fixed, the $D_{j,h}$'s are i.i.d.\ 
and equal in law to $\sqrt{h}D_{1,1} = \sqrt{h} \sup_{0\leq s\leq 1} |X_s|$ by 
Brownian scaling. 

\begin{lemma} 
\label{lem:expmomentsDsum}
Let $T=(T_i)_{i\in\N}$ be a continuous-time renewal process with interarrival law 
$\rho$ satisfying $\mathrm{supp}(\rho) \subset [h,\infty)$. For $\lambda \geq 0$
and $k \in \N_0$, define
\begin{align} 
\xi(\lambda,h) = \limsup_{N\to\infty} \frac1N \log 
\E\left[ \exp\Big[ \lambda \sum_{i=1}^N \sum_{m=0}^k D_{\lceil T_i/h \rceil+m,h} \Big] \, 
\Big| \, \sigma(D_{j,h}, j \in \N)\right],
\end{align}
which is $\geq 0$ and a.s.\ constant by Kolmogorov's $0$-$1$-law. Then 
\begin{align} 
\lim_{h \downarrow 0} \xi(\lambda,h) = 0 \qquad \forall\,\lambda \geq 0. 
\end{align}
\end{lemma}

\begin{proof} 
We consider only the case $k=0$, the proof for $k\in\N$ being analogous. 
Abbreviate $\mathscr{G}_h=\sigma(D_{j,h}, j \in \N)$, and let 
\begin{equation} 
\chi(u) = \E\Big[ \exp \big[ u \, {\textstyle\sup_{\,0 \leq t \leq 1} |X_t|} \big] \Big], 
\quad u \in \R. 
\end{equation}
Note that $\chi(\cdot)$ is finite and satisfies $\lim_{u\to 0} \chi(u) = 1$.
We have 
\begin{equation}
\begin{aligned} 
\E\bigg[ \E\bigg[ \exp\Big[ \lambda \sum_{i=1}^N D_{\lceil T_i/h \rceil,h} 
\Big] \, \Big| \, \mathscr{G}_h \bigg]^2 \bigg] 
&\leq \E\bigg[ \exp\Big[ 2\lambda \sum_{i=1}^N D_{\lceil T_i/h \rceil,h}
\Big] \bigg]\\ 
&= \E\big[ \exp[2\lambda D_{1,h}] \big]^N = \chi\big(2\lambda \sqrt{h}\big)^N.
\end{aligned}
\end{equation}
Thus, for any $\epsilon>0$, 
\begin{equation}
\begin{aligned} 
& \Pr\left( \E\left[ \exp\Big[ \lambda \sum_{i=1}^N D_{\lceil T_i/h \rceil,h} 
\Big] \, \Big| \, \mathscr{G}_h \right]^2 
\geq \big( \chi\big(2\lambda \sqrt{h}\big) + \epsilon \big)^N \right) \\
& \leq \, \big( \chi\big(2\lambda \sqrt{h}\big) + \epsilon \big)^{-N} 
\E\left[ \E\left[ \exp\Big[ \lambda \sum_{i=1}^N D_{\lceil T_i/h \rceil,h}
\Big] \, \Big| \, \mathscr{G}_h \right]^2 \right] 
\leq \left( \frac{\chi\big(2\lambda \sqrt{h}\big)}{\chi\big(2\lambda \sqrt{h}\big) + \epsilon} \right)^N,
\end{aligned}
\end{equation}
which is summable in $N$. The Borel-Cantelli lemma therefore yields 
\begin{equation}
\limsup_{N\to\infty} \frac1N \log 
\E\left[ \exp\Big[ \lambda \sum_{i=1}^N D_{\lceil T_i/h \rceil,h} 
\Big] \, \Big| \, \mathscr{G}_h \right] 
\leq \frac12 \log \chi\big(2\lambda \sqrt{h}\big). 
\end{equation}
\end{proof}


\subsection{Proof of Proposition~\ref{prop:qLDPtrunc1}}
\label{ss:prop2}

\begin{lemma} 
\label{lemma:Iquetrregularised}
For $\tr \in \N$ and $Q \in \mathcal{P}^{\mathrm{inv}}(F_{0,\tr}^\N)$,
\begin{align} 
\label{eq:Iquetrregularised}
I^{\mathrm{que}}_\tr(Q) 
= \lim_{\varepsilon \downarrow 0} \, \limsup_{h \downarrow 0} \, 
\inf\Big\{ I^{\mathrm{que}}_{h,\tr}(Q')\colon\, Q' \in B_\varepsilon(Q) 
\cap \mathcal{P}^{\mathrm{inv}}((\widetilde{E}_{h,\tr})^\N) \Big\},
\end{align}
where $h \downarrow 0$ along $2^{-m}$, $m\in \N$.
\end{lemma}

\noindent 
Note that after $\widetilde{E}_{h,\tr}$ is identified with a subset of $F_{0,\tr}$ (see \eqref{iotahdef}), 
\eqref{eq:Iquetrregularised} states that $I^{\mathrm{que}}_{h,\tr}$ converges to $I^{\mathrm{que}}_\tr$ 
as $h \downarrow 0$ in the sense of Gamma-convergence. 

\begin{proof}[Proof of Lemma~\ref{lemma:Iquetrregularised}] 
Note that, when restricted to $\mathcal{P}^\mathrm{inv}(F_{0,\tr}^{\otimes\N})$, 
\begin{equation} 
\label{eq:PsiQcont1}
\text{both } Q \mapsto m_Q \text{ and } Q \mapsto \Psi_Q \text{ are continuous} 
\end{equation}
(by dominated convergence), while this is not true when $Q$ is allowed to vary over the 
whole of $\mathcal{P}^\mathrm{inv}(F^{\otimes\N})$. A more general statement is the 
following: if $\wlim_{n\to\infty} Q_n = Q$ and $\{ \mathscr{L}_{Q_n}(\tau_1)\colon\, 
n \in \N\}$ are uniformly integrable, then $\lim_{n\to\infty} m_{Q_n} = m_Q$ and 
$\wlim_{n\to\infty} \Psi_{Q_n} = \Psi_Q$.

In the proof we use several properties of specific relative entropy derived in Appendix~\ref{entropy}. 
Let $Q \in \mathcal{P}^{\mathrm{inv}}(F_{0,\tr}^\N)$, and abbreviate the right-hand side of 
\eqref{eq:Iquetrregularised} by $\widetilde{I}^\mathrm{que}_\tr(Q)$. Note that, by \eqref{eq:PsiQcont1} 
and the lower semi-continuity of $\Psi \mapsto H(\Psi \mid \WM)$, the map
\begin{equation}
\mathcal{P}^\mathrm{inv}(F_{0,\tr}^{\N}) \ni Q' \mapsto 
m_{Q'} H(\Psi_{Q'} \mid \WM)
\end{equation}
is lower semi-continuous. Hence, for any $\delta>0$, we have $m_{Q'} H(\Psi_{Q'} \mid \WM) \geq 
m_Q H(\Psi_Q \mid \WM) - \delta$ for all $Q' \in B_\varepsilon(Q) \cap  \mathcal{P}^{\mathrm{inv}}
(\widetilde{E}_{h,\tr}^\N)$ when $\varepsilon$ is sufficiently small (depending on $\delta$). Combine 
this with \eqref{eq:Hregularised1} in Lemma~\ref{lemma:hregularised1} in Appendix~\ref{entropy},  
and note that $\wlim Q_{h,\tr} = Q_\tr$ as $h\downarrow0$, to obtain $\widetilde{I}^\mathrm{que}_\tr(Q) 
\geq I^{\mathrm{que}}_{\tr}(Q)$.
\smallskip

For the reverse direction, we need to find $h_n>0$ with $\lim_{n\to\infty} h_n = 0$ 
and $Q'_n \in \mathcal{P}^{\mathrm{inv}}((\widetilde{E}_{h_n,\tr})^\N)$ with 
$\wlim_{n\to\infty} Q'_n = Q$ such that $\liminf_{n\to\infty} I^{\mathrm{que}}_{h_n,\tr}
(Q'_n) \le I^{\mathrm{que}}_\tr(Q)$. Here a complication stems from the fact that we 
must ensure that both parts of $I^{\mathrm{que}}_{h_n,\tr}(Q'_n)$, namely, $H(Q'_n \mid 
Q_{\lceil \rho \rceil_{h_n},\WM,\tr})$ and $H(\Psi_{Q'_n,h_n} \mid \WM)$, converge 
simultaneously. The proof is deferred to Lemma~\ref{lem:cg.2lev.blockapprox} in 
Appendix~\ref{entropy}. 
\end{proof}

We are now ready to give the proof of Proposition~\ref{prop:qLDPtrunc1}.

\begin{proof} 
Fix $\tr\in\N$. Denote the right-hand side of \eqref{eq:Lambda0tr} by $\tilde{\Lambda}_\tr(\Phi)$. 
Let $\Phi\colon\,\mathcal{P}^{\mathrm{inv}}(F^\N) \to \R$ be of the form (\ref{eq:Phiform1}). 
For every $\delta > 0$ we can find a $Q^* \in \mathcal{P}^{\mathrm{inv}}(F_{0,\tr}^\N)$ such 
that $\Phi(Q^*) - I_\tr^{\mathrm{que}}(Q^*) \geq \tilde{\Lambda}_\tr(\Phi) - \delta$. For $\varepsilon>0$ 
sufficiently small (depending on $\delta$) we have $\big| \Phi(Q') - \Phi(Q^*)\big| \leq \delta$ 
for all $Q' \in B_\varepsilon(Q^*)$ and, by Lemma~\ref{lemma:Iquetrregularised}, 
\begin{align} 
\liminf_{h\downarrow 0} \, \inf\Big\{ I^{\mathrm{que}}_{h,\tr}(Q')\colon\, 
Q' \in B_\varepsilon(Q^*) \cap \mathcal{P}^{\mathrm{inv}}((\widetilde{E_h}_{,\tr})^\N) \Big\} 
\leq I^{\mathrm{que}}_\tr(Q^*) + \delta.
\end{align}
Thus 
\begin{align} 
\liminf_{h\downarrow 0} \, 
\sup \Big\{ \Phi(Q') - I^{\mathrm{que}}_{h,\tr}(Q')\colon\, 
Q' \in B_\varepsilon(Q^*) \cap \mathcal{P}^{\mathrm{inv}}((\widetilde{E_h}_{,\tr})^\N) \Big\} 
\geq \tilde{\Lambda}_\tr(\Phi) - 3\delta. 
\end{align}
Let $\delta\downarrow 0$ to obtain $\liminf_{h\downarrow 0} \Lambda_{h,\tr}(\Phi) = \Lambda_{0,\tr}(\Phi) 
\geq \tilde{\Lambda}_\tr(\Phi)$. 

For the reverse direction, pick for $h \in (0,1)$ a maximiser $Q^*_h \in \mathcal{P}^{\mathrm{inv}}
((\widetilde{E_h}_{,\tr})^\N)$ of the variational expression appearing in the right-hand side of
\eqref{eq:Phihtrform}, i.e., $\Phi(Q^*_h) - I^{\mathrm{que}}_{h,\tr}(Q^*_h) = \Lambda_{h,\tr}(\Phi)$. 
This is possible because $\Phi-I^{\mathrm{que}}_{h,\tr}$ is upper semi-continuous and bounded 
from above, and $I^{\mathrm{que}}_{h,\tr}$ has compact level sets. We claim that 
\begin{equation} 
\label{claim:Q*htight} 
\text{the family} \; \big\{ Q^*_h : h \in (0,1) \big\} \subset 
\mathcal{P}^{\mathrm{inv}}(F^\N)\; \text{is tight}.
\end{equation}
Assuming (\ref{claim:Q*htight}), we can choose a sequence $h(n) \downarrow 0$ such that 
\begin{equation}
\begin{aligned}
&\lim_{n\to\infty} \Big[ \Phi(Q^*_{h(n)}) - I^{\mathrm{que}}_{h(n),\tr}(Q^*_{h(n)}) \Big] 
= \limsup_{h\downarrow 0} \Lambda_{h,\tr}(\Phi),\\ 
&\wlim_{n\to\infty} Q^*_{h(n)} = \widetilde{Q} 
\;\; \text{for some} \; \widetilde{Q} \in \mathcal{P}^{\mathrm{inv}}(F^\N).
\end{aligned}
\end{equation} 
Then $\lim_{n\to\infty}\Phi(Q^*_{h(n)})=\Phi(\widetilde{Q})$ because $\Phi$ is continuous, and 
$\liminf_{n\to\infty} I^{\mathrm{que}}_{h(n),\tr}(Q^*_{h(n)}) \geq I^{\mathrm{que}}_\tr(\widetilde{Q})$ 
by Lemma~\ref{lemma:Iquetrregularised}. Hence 
\begin{align} 
\Lambda_{0,\tr}(\Phi) = \limsup_{h\downarrow 0} \Lambda_{h,\tr}(\Phi) = \lim_{n\to\infty} \Big[ 
\Phi(Q^*_{h(n)}) - I^{\mathrm{que}}_{h(n),\tr}(Q^*_{h(n)}) \Big]
\leq \Phi(\widetilde{Q}) - I^{\mathrm{que}}_\tr(\widetilde{Q}) 
\leq \tilde{\Lambda}_\tr(\Phi).
\end{align}

It remains to prove \eqref{claim:Q*htight}, which follows once we show that for each 
$N \in \N$ the family of projections $\pi_N(Q^*_h) \in \mathcal{P}^{\mathrm{inv}}(F^N)$, 
$h\in(0,1)$, is tight (because $F^\N$ carries the product topology; see Ethier and 
Kurtz~\cite[Chapter 3, Proposition 2.4]{EK86}). Let $M= \|\Phi\|_\infty+1$. Then 
necessarily $H( Q^*_h \mid [Q_{\lceil \rho \rceil_h,\WM}]_\tr) \leq M$, and hence 
$h(\pi_N(Q^*_h) \mid \pi_N([Q_{\lceil \rho \rceil_h,\WM}]_\tr) \leq N M$ for all $h \in (0,1)$. 
Since $\pi_N([Q_{\lceil \rho \rceil_h,\WM}]_\tr) = ([q_{\lceil \rho \rceil_h,\WM}]_\tr)^{\otimes N}$
converges weakly to $\pi_N ([Q_{\rho,\WM}]_\tr) = ([q_{\rho,\WM}]_\tr)^{\otimes N}$ as 
$h \downarrow 0$, the family $\{\pi_N([Q_{\lceil \rho \rceil_h,\WM}]_\tr)\colon\,h \in (0,1)\}$ 
is tight, and so for any $\varepsilon > 0$ we can find a compact $\mathcal{C} \subset F^N$ 
such that $\pi_N([Q_{\lceil \rho \rceil_h,\WM}]_\tr)(\mathcal{C}^c) \leq \exp[-(NM + \log 2)/\varepsilon]$ 
uniformly in $h\in(0,1)$. By a standard entropy inequality (see \eqref{ineq:entropy} in
Appendix~\ref{entropy}), for all $h\in (0,1)$ we have
\begin{align} 
\pi_N(Q^*_h)(\mathcal{C}^c) \leq 
\frac{\log 2 + h\big(\pi_N(Q^*_h) \mid \pi_N([Q_{\lceil \rho \rceil_h,\WM}]_\tr)\big)}
{\log\Big(1+\big(\pi_N([Q_{\lceil \rho \rceil_h,\WM}]_\tr)(\mathcal{C}^c)\big)^{-1} \Big)} 
\leq \frac{\log 2 + MN}{\log\big(1+\exp[(N M + \log 2)/\varepsilon]\big)} 
\leq \varepsilon. 
\end{align}
This proves the representation \eqref{eq:Lambda0tr} of the limit $\Lambda_{0,\tr}(\Phi)$ 
from \eqref{eq:LambdaPhilimit1tr}. From \eqref{eq:LambdaPhilimit1tr} and \eqref{eq:Lambda0tr}, 
plus the exponential tightness in Proposition~\ref{prop:LambdaPhilimit1tr}, we obtain the 
LDP via Bryc's inverse of Varadhan's lemma.
\end{proof}


\subsection{Proof of Proposition~\ref{prop:Ique.tr.cont}}
\label{prop3}
 
 
\subsubsection{Proof of part (1)}

We first verify \eqref{eq:lemma:Ique.tr.cont1}, i.e., for $Q \in \mathcal{P}^{\mathrm{inv,fin}}(F^\N)$,
\begin{align} 
\label{eq:relentrsumlim}
\lim_{\tr\to\infty} I^{\mathrm{que}}_\tr([Q]_\tr) 
& = \lim_{\tr\to\infty} \Big[ 
H([Q]_\tr \mid [Q_{\rho,\WM}]_\tr) + 
(\alpha-1) m_{[Q]_\tr} H(\Psi_{[Q]_\tr} \mid \WM) \Big] \notag \\ 
& = H(Q \mid Q_{\rho,\WM}) + (\alpha-1) m_Q H(\Psi_Q \mid \WM).
\end{align}
The proof comes in 5 Steps.

\medskip\noindent
{\bf Step 1.} 
Note that $\lim_{\tr\to\infty} H([Q]_\tr \mid [Q_{\rho,\WM}]_\tr) = H(Q \mid Q_{\rho,\WM})$ 
by the projective property of word truncations, $\lim_{\tr\to\infty} m_{[Q]_\tr} = m_Q<\infty$ 
by dominated convergence, and 
\begin{equation}
\liminf_{\tr\to\infty} H(\Psi_{[Q]_\tr} \mid \WM) \geq H(\Psi_Q \mid \WM)
\end{equation}
by the lower semi-continuity of specific relative entropy together with $\wlim_{\tr\to\infty} 
\Psi_{[Q]_\tr} = \Psi_Q$. Hence, to obtain \eqref{eq:relentrsumlim} it remains to prove 
that
\begin{align} 
\label{ineq:HPsiQtr.upper}
\limsup_{\tr\to\infty} H(\Psi_{[Q]_\tr} \mid \WM) \leq H(\Psi_Q \mid \WM).
\end{align}

\medskip\noindent
{\bf Step 2.} \label{prop:Ique.tr.cont.part1step2}
To prove \eqref{ineq:HPsiQtr.upper}, we use coarse-graining. For every $h>0$ we 
can identify $\widetilde{E_h}$ with $F_h \subset F$ (recall \eqref{def:Fh}). In order 
to represent $Q\in\mathcal{P}^{\mathrm{inv,fin}}(F^\N)$ by a shift-invariant law on 
$(F_h)^\N$, we discretise the cut-points onto a \emph{uniformly shifted} grid of width 
$h$, as follows. For $t \in \R$, $h>0$ and $u\in [0,1)$, define (compare with 
Section~\ref{trun})
\begin{align} 
\label{eq:t.hu}
\lceil t \rceil_{h,u} = \min\big\{ (k+u)h \colon k \in \Z, (k+u)h \geq t \big\}
\quad \big(= \lceil t -uh \rceil_h + uh \big).
\end{align}
Draw $Y=(Y^{(i)})_{i\in\N}=((\tau_i, f_i))_{i\in\N}$ from law $Q$, and let $U$ be an independent 
random variable with uniform distribution on $[0,1]$. Put $T_0=0$, $T_n=\tau_1+\cdots+\tau_n$, 
$n \in \N$, 
\begin{equation}
\tilde{T}_i = \lceil T_i \rceil_{h,U}, \quad i \in \N_0,
\qquad \tilde{\tau}_i = \tilde{T}_i - \tilde{T}_{i-1}, \;
\tilde{f}_i = \big(\theta^{\tilde{T}_{i-1}} \kappa(Y)\big)(\, \cdot \wedge \tilde{\tau}_i), \quad i\in\N.
\end{equation} 
(Note that it may happen that $\tilde{\tau}_i=0$. We can remedy this by allowing ``empty words'', 
i.e., by formally passing to $\widehat{F}$ as in Section~\ref{subsect:notations}.) Write $\lceil Q 
\rceil_h$ for the distribution of $\tilde{Y}=(\tilde{Y}^{(i)})_{i \in \N}=((\tilde{\tau}_i, 
\tilde{f}_i))_{i \in \N}$ obtained in this way. We view $\lceil Q \rceil_h$ as an element of 
$\mathcal{P}^{\mathrm{inv,fin}}((F_h)^\N)$. To check the shift-invariance of $\lceil Q \rceil_h$, 
note that by construction an initial part of length $S_1=\tilde{T}_0-T_0 = U h$ of the content of 
the first word is removed (in a two-sided situation, this part would be added at the end of the 
zero-th word). The corresponding quantity for the second word is $S_2=\tilde{T}_1-T_1 = \lceil 
T_1 \rceil_{h,U} - T_1$. Observe that, for measurable $A \subset [0,h)$ and $B \subset [0,\infty)$, 
\begin{equation}
\Pr(S_2 \in A, T_1 \in B) = \int_B \Pr(T_1 \in dt) \int_{[0,1]} du\,
\1_A\big( \lceil t -uh \rceil_h - (t-uh) \big) = \frac1h\, \Pr(T_1 \in B)\, \lambda(A), 
\end{equation}
i.e., $S_2$ is distributed as $U h$ and independent of $Y$, and so $(\tilde{Y}^{(i+1)})_{i\in\N}$ 
again has law $\lceil Q \rceil_h$. This settles the shift-invariance. The key feature of the 
construction of $\lceil Q \rceil_h$ is that $\kappa(\tilde{Y}) = (\theta^{U h} \kappa)(Y)$, so that 
\begin{equation}
\label{Psihrel}
\Psi_{\lceil Q \rceil_h, h} = \Psi_Q,
\end{equation} 
and therefore
\begin{align} 
\label{HhHrel}
H(\Psi_{\lceil Q \rceil_h, h} \mid \WM) = H(\Psi_Q \mid \WM).
\end{align} 
Thus, \eqref{HhHrel} gives us a coarse-grained version of the right-hand of \eqref{ineq:HPsiQtr.upper}.

\medskip\noindent
{\bf Step 3.}
If $\tr$ is an integer multiple of $h$, then the coarse-graining $\lceil Q \rceil_h \in 
\mathcal{P}^{\mathrm{inv,fin}}((F_h)^\N)$ of $Q \in \mathcal{P}^{\mathrm{inv,fin}}(F^\N)$ 
defined in Step 2 commutes with the word length truncation $[ \cdot ]_\tr$, i.e., $[ \lceil Q 
\rceil_h ]_\tr = \lceil [ Q ]_\tr \rceil_h$. This is a deterministic property of the construction in 
\eqref{eq:t.hu}. Indeed, fix $u \in [0,1)$ and $h$ with $\tr = Mh$ for some $M\in \N$, consider 
$t_{i-1}<t_i$ with $t_i-t_{i-1} > \tr$ (so that in the un-coarse-grained truncation procedure 
the  $i$-th loop length would be replaced by $\tr$), let $k_{i-1}, k_i \in \N$ be such 
that $\lceil t_{i-1} \rceil_{h,u} = (k_{i-1}+u)h$ and $\lceil t_i \rceil_{h,u} = (k_i+u)h$. 
When we first truncate and then coarse-grain, the $i$-th point becomes $\lceil t_{i-1} 
+ \tr \rceil_{h,u} = (k_{i-1}+M+u)h$. When we first coarse-grain and then truncate, 
the $i$-th point becomes $\lceil t_{i-1} \rceil_{h,u} + \big( (\lceil t_i \rceil_{h,u} 
- \lceil t_{i-1} \rceil_{h,u}) \wedge M h \big) = (k_{i-1}+u)h + M h$, which is the same.

\medskip\noindent
{\bf Step 4.} 
Let $h=2^{-M}$, define $\lceil Q \rceil_h \in \mathcal{P}^{\mathrm{inv,fin}}((F_h)^\N)$ as 
in Step 2, and write $Q'_h = \lceil Q \rceil_h \circ \iota_h^{-1}$ for the same object 
considered as an element of $\mathcal{P}^{\mathrm{inv,fin}}((\widetilde{E_h})^\N)$ 
(recall (\ref{def:Eh}--\ref{iotahdef})). Write $\nu_h=\mathscr{L}\big((X_{\cdot \wedge h})\big)$ 
for the Wiener measure on $E_h$. Then $m_{Q'_h} = m_{\lceil Q \rceil_h}/h$ (the mean 
word length counted in $h$-letters), while
\begin{align} 
\label{HPsihrel}
H(\Psi_{Q'_h} \mid \nu_h^{\otimes \N}) = H(\Psi_{\lceil Q \rceil_h,h} \mid \WM),
\end{align}
by construction,  and 
\begin{align}
\label{Qtrhrel}
\lceil [ Q ]_\tr \rceil_h = [ \lceil Q \rceil_h ]_\tr = [Q'_h]_{(\tr/h)} \circ \iota_h,
\end{align} 
where the first equality follows from the commutation property in Step 3 and the 
second equality is a truncation of the words from $Q'_h$ as elements of $\widetilde{E_h}$. 

\medskip\noindent
{\bf Step 5.}
Fix $\varepsilon>0$ and let $\tr_0 = \tr_0(Q,\varepsilon)$ 
be so large that 
\begin{align} 
\E_Q\big[ \big( |Y^{(1)}|-\tr \big)_+ \big] < \tfrac13 \varepsilon m_Q, 
\qquad \tr \geq \tr_0.
\end{align}
Then, for $0<h<\tfrac{1}{24} \varepsilon m_Q$, we have
\begin{align}
\label{eq:justso} 
\E_{\lceil Q \rceil_h}\big[ h\big( \tfrac{|Y^{(1)}|}{h}-\tfrac{\tr}{h} \big)_+ \big] 
< \tfrac13 \varepsilon  m_Q + 2h < \tfrac12 \varepsilon m_{\lceil Q \rceil_h}. 
\end{align}
Divide both sides of \eqref{eq:justso} by $h$, and observe that the continuum word of length 
$|Y^{(1)}|$ under $\lceil Q \rceil_h$ corresponds to the discrete word of $|Y^{(1)}|/h$ 
$h$-letters under $Q_h'$, to obtain 
\begin{align} 
\E_{Q_h'}\big[ \big( |Y^{(1)}|-\tfrac{\tr}{h} \big)_+ \big]
< \tfrac12 \varepsilon m_{Q_h'}.
\end{align}
This estimate allows us to use Lemma~\ref{lem:trcontinuous} in Appendix~\ref{entropy}, 
which says that for every $0<\varepsilon<\tfrac12$,
\begin{align} 
\label{bepsest}
(1-\varepsilon) \Big[ H(\Psi_{[Q_h']_{(\tr/h)}} \mid \nu_h^{\otimes \N}) 
+ b(\varepsilon) \Big] \leq H(\Psi_{Q_h'} \mid \nu_h^{\otimes \N})
\end{align}
with $b(\varepsilon)= - 2\varepsilon + [\varepsilon \log \varepsilon + (1-\varepsilon) 
\log (1-\varepsilon)]/(1-\varepsilon)$. However, by (\ref{HPsihrel}--\ref{Qtrhrel}) we have
\begin{equation}
H(\Psi_{[Q'_h]_{(\tr/h)}} \mid \nu_h^{\otimes \N}) 
= H(\Psi_{\lceil [Q]_\tr \rceil_h,h} \mid \WM) = H(\Psi_{[Q]_\tr} \mid \WM).
\end{equation} 
Substitute this relation into \eqref{bepsest} and use (\ref{HhHrel}--\ref{HPsihrel}), to obtain 
\begin{equation}
(1-\varepsilon) \Big[ H(\Psi_{[Q]_\tr} \mid \WM)
+ b(\varepsilon) \Big] \leq H(\Psi_Q \mid \WM).
\end{equation}
Now let $\varepsilon \downarrow 0$ and use that $\lim_{\varepsilon\downarrow 0} 
b(\varepsilon)=0$, to obtain \eqref{ineq:HPsiQtr.upper}.


\subsubsection{Proof of part (2)}
\label{subsect:prop:Ique.tr.cont.part2}

Fix $Q \in \mathcal{P}^{\mathrm{inv}}(F^\N)$ with $m_Q= \infty$ and $H(Q \mid Q_{\rho,\WM}) 
< \infty$. We construct $\widetilde{Q}_\tr \in \mathcal{P}^{\mathrm{inv,fin}}(F^\N)$, $\tr\in\N$, 
satisfying \eqref{eq:lemma:Ique.tr.approx2} via a ``smoothed truncation'' that has the same 
concatenated word content as its ``hard truncation'' equivalent. The proof comes in 5 Steps.
 
\medskip\noindent
{\bf Step 1.}
It will we be convenient to consider the two-sided scenario, i.e., we regard $Q$ as a shift-invariant 
probability measure on $F^\Z$. Define
\begin{equation} 
\chi_\tr\colon\,F_{0,\tr}^\Z \times [0,1]^\Z \to F^\Z,
\qquad 
\chi_\tr\colon\,\big( (f_i,\tau_i)_{i\in\Z}, (u_i)_{i\in\Z} \big) \mapsto (\tilde f_i, \tilde \tau_i)_{i\in\Z},
\end{equation} 
as follows. Put $t_0=0$, $t_i=t_{i-1}+\tau_i$, $t_{-i}=t_{-i+1}-\tau_{-i+1}$  for $i\in\N$, and $\varphi 
= \kappa\big( (f_i,\tau_i)_{i\in\Z}\big)$, set
\begin{align} 
\tilde{t}_i = 
\begin{cases} t_i-u_i & \text{if} \; \tau_i=\tr,\\
t_i & \text{if} \; \tau_i <\tr,
\end{cases}
\end{align}
$\tilde\tau_i=t_i-t_{i-1}$ and $\tilde f_i(\cdot)=\varphi( (\,\cdot \wedge \tilde\tau_i)+t_{i-1})$ 
for $i\in\Z$. In words, the total concatenated word content remains unchanged, and if the length 
of the $i$-th word $\tau_i$ equals $\tr$, then its end-point $t_i$ is moved $u_i$ to the left. Put 
$\widetilde{Q}_\tr = ([Q]_\tr \otimes \mathrm{Unif}[0,1]^{\otimes \Z}) \circ \chi_\tr^{-1} \in 
\mathcal{P}^{\mathrm{inv}}(F^\Z)$. By construction, $\Psi_{\widetilde{Q}_\tr} = \Psi_{[Q]_\tr}$ 
and $m_{\widetilde{Q}_\tr} = m_{[Q]_\tr}$. In particular, 
\begin{align} 
m_{\widetilde{Q}_\tr} H(\Psi_{\widetilde{Q}_\tr} \mid \WM) 
= m_{[Q]_\tr} H(\Psi_{[Q]_\tr} \mid \WM) .
\end{align}

\medskip\noindent
{\bf Step 2.}
Write $\widetilde{Q}_\tr^{\mathrm{ref}} = ([q_{\rho,\WM}]_\tr^{\otimes\Z} \otimes \mathrm{Unif}
[0,1]^{\otimes \Z}) \circ \chi_\tr^{-1}$ for the result of the analogous operation on the reference 
measure $(q_{\rho,\WM})^{\otimes \Z}$. We have $\wlim_{\tr\to\infty} \widetilde{Q}_\tr = Q$ and 
$\wlim_{\tr\to\infty} ([q_{\rho,\WM}]_\tr^{\otimes\Z} \otimes \mathrm{Unif}[0,1]^{\otimes \Z}) \circ 
\chi_\tr^{-1} = (q_{\rho,\WM})^{\otimes \Z}$, and hence
\begin{align} 
\label{eq:string}
\liminf_{\tr\to\infty} H( \widetilde{Q}_\tr \mid \widetilde{Q}_\tr^{\mathrm{ref}} ) 
& \geq \sup_{\varepsilon > 0} \liminf_{\tr\to\infty} 
\inf_{Q' \in B_\varepsilon(Q)} H( Q' \mid \widetilde{Q}_\tr^{\mathrm{ref}} ) \notag \\
& \geq H(Q \mid (q_{\rho,\WM})^{\otimes \Z}) 
= \lim_{\tr\to\infty} H([Q]_\tr \mid [q_{\rho,\WM}]_\tr^{\otimes \Z}),
\end{align}
where we use Lemma~\ref{lemma:hregularised1}~(2) in the second inequality. (Note: Inspection 
of the proof of Lemma~\ref{lemma:hregularised1}~(2) shows that the inequality ``$\leq$'' in 
\eqref{eq:Hregularised1} also holds for $Q$'s that are not product.) The last equality in \eqref{eq:string}
holds because the truncations $[\,\cdot\,]_\tr$ form a projective family (see \cite[Lemma~8.1]{BiGrdHo10}).
As specific relative entropy can only decrease under the operation of taking image measures, we 
have  $H( \widetilde{Q}_\tr \mid \widetilde{Q}_\tr^{\mathrm{ref}} ) \leq H([Q]_\tr \mid [q_{\rho,
\WM}]_\tr^{\otimes \Z}) \leq H(Q \mid q_{\rho,\WM}^{\otimes \Z})$, so $\limsup_{\tr\to\infty} 
H( \widetilde{Q}_\tr \mid \widetilde{Q}_\tr^{\mathrm{ref}} ) \leq H(Q \mid q_{\rho,\WM}^{\otimes \Z})$ 
and, indeed, 
\begin{align} 
\lim_{\tr\to\infty} H( \widetilde{Q}_\tr \mid \widetilde{Q}_\tr^{\mathrm{ref}} ) 
= \lim_{\tr\to\infty} H([Q]_\tr \mid [q_{\rho,\WM}]_\tr^{\otimes \Z}) 
= H(Q \mid q_{\rho,\WM}^{\otimes \Z}).
\end{align}
The proof of \eqref{eq:lemma:Ique.tr.approx2} is complete once we show that
\begin{align} 
\label{eq:HtildeQtr.bd}
H( \widetilde{Q}_\tr \mid q_{\rho,\WM}^{\otimes \Z}) \leq 
H( \widetilde{Q}_\tr \mid \widetilde{Q}_\tr^{\mathrm{ref}} ) + o(1),
\end{align}
since, by part (1), 
\begin{align} 
\widetilde{I}^{\mathrm{que}}(\widetilde{Q}_\tr) = 
H( \widetilde{Q}_\tr \mid q_{\rho,\WM}^{\otimes \Z}) + 
m_{\widetilde{Q}_\tr} H(\Psi_{\widetilde{Q}_\tr} \mid \WM).
\end{align}

\medskip\noindent
{\bf Step 3.} 
It remains to verify \eqref{eq:HtildeQtr.bd}. Note that 
\begin{align} 
H( \widetilde{Q}_\tr \mid q_{\rho,\WM}^{\otimes \Z}) 
- H( \widetilde{Q}_\tr \mid \widetilde{Q}_\tr^{\mathrm{ref}} ) 
& = \lim_{N\to\infty} \frac1N \E_{\widetilde{Q}_\tr}\bigg[ 
\log\frac{d\pi_N \widetilde{Q}_\tr}{d q_{\rho,\WM}^{\otimes N}} 
- \log\frac{d\pi_N \widetilde{Q}_\tr}{d\pi_N \widetilde{Q}_\tr^{\mathrm{ref}}} \bigg] 
\notag \\
& = \lim_{N\to\infty} \frac1N \E_{\widetilde{Q}_\tr}
\bigg[ \log\frac{d\pi_N \widetilde{Q}_\tr^{\mathrm{ref}}}{d q_{\rho,\WM}^{\otimes N}} \bigg],
\end{align}
and that, by construction, ${d\pi_N \widetilde{Q}_\tr^{\mathrm{ref}}}/{dq_{\rho,\WM}^{\otimes N}}$ 
is a function of the word lengths $\tilde{\tau}_1,\dots,\tilde{\tau}_N$ only (indeed, because of the 
i.i.d.\ property of Brownian increments it easy to see that under both laws the word contents 
given their lengths are the same, namely, independent pieces of Brownian paths). Write 
$\widetilde{R}_\tr^{\mathrm{ref}}$ for the law of the sequence of word lengths under 
$\widetilde{Q}_\tr^{\mathrm{ref}}$. Then we must show that 
\begin{align} 
\label{eq:ElogdRtildedrho.bd}
\limsup_{\tr\to\infty} 
\lim_{N\to\infty} \frac1N \E_{\widetilde{Q}_\tr}
\bigg[ \log\frac{d\pi_N \widetilde{R}_\tr^{\mathrm{ref}}}{d \rho^{\otimes N}}
(\tilde\tau_1,\dots,\tilde\tau_N) \bigg] \leq 0. 
\end{align}

\medskip\noindent
{\bf Step 4.} 
Denote the density of $\pi_N \widetilde{R}_\tr^{\mathrm{ref}}$ with respect 
to Lebesgue measure on $\R_+^N$ by $\widetilde{f}_{\tr,N}^{\mathrm{ref}}$. 
Consider fixed $\tilde\tau_1,\dots,\tilde\tau_N$, and decompose into 
maximal stretches of $\tilde\tau_i$'s with values in $(\tr-1,\tr+1)$ 
(note that under $\chi_\tr$ no word can become longer than $\tr+1$, 
while when $\tilde\tau_i < \tr-1$ the corresponding word is not truncated, i.e., 
$\tilde{t}_i=t_i$). Thus, there are $0 \leq M < N$, $i'_1 \leq j'_2 < i'_2 \leq j'_2 
< \cdots < i'_M \leq j'_M \leq N$ such that $\{ 1 \leq i \leq N\colon\,\tilde\tau_i 
\in (\tr-1,\tr+1) \} = \cup_{k=1}^M [i'_k, j'_k] \cap \N$. Observe that, by construction, 
$\widetilde{f}_{\tr,N}^{\mathrm{ref}}(\tilde\tau_1,\dots,\tilde\tau_N)$ can be 
decomposed into a product of $\prod_{j \colon \tilde\tau_j\leq \tr-1} \bar{\rho}
(\tilde\tau_j)$ and $M$ further factors involving the $\tilde\tau_i$'s from these 
stretches, where the $k$-th factor depends only on $(\tilde\tau_i\colon\,
i'_k \leq i \leq j'_k)$. We claim that 
\begin{align} 
\label{eq:ftrN.ref.bd}
\frac{\widetilde{f}_{\tr,N}^{\mathrm{ref}}(\tilde\tau_1,\dots,\tilde\tau_N)}{
\prod_{j=1}^N \bar{\rho}(\tilde\tau_j)} 
\leq \prod_{k=1}^M \big( C_1 \tr^{1+\epsilon} \big)^{j'_k-i'_k+1} 
= \big( C_1 \tr^{1+\epsilon} \big)^{\# \{ 1 \leq i \leq N\colon\, \tilde\tau_i > \tr-1 \}}
\end{align}
for some $C_1=C_1(\rho) <\infty$ and $\epsilon=\epsilon(\rho) \in [0,1]$ uniformly 
in $\tr$ for $\tr$ sufficiently large. To see why \eqref{eq:ftrN.ref.bd} holds, consider for 
example the first stretch and assume for simplicity that $i'_1=1<j'_1$ and that we 
know that the $0$-th word is not truncated (i.e., $\tilde{t}_0=t_0=0$). Let $\ell 
\leq j'_1+1$, and pretend we know that the first $\ell-1$ words are truncated (i.e., 
$\tau_1=\cdots=\tau_{\ell-1}=\tr$), while the $\ell$-th word is not ($\tau_\ell<\tr$). 
Then $\tilde\tau_1=\tr-u_1$ and $\tilde\tau_i=\tr-u_i+u_{i-1}$ for $2 \leq i \leq \ell-1$, 
and so $u_i=\sum_{j=1}^i (\tr-\tilde\tau_j)$ for $1 \leq i \leq \ell-1$ and $\tau_\ell
=\tilde\tau_\ell-u_{\ell-1} = \tilde\tau_\ell-\sum_{j=1}^{\ell-1} (\tr-\tilde\tau_j)$. This 
case contributes to $\widetilde{f}_{\tr,\ell}^{\mathrm{ref}}$ the term 
\begin{align} 
\label{eq:termf.ell.ref} 
\rho([\tr,\infty))^{\ell-1} \bar\rho\Big(\tilde\tau_\ell-
{\textstyle \sum_{j=1}^{\ell-1} (\tr-\tilde\tau_j)}\Big) \prod_{i=1}^{\ell-1} 
\1_{[0,1]}\Big( {\textstyle \sum_{j=1}^{i} (\tr-\tilde\tau_j)}\Big).
\end{align}
Note that, by \eqref{ass:rhodensdecay}, we have $\eqref{eq:termf.ell.ref}/\prod_{j=1}^\ell 
\bar{\rho}(\tilde\tau_j) \le C_2 (C_3 \tr^{1+\epsilon})^{\ell-1}$ for some $C_2=C_2(\rho), 
C_3=C_3(\rho) <\infty$ and $\epsilon=\epsilon(\rho) \in [0,1]$ uniformly in $\tr$ for $\tr$
sufficiently large. The contribution of any given stretch of length $j'_k-i'_k+1$ can be 
written as a sum of at most $2^{j'_k-i'_k+1}$ cases where the indices of the truncated 
words are specified. Each such case can be estimated by a suitable product of 
terms as in \eqref{eq:termf.ell.ref}. Furthermore, outside the stretches the words are 
necessarily untruncated and thus contribute $\bar{\rho}(\tilde\tau_i)$ to 
$\widetilde{f}_{\tr,N}^{\mathrm{ref}}$, which cancels with the corresponding 
term in $\rho^{\otimes N}$. 

\medskip\noindent
{\bf Step 5.}
From \eqref{eq:ftrN.ref.bd} and the shift-invariance of $\widetilde{Q}_\tr$ we obtain that 
\begin{align} 
\label{eq:ElogdRtildedrho.bd2}
\lim_{N\to\infty} \frac1N \E_{\widetilde{Q}_\tr}
\bigg[ \log\frac{d\pi_N \widetilde{R}_\tr^{\mathrm{ref}}}{d \rho^{\otimes N}}
(\tilde\tau_1,\dots,\tilde\tau_N) \bigg] \leq 
C(1+\log \tr) Q(\tau_1>\tr-1). 
\end{align}
Now, $h(\mathscr{L}_Q(\tau_1) \mid \rho) \leq H(Q \mid q_{\rho,\WM}^\N) < \infty$ by 
assumption. Because of \eqref{ass:rhodensdecay}, this implies that $\E_Q[\log(\tau_1)] 
< \infty$,  and hence that $Q(\tau_1>\tr) = o(1/\log \tr)$. Therefore \eqref{eq:ElogdRtildedrho.bd2} 
implies \eqref{eq:ElogdRtildedrho.bd}. 

\hfill $\qed$


\section{Removal of Assumptions~\eqref{ass:rhobar.reg2}--\eqref{ass:rhobar.reg1}}
\label{removeass}

We give a brief sketch of the proof only, leaving the details to the reader. Assumptions
\eqref{ass:rhobar.reg2}--\eqref{ass:rhobar.reg1} are satisfied when $\bar{\rho}$ satisfies 
\eqref{ass:rhobar.reg0} and varies regularly at $\infty$ with index $\alpha$. The latter condition 
is stronger than \eqref{ass:rhodensdecay}. To prove the claim under \eqref{ass:rhodensdecay} 
alone, note that for every $\delta>0$ and $\alpha'<\alpha$ there exists a probability density $\bar{\rho}'=\bar{\rho}'
(\delta,\alpha')$ such that $\bar{\rho} \leq (1+\delta)\bar{\rho}'$, $\bar{\rho}'$ varies regularly at 
$\infty$ with index $\alpha'$, and $\bar{\rho}'(t)dt$ converges weakly to $\bar{\rho}(t)dt$ as $\delta
\downarrow 0$ and $\alpha' \uparrow \alpha$. Since the quenched LDP holds for $\bar{\rho}'$, 
we can proceed similarly as in \cite[Sections 3.6 and 5]{BiGrdHo10} to get the quenched LDP 
for $\bar{\rho}$. 

More precisely, for $B \subset \mathcal{P}^\mathrm{inv}(F^\N)$ we may write (recall
\eqref{RNdef} and \eqref{eq:defRNphi})
\begin{align}
P(R_N \in B \mid X) &= \int_{0 \leq t_1 < \cdots < t_N < \infty} dt_1 \cdots dt_N\,
\bar{\rho}(t_1)\,\bar{\rho}(t_2-t_1) \cdots \bar{\rho}(t_N-t_{N-1})\\[-2ex]
&\hspace{18em} \times 1_B\big(R_{N;t_1,\ldots,t_N}(X)\big), \notag
\end{align}
and estimate $\bar{\rho}(t_1) \leq (1+\delta)\bar{\rho}'(t_1)$, etc., to get $P(R_N \in B \mid X) 
\leq (1+\delta)^N\,P'(R_N \in B \mid X)$, where $P,P'$ have $\bar{\rho},\bar{\rho}'$ as excursion 
length distributions. Let $\mathcal{C} \subset \mathcal{P}^\mathrm{inv}(F^\N)$ be a closed set, and let
$\mathcal{C}^{(\varepsilon)}$ be its $\varepsilon$-blow-up. Then the LDP upper bound for $\bar{\rho}'$
gives
\begin{equation}
\limsup_{N\to\infty} \frac{1}{N} \log P(R_N \in \mathcal{C}^{(\varepsilon)}\mid X) 
\leq \log (1+\delta) - \inf_{Q \in \mathcal{C}^{(\varepsilon)}} I^\mathrm{que}_{\bar{\rho}'}(Q) 
\qquad X\text{-a.s.},
\end{equation} 
where the lower index $\bar{\rho}'$ indicates the excursion length distribution. Let $\delta \downarrow 0$
and $\alpha' \uparrow \alpha$, and use Lemma~\ref{lemma:hregularised1}~(2), to get 
\begin{equation}
\limsup_{N\to\infty} \frac{1}{N} \log P(R_N \in \mathcal{C}^{(\varepsilon)}\mid X) 
\leq - \inf_{Q \in \mathcal{C}^{(2\varepsilon)}} I^\mathrm{que}_{\bar{\rho}}(Q) \qquad X\text{-a.s.}
\end{equation} 
Finally, let $\varepsilon \downarrow 0$ and use the lower semi-continuity of $I^\mathrm{que}_{\bar{\rho}}$
to get the LDP upper bound for $\bar{\rho}$.
\smallskip

An analogous argument works for the LDP lower bound: 
Now we pick $\alpha' > \alpha$, $\delta > 0$ and  a probability density 
$\bar{\rho}'=\bar{\rho}' (\delta,\alpha')$ such that 
$\bar{\rho} \geq (1-\delta)\bar{\rho}'$, and $\bar{\rho}'$ satisfies the same 
conditions as above. Arguing as before, we obtain for any 
open $\mathcal{O} \subset \mathcal{P}^\mathrm{inv}(F^\N)$,
\begin{equation}
\liminf_{N\to\infty} \frac{1}{N} \log P(R_N \in \mathcal{C}^{(\varepsilon)}\mid X) 
\geq - \inf_{Q \in \mathcal{O}} I^\mathrm{que}_{\bar{\rho}}(Q) \qquad X\text{-a.s.}
\end{equation}


\section{Proof of Theorems~\ref{mainthmboundarycases}--\ref{thmexp}}
\label{proofalpha1infty}

We again give a brief sketch of the proofs only, leaving 
many details to the reader. 
\smallskip

Theorem~\ref{mainthmboundarycases}(a), which says that for $\alpha=1$ the 
quenched rate function coincides with the annealed rate function, can 
be proved as follows: Since the claimed LDP upper bound holds automatically 
by the annealed LDP, it suffices to verify the matching lower bound.
For this we can argue as in the proof of the lower bound in 
Section~\ref{removeass}. 
For any $\alpha'>1$ and $\delta>0$ we can approximate $\bar{\rho}$ by 
a suitable $\bar{\rho}'=\bar{\rho}' (\delta,\alpha')$ such that 
$\bar{\rho} \geq (1-\delta)\bar{\rho}'$. Then, using Theorem~\ref{thm0:contqLDP} 
with $\bar{\rho}'$ and taking $\delta \downarrow 0$, $\alpha' \downarrow 1$, 
we see that for any open $\mathcal{O} \subset \mathcal{P}^\mathrm{inv}(F^\N)$, 
\begin{equation} 
\liminf_{N\to\infty} \frac{1}{N} \log P(R_N \in \mathcal{C}^{(\varepsilon)}\mid X) 
\geq - \inf_{Q \in \mathcal{O} \cap \mathcal{P}^{\mathrm{inv,fin}}(F^\N)} 
I^\mathrm{ann}(Q) \qquad X\text{-a.s.}
\end{equation} 
(recall \eqref{eq:Ique}). Finally note that any 
$Q \in \mathcal{P}^{\mathrm{inv}}(F^\N)$ with $H(Q \mid Q_{\rho,\WM}) < \infty$ 
can be approximated by a sequence $(Q_n) \subset \mathcal{P}^{\mathrm{inv,fin}}(F^\N)$ 
in such a way that $H(Q_n \mid Q_{\rho,\WM}) \to H(Q \mid Q_{\rho,\WM})$ 
to obtain the claim (using for example a ``smoothed truncation'' operation 
similar to Section~\ref{subsect:prop:Ique.tr.cont.part2}). 
\smallskip

Theorem~\ref{mainthmboundarycases}(b), which says that for $\alpha=\infty$ the
quenched rate function coincides with the annealed rate function on the set 
$\{Q\in\mathcal{P}^{\mathrm{inv}}(F^\N)\colon\,\lim_{\tr\to\infty} m_{[Q]_\tr} 
H( \Psi_{[Q]_\tr} \mid \WM) = 0\}$ and is infinite elsewhere, follows from
arguments analogous to \cite[Section 7, Part (b)]{BiGrdHo10}: 
For the upper bound, we can pick arbitrarily large $\alpha'>1$ and approximate 
$\bar{\rho} \leq (1+\delta) \bar{\rho}'$ with the help of a suitable probability 
density $\bar{\rho}'$ which has decay exponent $\alpha'$. Using 
Theorem~\ref{thm0:contqLDP} with $\bar{\rho}'$ and taking 
$\alpha' \uparrow \infty$, $\delta \downarrow 0$, we see that the upper bound 
holds with the claimed form of the rate function.
For the matching lower bound we can trace through the proof of the 
lower bound contained in Theorem~\ref{thm0:contqLDP} but replacing our 
``coarse-graining work horses'' Proposition~\ref{thm00:contqLDP} 
and Corollary~\ref{prop:qLDPhtr} (which rely on \cite[Cor.~1.6]{BiGrdHo10}) 
by versions that are suitable for $\alpha=\infty$ (which rely on 
\cite[Thm.~1.4~(b)]{BiGrdHo10} instead), still using a suitable 
truncation approximation of the quenched rate function analogous to the 
one proven in Proposition~\ref{prop:Ique.tr.cont}. This constitutes a way 
of rigorously implementing the ``first long string strategy'' from 
\cite[Section 4]{BiGrdHo10}, as explained in the heuristic given in item 0 of 
Section~\ref{disc}, through the coarse-graining approximation. 
\smallskip

Theorem~\ref{thmexp} follows from Theorem~\ref{mainthmboundarycases}(b) via 
an observation that is the analogue of \cite[Lemma 6]{Bi08}: subject to the exponential 
tail condition in \eqref{ass:rhoexp}, any $Q\in\mathcal{P}^{\mathrm{inv}}(F^\N)$ with 
$H(Q \mid Q_{\rho,\WM}) <\infty$ necessarily has $m_Q<\infty$. Because of this 
observation we can argue as follows. If $m_Q<\infty$, then $\lim_{\tr\to\infty} m_{[Q]_\tr} 
= m_Q$ and $\lim_{\tr\to\infty} \Psi_{[Q]_\tr} = \Psi_Q$ by dominated convergence
(recall \eqref{eq:PsiQcont}), which in turn imply that $\liminf_{\tr\to\infty} m_{[Q]_\tr} 
H(\Psi_{[Q]_\tr} \mid \WM) = m_Q H(\Psi_Q \mid \WM)$, as shown in Lemma~\ref{lem:trcontinuous} 
in Appendix~\ref{entropy}. The limit is zero if and only if $\Psi_Q = \WM$, which by \eqref{eq:calRchar} 
holds if and only if  $Q \in \cR_\WM$. This explains the link between \eqref{eq:ratefctalphainfty} 
and \eqref{eq:ratefctexptail}.


\appendix


\section{Basic facts about metrics on path space}
\label{metrics}

We metrise $F$, defined in \eqref{eq:defF} (and $F_h \subset F$ defined in \eqref{def:Fh}) as follows. 
Let $d_S(\phi_1, \phi_2)$ be a metric on $C([0,\infty))$ that generates Skorohod's $J_1$-topology 
on $D([0,\infty)) \supset C([0,\infty))$, allowing for a certain amount of ``rubber time'' (see e.g.\ 
Ethier and Kurtz~\cite[Section 3.5 and Eqs.\ (5.1--5.3)]{EK86}) 
\begin{equation}
\label{def:dS}
d_S(\phi_1, \phi_2) = \inf_{\lambda \in \Lambda} 
\bigg\{ \gamma(\lambda) \vee \int\nolimits_0^\infty 
e^{-u} \sup_{t\geq 0} \big| \phi_1(t \wedge u) - \phi_2(\lambda(t) \wedge u) 
\big| \, du \bigg\}, 
\end{equation} 
where $\Lambda$ is the set of Lipschitz-continuous bijections from $[0,\infty)$ into 
itself and 
\begin{equation} 
\gamma(\lambda) = \sup_{0 \leq s < t < \infty} 
\Big| \log \frac{\lambda(t)-\lambda(s)}{t-s} \Big|.
\end{equation}
With 
\begin{equation} 
\label{eq:metriconF}
d_F(y_1,y_2) = |t_1-t_2| + d_S(\phi_1,\phi_2)
\end{equation}
for $y_i=(t_i,\phi_i) \in F$, $(F,d_F)$ becomes complete and separable, and the same holds 
for $(F_h,d_F)$ for any $h>0$.
\medskip

\noindent {\bf Remark.\ } 
We might at first be inclined to metrise $F$ in a more straightforward way than 
(\ref{eq:metriconF}), e.g.\ via
\begin{equation} 
\label{eq:firstmetriconF}
d^{\mathrm{first}}_F(y_1,y_2) = |t_1-t_2| + \|\phi_1-\phi_2\|_\infty, 
\quad y_i=(t_i,\phi_i) \in F, \: i=1,2. 
\end{equation}
However, if we would use Lipschitz functions with $d_F$ replaced by $d^{\mathrm{first}}_F$ 
in (\ref{eq:g_Lipschitz}), then in the analogue of Lemma~\ref{obs:Rdiscdiff} we would be 
forced to use terms of the form $\sup_{s \geq 0} |\varphi(s+t \wedge t') - \varphi(s+ih 
\wedge jh)|$ in the right-hand side. When used for $\varphi=X$ (a realisation of Brownian 
motion as in Proposition~\ref{prop:LambdaPhilimit1tr}), this would in turn force us 
to control the increments of the Brownian motion not only locally near the beginning 
and the end of each loop, but uniformly inside loops. In fact, it seems plausible that 
an analogue of Proposition~\ref{prop:LambdaPhilimit1tr} where $d_F$ is replaced by 
$d^{\mathrm{first}}_F$ actually fails. Furthermore, note that we cannot arrange $d_S$ 
in such a way that, for $\phi \in C([0,\infty))$, $h>0$, $t_1 \leq t'_1 < t_2 \leq t_2'$ 
with $|t'_1-t_1| \leq h$, $|t'_2-t_2| \leq h$, 
\begin{align} 
\label{eq:dS_wishful1}
d_S\big( \phi((t_1+\cdot) \wedge t_2), 
\phi((t'_1+\cdot) \wedge t'_2)\big) \leq 2h 
+ \sup_{t_1 \leq s \leq t'_1} |\phi(s)-\phi(t'_1)| 
+ \sup_{t_2 \leq s \leq t'_2} |\phi(s)-\phi(t'_2)|.
\end{align}
This is why in Lemma~\ref{obs:Rdiscdiff} we need the freedom to use an extra $k$ 
and to ``look in a neighbourhood of the cut-points of size $kh$''.


\section{Basic facts about specific relative entropy}
\label{entropy}

In Section~\ref{ss:definitions} we recall the definition of (specific) relative entropy of two 
probability measures. In Section~\ref{ss:approximations} we prove various approximation 
results for (specific) relative entropy, which were used heavily in Sections~\ref{props}.
Especially the parts with $\Psi_Q$ require care because of the delicate nature of the
word concatenation map $Q \mapsto \Psi_Q$. The latter is looked at in closer detail
in Section~\ref{subs:towards.Ique.tr.cont}.


\subsection{Definitions}
\label{ss:definitions}

For $\mu,\nu$ probability measures on a measurable space $(S,\mathscr{S})$,
\begin{equation}
h(\mu \mid \nu) 
= 
\begin{cases} 
\int_S (\log\frac{d\mu}{d\nu})\,d\mu, 
&\text{if} \; \mu \ll \nu, \\
\infty, 
& \text{otherwise,}
\end{cases}
\end{equation}
is the relative entropy of $\mu$ w.r.t.\ $\nu$. When the measurable space is a 
Polish space $E$ equipped with its Borel-$\sigma$-algebra, we also have the 
representation (see e.g.\ \cite[Lemma 6.2.13]{DeZe98})
\begin{align} 
\label{eq:relentraslegendretransf}
h(\mu \mid \nu) 
= \sup_{f \in C_b(E)} \Big\{ \int f\,d\mu - \log \int e^f \, d\nu \Big\} 
= \sup_{\scriptstyle f\colon\, E \to \R \; \atop \scriptstyle \text{bounded measurable}} 
\Big\{ \int f\,d\mu - \log \int e^f \, d\nu \Big\} 
\end{align}
(and if $\mu \ll \nu$ with a bounded and uniformly positive density, then
the supremum in the right-hand side is achieved by $f=\log d\mu/d\nu$).

Equation \eqref{eq:relentraslegendretransf} implies the entropy inequality 
\begin{equation} 
\label{ineq:entropy}
\mu(A)\leq \frac{\log 2 + h(\mu \mid \nu)}{\log[1+1/\nu(A)]}
\end{equation} 
by choosing $f=\alpha \1_A$ and $\alpha=\log[1+1/\nu(A)]$ (see e.g.\ Kipnis and 
Landim~\cite[Appendix 1, Proposition 8.2]{KiLa99}).

For $Q \in \mathcal{P}^{\mathrm{inv}}(F^\N)$, 
\begin{align} 
\label{eq:SREwrtProd}
H(Q \mid (q_{\rho,\WM})^{\otimes\N}) 
= \lim_{N\to\infty} \frac1N h\big( \pi_N Q \mid (q_{\rho,\WM})^{\otimes N}\big)
= \sup_{N\in\N} \frac1N h\big(\pi_N Q \mid (q_{\rho,\WM})^{\otimes N}\big)
\end{align}
with $\pi_N$ the projection onto the first $N$ words, is the specific relative entropy of 
$Q$ w.r.t.\ $(q_{\rho,\WM})^{\otimes\N}$. Similarly, using the canonical filtration 
$(\mathscr{F}^C_t)_{t \ge 0}$ on $C([0,\infty))$, for a probability measure $\Psi$ on 
$C([0,\infty))$ with stationary increments we denote by 
\begin{align} 
\label{eq:SREwrtWM}
H(\Psi \mid \WM) = \lim_{t\to\infty} 
\frac{1}{t} h\big(\Psi_{|}{}_{\mathscr{F}^C_t} \mid \WM_{|}{}_{\mathscr{F}^C_t}\big)
= \sup_{t>0} \frac{1}{t} h\big(\Psi_{|}{}_{\mathscr{F}^C_t} \mid \WM_{|}{}_{\mathscr{F}^C_t}\big)
\end{align}
the specific relative entropy w.r.t.\ Wiener measure. See Appendix~\ref{contrelentr} for a
proof of \eqref{eq:SREwrtWM}.


\subsection{Approximations}
\label{ss:approximations}

Let $E$ be a Polish space. Equip $\mathcal{P}(E)$ with the weak topology (suitably metrised). 
$E^\N$ carries the product topology, and the set of shift-invariant probability measures 
$\mathcal{P}^\mathrm{inv}(E^\N)$ carries the weak topology (also suitably metrised).


\subsubsection{Blocks}

For $M\in\N$ and $r \in \mathcal{P}(E^M)$, denote by $r^{\otimes \N} \in \mathcal{P}(E^\N)$ the 
law of an infinite sequence obtained by concatenating $M$-blocks drawn independently from 
$r$ (i.e., we identify $(E^M)^\N$ and $E^\N$ in the obvious way), and write 
\begin{equation} 
\label{def:blockmeas}
\block_M(r) = \frac1M \sum_{j=0}^{M-1} r^{\otimes \N} \circ (\theta^j)^{-1} 
\; \in \mathcal{P}^\mathrm{inv}(E^\N)
\end{equation} 
for its stationary mean. 

\begin{lemma}
\label{lemma:HblockmeasQ}
For $Q = q^{\otimes \N} \in\mathcal{P}^\mathrm{inv}(E)$ and $r \in \mathcal{P}(E^M)$, 
\begin{equation} 
\label{eq:HblockmeasQ}
H\big(\block_M(r) \mid Q \big) = \frac1M h\big(r \mid \pi_M Q \big). 
\end{equation}
Moreover, for any $R \in \mathcal{P}^\mathrm{inv}(E)$, 
\begin{equation}
\label{eq:reconstrfromblocks}
\wlim_{M\to\infty} \block_M\big(\pi_M R\big) = R.
\end{equation}
\end{lemma}

\begin{proof} 
This proof is standard. Equation \eqref{eq:HblockmeasQ} follows from the 
results in Gray \cite[Section 8.4, see Theorem 8.4.1]{Gr09b} by observing that 
$\block_M(r)$ is the asymptotically mean stationary measure of $r^{\otimes\N}$.
It is also contained in F\"ollmer\cite[Lemma 4.8]{Foe88},  or can be proved with 
``bare hands'' by explicitly spelling out $d\pi_N \block_M(r)/dq^{\otimes N}$ for 
$N \gg M$ and using suitable concentration arguments under $q^{\otimes N}$ 
as $N\to\infty$. Equation \eqref{eq:reconstrfromblocks} is obvious from the definition 
of weak convergence. 
\end{proof}


\subsubsection{Change of reference measure}

\begin{lemma} 
\label{lemma:hregularised1}
{\rm (1)} 
Let $\nu, \nu_1,\nu_2,\ldots \in \mathcal{P}(E)$ with $\wlim_{n\to\infty} \nu_n = \nu$. 
Then 
\begin{align}
\label{eq:hregularised1}
h(\mu \mid \nu) = \lim_{\varepsilon \downarrow 0} 
\limsup_{n \to \infty} \inf_{\mu' \in B_\varepsilon(\mu)} h(\mu' \mid \nu_n), 
\quad \mu \in \mathcal{P}(E). 
\end{align}
{\rm (2)} 
Let $Q=q^{\otimes \N}, Q_1=q_1^{\otimes \N},Q_2=q_2^{\otimes \N},\ldots \in 
\mathcal{P}^\mathrm{inv}(E^\N)$ be product measures with $\wlim_{n\to\infty} 
Q_n$ $= Q$. Then 
\begin{align}
\label{eq:Hregularised1}
H(R \mid Q) = \lim_{\varepsilon \downarrow 0} 
\limsup_{n \to \infty} \inf_{R' \in B_\varepsilon(R)} H(R' \mid Q_n), 
\quad R \in \mathcal{P}^\mathrm{inv}(E^\N). 
\end{align}
\end{lemma}
 
\begin{proof} 
(1) 
Denote the term in the right-hand side of (\ref{eq:hregularised1}) by $\tilde{h}(\mu)$. 
Let $f \in C_b(E)$, $\delta > 0$. We can find $\varepsilon_0 > 0$ and $n_0 \in \N$ such 
that 
\begin{align} 
\forall \, 0 < \varepsilon \leq \varepsilon_0, \,
\mu' \in B_\varepsilon(\mu)\colon\,\, 
&\int_E f\, d\mu' \geq \int_E f\, d\mu - \frac{\delta}{2}, \\
\forall \, n \geq n_0\colon\,\,
&\log \int_E e^f \, d\nu_n \leq \log \int_E e^f \, d\nu + \frac{\delta}{2}. 
\end{align}
Therefore, for $0 < \varepsilon \leq \varepsilon_0$ and $n \geq n_0$,
\begin{align} 
\inf_{\mu' \in B_\varepsilon(\mu)} h(\mu' \mid \nu_n) 
\geq \int_E f\, d\mu' - \log \int_E e^f \, d\nu_n 
\geq \int_E f\, d\mu \smallskip - \log \int_E e^f \, d\nu - \delta.
\end{align}
Now optimise over $f$ and take $\delta \downarrow 0$, to obtain $\tilde{h}(\mu) 
\geq h(\mu \mid \nu)$ via \eqref{eq:relentraslegendretransf}.

For the reverse inequality, we may without loss of generality assume that $h(\mu \mid \nu) 
= \int_E \varphi \log \varphi\, d\nu < \infty$, where $\varphi=d\mu/d\nu \geq 0$ is in 
$L^1(\nu)$. Then for any $\delta > 0$ we can find a $\widetilde{\varphi} \geq 0$ 
in $C_b(E) \cap L^1(\nu)$ such that $\int_E \widetilde{\varphi}\,d\nu = 1$ and 
\begin{align} 
\int_E \big| \widetilde{\varphi} - \varphi \big| \, d\nu < \delta, \;\; 
\int_E \big| \widetilde{\varphi}\log\widetilde{\varphi} - 
\varphi\log\varphi \big| \, d\nu < \delta.
\end{align}
Note that $\lim_{n\to\infty} \int_E \widetilde{\varphi}\,d\nu_n = 1$, and let 
$\widetilde{\varphi}_n = \widetilde{\varphi}/\int \widetilde{\varphi}\,d\nu_n$ 
and $\mu_n = \widetilde{\varphi}_n \nu_n$. Then, for $g \in C_b(E)$, 
\begin{equation}
\begin{aligned} 
\Big| \int_E g \, d\mu_n - \int_E g \,d\mu \Big| 
&= \Big| \frac1{\int_E \widetilde{\varphi} \, d\nu_n} 
\int_E g \widetilde{\varphi}\, d\nu_n - \int_E g\varphi \,d\nu \Big| \\
&\leq \Big| \frac{1}{\int_E \widetilde{\varphi}\, d\nu_n} - 1 \Big| 
\, \| g \widetilde{\varphi} \|_\infty 
+ \Big| \int_E g \widetilde{\varphi}\, d\nu_n - \int_E g \widetilde{\varphi}\, d\nu \Big| 
+ \Big| \int_E g (\widetilde{\varphi} - \varphi) \,d\nu\Big|,
\end{aligned}
\end{equation}
which can be made arbitrarily small by choosing $\delta$ small enough and $n$ large enough. 
In particular, for any $\varepsilon > 0$ we can choose $\delta$, $\widetilde{\varphi}$ and 
$n_0$ such that $\mu_n \in B_\varepsilon(\mu)$ for $n \geq n_0$. Hence 
\begin{equation}
\begin{aligned} 
\limsup_{n\to\infty} \inf_{\mu' \in B_\varepsilon(\mu)} h(\mu' \mid \nu_n) 
&\leq \limsup_{n\to\infty} h(\mu_n \mid \nu_n)\\ 
&= \limsup_{n\to\infty} \int_E \widetilde{\varphi}_n \log \widetilde{\varphi}_n \, d\nu_n 
= \int_E \widetilde{\varphi} \log \widetilde{\varphi} \, d\nu \leq h(\mu \mid \nu) + \delta, 
\end{aligned}
\end{equation}
and letting $\delta \downarrow 0$ we  $\tilde{h}(\mu) \leq h(\mu \mid \nu)$.

\medskip\noindent 
(2) Recall that for $R \in \mathcal{P}^\mathrm{inv}(E^\N)$ and $Q$ a product measure, 
\begin{equation}
\lim_{N\to\infty} \frac1N h\left( \pi_N R \mid \pi_N Q\right) 
= H(R \mid Q) = \sup_{N\in\N} \frac1N h\left( \pi_N R \mid \pi_N Q\right). 
\end{equation}
Denote the expression in the right-hand side of (\ref{eq:Hregularised1}) by $\tilde{H}(R)$. 
Fix $N\in\N$. Since for each $\varepsilon>0$, we have $B_{\varepsilon'}(R) \subset 
\{ R'\colon\, \pi_N R' \in B_{\varepsilon}(\pi_N R) \}$ for $\varepsilon'$ sufficiently 
small we also have 
\begin{align} 
\lim_{\varepsilon' \downarrow 0} \limsup_{n\to\infty} 
\inf_{R' \in B_{\varepsilon'}(R)} H(R' \mid Q_n) 
\ge \limsup_{n \to \infty} \inf_{\mu' \in B_\varepsilon(\pi_N R)} 
\frac1N h(\mu' \mid \pi_N Q_n). 
\end{align}
Let $\varepsilon \downarrow 0$ and use Part (1), to see that $\tilde{H}(R) \ge \frac1N 
h(\pi_N R \mid \pi_N Q)$ for any $N$. Hence also $\tilde{H}(R) \geq H(R \mid Q)$. 

For the reverse inequality, we may w.l.o.g.\ assume that $H(R \mid Q) < \infty$. 
Fix $\varepsilon>0$ and $\delta>0$. There is an $N \in \N$ such that 
$H(R \mid Q) \leq \frac1N h \big( \pi_N R \mid \pi_N Q \big) + \delta$, and since 
$\pi_N R \ll \pi_N Q=q^ {\otimes N}$ we can find a continuous, bounded and 
uniformly positive function $f_N\colon\,E^N \to [0,\infty)$ such that $\int_E f_N \, 
dq^{\otimes N} = 1$, $\int_E f_N \log f_N \, dq^{\otimes N} \leq h \big( \pi_N R \mid 
\pi_N Q \big) + N\delta$ and $\tilde{R}_N  \in B_{\varepsilon/2}(R)$, where $\tilde{R}_N 
= \block_N\big(f_N\,q^{\otimes N}\big) \in \mathcal{P}^\mathrm{inv}(E^\N)$ (see 
Lemma~\ref{lemma:HblockmeasQ}). By \eqref{eq:HblockmeasQ}, we have
\begin{equation} 
H(\tilde{R}_N \mid Q) 
= \frac1N \int_E f_N \log f_N \, dq^{\otimes N}.
\end{equation}
Now put $f_{N,n} = f_N/\int_E f_N\,q_n^{\otimes N}$, and define $\tilde{R}_{N,n} = \block_N
\big(f_{N,n} \, q^{\otimes N}\big)$ as the ``stationary version'' of $(f_{N,n}\,q_n^{\otimes 
N})^{\otimes \N}$. In particular, $H(\tilde{R}_{N,n} \mid Q_n) = \frac1N \int f_{N,n} \log 
f_{N,n} \, dq_n^{\otimes N}$. Since $f_N$ is continuous, we have $\tilde{R}_{N,n} \in 
B_{\varepsilon}(R)$ and $\int_E f_{N,n} \log f_{N,n}\,dq_n^{\otimes N} \le H(R \mid Q) 
+ 3\delta$ for $n$ large enough. Hence 
\begin{equation}
\limsup_{n\to\infty} 
\inf_{R' \in B_{\varepsilon}(R)} H(R' \mid Q_n) 
\le \limsup_{n\to\infty} H(\tilde{R}_{N,n} \mid Q_n) 
\le H(R \mid Q) + 4\delta.
\end{equation}
Now let $\delta \downarrow 0$ followed by $\lim_{\varepsilon\downarrow 0}$ to conclude
the proof. 
\end{proof}


\subsubsection{Existence of sharp coarse-graining approximations to 
the quenched rate function}

The following lemma was used in the proof of Lemma~\ref{lemma:Iquetrregularised}. 

\begin{lemma}
\label{lem:cg.2lev.blockapprox}
Let $Q \in \mathcal{P}^{\mathrm{fin}}(F^\N)$ with $H( Q \mid Q_{\rho,\WM}) < \infty$. There 
exist a sequence $(h_n)_{n\in\N}$ with $h_n>0$ and $\lim_{n\to\infty} h_n = 0$ and a sequence 
$(Q'_n)_{n\in\N}$ with $Q'_n \in \mathcal{P}^{\mathrm{fin}}(\widetilde{E}_{h_n}^\N)$ and 
$\wlim_{n\to\infty} Q'_n = Q$ such that $\limsup_{n\to\infty} I^{\mathrm{que}}_{h_n}(Q'_n) 
\leq I^{\mathrm{que}}(Q)$. The same holds with $F$ replaced by $F_{0,\tr}$ and 
$\widetilde{E}_{h_n}$ replaced by $\widetilde{E}_{h_n,\tr}$.
\end{lemma}

\begin{proof} 
Recall the definition of $\lceil Q \rceil_h$ in Step\ 2 of the proof 
of part\ (1) of Proposition~\ref{prop:Ique.tr.cont} 
(see page~\pageref{prop:Ique.tr.cont.part1step2}). 
For any $N\in\N$, we have
\begin{align}
\label{eq:anyway1}
h(\pi_N \lceil Q \rceil_h \mid \pi_N \lceil Q_{\rho,\WM} \rceil_h) 
\leq h(\pi_N Q \mid \pi_N Q_{\rho,\WM}) \leq N \, H( Q \mid Q_{\rho,\WM}).
\end{align} 
The second inequality follows from \eqref{eq:SREwrtProd}. For the first inequality, use the fact that 
the construction of $\lceil Q \rceil_h$ can be implemented as a deterministic function of the pair of
random variables $(Y,U)$, together with the fact that relative entropy can only decrease when image 
measures are taken. Write $\hat{\tau}_i = (\tilde{T}_i-\tilde{T}_{i-1})/h$, $i \in \N$. Since letters both 
under $\lceil Q_{\rho,\WM} \rceil_h$ and under $Q_{\lceil \rho \rceil_h,\WM}$ are constructed from 
a Brownian path that is independent of the word lengths, we have (recall \ref{def:rho.h.trunc})
\begin{align}
\1(\hat{\tau}_1=k_1,\dots,\hat{\tau}_N=k_N) \,
\frac{d\pi_N \lceil Q_{\rho,\WM} \rceil_h}{d\pi_N Q_{\lceil \rho \rceil_h,\WM}} 
= \frac{(\pi_N \lceil Q_{\rho,\WM} \rceil_h)\big(\hat{\tau}_1=k_1,\dots,\hat{\tau}_N=k_N\big)}
{\prod_{\ell=1}^N \lceil \rho \rceil_h(h k_\ell)}
\end{align}
with 
\begin{align} 
& (\pi_N \lceil Q_{\rho,\WM} \rceil_h)\big(\hat{\tau}_1=k_1,\dots,\hat{\tau}_N=k_N\big) \notag \\
& = \int_{[0,1]} du \, 
\int_0^\infty \bar{\rho}(t_1) dt_1
\int_{t_1}^\infty \bar{\rho}(t_2-t_1) d(t_2-t_1) \cdots
\int_{t_{N-1}}^\infty \bar{\rho}((t_N-t_{N-1})) d(t_N-t_{N-1}) \notag\\
&\qquad\qquad \times \prod_{\ell=1}^N \1_{(h(\bar{k}_\ell-1+u), h(\bar{k}_\ell+u)]}(t_\ell),
\end{align}
where $\bar{k}_\ell = k_1+\cdots+k_\ell$. Thus, by (\ref{eq:Vbarrhodef}--\ref{ass:rhobar.reg1}), 
\begin{align} 
\label{eq:anyway2}
\sup_{N \in \N} \frac1N E_{\lceil Q \rceil_h}\left[ \Big| \log 
\frac{d\pi_N \lceil Q_{\rho,\WM} \rceil_h}{d\pi_N Q_{\lceil \rho \rceil_h,\WM}}\Big|\right] \leq r_Q(h)
\end{align}
with
\begin{align}
r_Q(h) = \eta_n \lceil Q \rceil_h(\hat\tau_1 \in \bar{A}_n) 
+ \eta_0 \lceil Q \rceil_h(\hat\tau_1 \not\in \bar{A}_n), \qquad h=2^{-n},
\end{align}
where $\bar{A}_n \subset (s_*,\infty)$ is the set obtained from $A_n$  by removing pieces of 
length $2^{-n}$ from its edges (i.e., $\bar{A}_n$ is the $2^{-n}$-interior of $A_n$). But 
$\lim_{n\to\infty} \lceil Q \rceil_{2^{-n}}(\hat\tau_1 \not\in \bar{A}_n)=0$ because $A_n$ fills up 
$(s_*,\infty)$ as $n\to\infty$. Since $\lim_{n\to\infty} \eta_n=0$, we get $\lim_{h \downarrow 0} 
r_Q(h) = 0$. Combining (\ref{eq:anyway1}--\ref{eq:anyway2}), we obtain that 
\begin{align} 
H(\lceil Q \rceil_h \mid Q_{\lceil \rho \rceil_h,\WM}) 
= \sup_{N \in \N} 
\frac1N h(\pi_N \lceil Q \rceil_h \mid \pi_N Q_{\lceil \rho \rceil_h,\WM}) 
\leq H( Q \mid Q_{\rho,\WM}) + r_Q(h)
\end{align}
and, finally, 
\begin{align} 
&\limsup_{h \downarrow 0} 
H(\lceil Q \rceil_h \mid Q_{\lceil \rho \rceil_h,\WM}) 
+ (\alpha-1) m_{\lceil Q \rceil_h} H(\Psi_{\lceil Q \rceil_h, h} \mid \WM) \notag\\
&\qquad \leq H( Q \mid Q_{\rho,\WM}) +  (\alpha-1) m_Q H(\Psi_Q \mid \WM).
\end{align}

\smallskip

The truncated case, where $F$ is replaced by $F_{0,\tr}$, etc., can be handled 
analogously.
\end{proof}


\subsubsection{Approximation of $\Psi_Q$}

The approximation in \eqref{eq:PsiQ} is stronger than just weak convergence.

\begin{lemma}
\label{lemma:PsiQ:TVlim}
For $Q \in \mathcal{P}^{\mathrm{inv,fin}}(F^\N)$, 
\begin{align} 
\label{eq:PsiQTVlimit}
\lim_{T\to\infty} \sup_{A \subset C[0,\infty)\; \text{measurable}} 
\bigg| \Psi_Q(A) - \frac{1}{T} 
\int_0^T \big(\kappa(Q) \circ (\theta^s)^{-1}\big)(A) \,ds\bigg| = 0,
\end{align}
i.e., the convergence in \eqref{eq:PsiQ} holds in total variation. 
\end{lemma}

\begin{proof} 
Note that, by shift-invariance, 
\begin{align} 
\Psi_Q(A) = \frac{1}{N m_Q} 
\E_Q \left[ \int_0^{\tau_N} \1_A\big( \theta^s \kappa(Y) \big) \, ds \right], \qquad N\in\N.
\end{align}
Suppose that $Q$ is also ergodic. Then $\lim_{N\to\infty} \tau_N/N = m_Q$ $Q$-a.s.\ 
and in $L^1(Q)$. Hence, for given $\varepsilon > 0$ we can find a $T_0(\varepsilon)$ 
such that, for $T \geq T_0(\varepsilon)$, 
\begin{align}
\label{eq:Qerg.cons1}
\E_Q\Big[ \Big| \frac{\tau_{N(T)}-T}{m_Q N(T)} \Big| \Big] 
+ \Big| \frac{T}{m_Q N(T)} - 1 \Big| \leq \varepsilon,
\end{align}
where $N(T) = \lceil T/m_Q \rceil$. Thus, for $T \geq T_0(\varepsilon)$ and any measurable 
$A \subset C[0,\infty)$, we have 
\begin{align} 
& \bigg| \Psi_Q(A) - \frac{1}{T} 
\int_0^T \big(\kappa(Q) \circ (\theta^s)^{-1}\big)(A) \,ds\bigg| \notag \\
& \leq \frac{1}{m_Q N(T)} \bigg| 
\E_Q\left[ \int_0^{\tau_{N(T)}} \1_A\big( \theta^s \kappa(Y) \big) \, ds 
- \int_0^{T} \1_A\big( \theta^s \kappa(Y) \big) \, ds \right] \bigg| \notag \\
& \qquad + \bigg| \Big(\frac{1}{m_Q N(T)}-\frac1T \Big) 
\int_0^{T} \1_A\big( \theta^s \kappa(Y) \big) \, ds \bigg| 
\leq \E_Q\left[ \Big| \frac{\tau_{N(T)}-T}{m_Q N(T)} \Big| \right] 
+ \Big| \frac{T}{m_Q N(T)} - 1 \Big| \leq \varepsilon,
\end{align}
i.e., \eqref{eq:PsiQTVlimit} holds. 

If $Q$ is not ergodic, then use the ergodic decomposition 
\begin{equation}
Q = \int_{\mathcal{P}^{\mathrm{erg,fin}}(F^\N)} Q' \, W_Q(dQ')
\end{equation}
and note that 
\begin{equation}
m_Q = \int_{\mathcal{P}^{\mathrm{erg,fin}}(F^\N)} m_{Q'} \, W_Q(dQ'),
\quad \Psi_Q = \int_{\mathcal{P}^{\mathrm{erg,fin}}(F^\N)} \frac{m_{Q'}}{m_Q}
\, \Psi_{Q'} \, W_Q(dQ')
\end{equation}
(see also \cite[Section 6]{BiGrdHo10}). 
We can choose $N(T)$ so large that the set of $Q'$s for which \eqref{eq:Qerg.cons1} 
holds (with $Q$ replaced by $Q'$) has $W_Q$-measure arbitrarily close to $1$.
\end{proof}


\subsection{Continuity of the ``letter part'' of the rate function under truncation: discrete-time}
\label{subs:towards.Ique.tr.cont}

In this section we consider a discrete-time scenario as in \cite{BiGrdHo10}: $\rho \in \mathcal{P}(\N)$, 
$E$ is a Polish space, $\nu \in \mathcal{P}(E)$, the sequence of words $(Y^{(i)})_{i\in\N}$ with 
discrete lengths has reference law $q_{\rho,\nu}^{\otimes\N}$ with $q_{\rho,\nu}$ as in 
\cite[Eq.\ (1.4)]{BiGrdHo10}. The following lemma extends \cite[Lemma~A.1]{BiGrdHo10} to 
Polish spaces (in \cite{BiGrdHo10} it was only proved and used for finite $E$, and without explicit 
control of the error term). Via coarse-graining, this lemma was used in the proof of 
Proposition~\ref{prop:Ique.tr.cont}.

\begin{lemma} 
\label{lem:trcontinuous}
Let $Q \in \mathcal{P}^{\mathrm{fin}}(\widetilde{E}^\N)$ 
and $0 < \varepsilon < \tfrac12$. 
Let $\tr \in \N$ be so large that 
\begin{align} 
\label{eq:mQtr.qb}
\E_Q\Big[ \big( |Y^{(1)}|-\tr \big)_+ \Big] < \frac{\varepsilon}{2} m_Q.
\end{align}
Then 
\begin{align} 
\label{eq:HPsiQtr.qb}
(1-\varepsilon) \big( H(\Psi_{[Q]_\tr} \mid \nu^{\otimes \N}) 
+ b(\varepsilon) \big) \leq H(\Psi_Q \mid \nu^{\otimes \N})
\end{align}
with $b(\varepsilon) = -2\varepsilon + [\varepsilon \log\varepsilon 
+ (1-\varepsilon) \log (1-\varepsilon)]/(1-\varepsilon) $, satisfying 
$\lim_{\varepsilon \downarrow 0} b(\varepsilon)=0$. In particular, 
\begin{align} 
\label{eq:HPsiQtrlim.disc}
\lim_{\tr\to\infty} H(\Psi_{[Q]_\tr} \mid \nu^{\otimes \N}) 
= H(\Psi_Q \mid \nu^{\otimes \N}). 
\end{align}
\end{lemma}

\begin{proof} 
We can assume w.l.o.g.\ that $H(\Psi_Q \mid \nu^{\otimes\N}) < \infty$ 
for otherwise \eqref{eq:HPsiQtr.qb} is trivial and \eqref{eq:HPsiQtrlim.disc} 
follows from lower-semicontinuity of specific relative entropy.
\smallskip 

First, assume that $Q$ is ergodic, then $\Psi_Q$ is ergodic as 
well (see \cite[Remark~5]{Bi08}).
For $\Psi \in \mathcal{P}^{\mathrm{erg}}(E^\N)$ and $\delta \in (0,1)$, 
\begin{align} 
\label{eq:HPsinu.typset0}
H(\Psi \mid \nu^{\otimes \N}) 
& = \lim_{L\to\infty} -\frac{1}{L} \log \Big( \inf\big\{ \nu^{\otimes L}(B) 
\colon\, B \subset E^L, (\pi_L \Psi)(B) \geq 1-\delta \big\} \Big), \\
\label{eq:HPsinu.typset1}
& = \lim_{L\to\infty} \sup\Big\{ -\frac{1}{L} \log \nu^{\otimes L}(B) \, 
\colon\, B \subset E^L, (\pi_L \Psi)(B) \geq 1-\delta \Big\}.
\end{align}
This replaces the asymptotics of the covering number and its relation to specific entropy 
for ergodic measures on discrete shift spaces that was employed in the proof of 
\cite[Lemma A.1]{BiGrdHo10}, and can be deduced with bare hands from the 
Shannon-McMillan-Breiman theorem. Indeed, asymptotically optimal $B$'s are of the form 
$\{ \frac1L \log \frac{d\pi_L\Psi}{d\nu^{\otimes L}} \in H(\Psi \mid \nu^{\otimes \N}) \pm \epsilon\}$: 
Put $f_L = \frac{d\pi_L \Psi}{d\nu^{\otimes L}}$ and set $B_L = \{ \frac1L \log f_L > H(\Psi \mid 
\nu^{\otimes \N}) - \epsilon \}$. Then $(\pi_L \Psi)(B_L) \to 1$ by the Shannon-McMillan-Breiman, 
and $\nu^{\otimes L}(B_L) = \int_{B_L} \frac1{f_L} d\pi_L\Psi \leq \exp[-L(H(\Psi \mid \nu^{\otimes \N})
- \epsilon)]$, i.e., the right-hand side of \eqref{eq:HPsinu.typset1} is $\geq H(\Psi \mid \nu^{\otimes \N})$. 
For the reverse inequality, consider any $B \subset E^L$ with $(\pi_L\Psi)(B) \geq \tfrac12$, say. Set 
$B'=B \cap \{ \frac1L \log f_L < H(\Psi \mid \nu^{\otimes \N}) + \epsilon\}$. Then $\pi_L\Psi(B') \geq 
\tfrac13$ for $L$ large enough and $\nu^{\otimes L}(B) \geq \nu^{\otimes L}(B') \geq 
\exp[-L(H(\Psi \mid \nu^{\otimes \N}) + \epsilon)] \pi_L\Psi(B')$. Hence the right-hand side of 
\eqref{eq:HPsinu.typset1} is also $\leq H(\Psi \mid \nu^{\otimes \N})$.
        
To check \eqref{eq:HPsiQtr.qb}, fix $\varepsilon>0$. For $L$ sufficiently large, we construct a set 
$B_L \subset E^L$ such that $\pi_L \Psi_Q(B_L) \geq \tfrac12$ and $\nu^{\otimes L}(B_L) \leq 
\exp[ - L(1-\varepsilon)(b_L(\varepsilon) + H(\Psi_{[Q]_\tr} \mid \nu^{\otimes \N}))]$, 
i.e., 
\begin{align}
-\frac1L \log \nu^{\otimes L}(B_L) \geq 
(1-\varepsilon) \big[ H(\Psi_{[Q]_\tr} \mid \nu^{\otimes \N}) + b_L(\varepsilon) \big], 
\end{align}
where $\lim_{L\to\infty} b_L(\varepsilon) = b(\varepsilon)$. Via \eqref{eq:HPsinu.typset1} applied to 
$\Psi=\Psi_Q$, this yields \eqref{eq:HPsiQtr.qb}.

To construct the sets $B_L$, we proceed as follows. Put $N= \lceil (1+2\varepsilon) L/m_Q \rceil$. 
By the ergodicity of $Q$ (see \cite[Section 3.1]{BiGrdHo10} for analogous arguments), we can find 
a set $A \subset \widetilde{E}^N$ such that 
\begin{align} 
&\forall\, (y^{(1)},\dots,y^{(N)}) \in A \colon\, \notag \\
\label{eq:BL.prop.1}
& \hspace{2em} | \kappa(y^{(1)},\dots,y^{(N)}) | \geq L(1+\varepsilon), \;\; 
|y^{(1)}| \leq \tr, \;\; \sum_{i=1}^N (|y^{(i)}|-\tr)_+ < \varepsilon L, \\ 
\label{eq:BL.prop.2}
& \hspace{2em} \E_Q\Big[ |Y^{(1)}| \1_A(Y^{(1)},\dots,Y^{(N)}) \big] 
\geq (1-\varepsilon) m_Q, 
\end{align}
and the set 
\begin{align} 
B'_L = B'_L(A) &= \Big\{ \pi_L \big( \theta^i \kappa([y^{(1)}]_\tr,\dots,[y^{(N)}]_\tr)\big)\colon \,\notag\\
&\qquad (y^{(1)},\dots,y^{(N)}) \in A, i=0,1,\dots,|y^{(1)}|-1 \Big\} \subset E^L 
\end{align}
satisfies 
\begin{align} 
\pi_L\Psi_{[Q]_\tr}(B'_L) \geq \frac12, \quad 
\nu^{\otimes \lceil L(1-\varepsilon) \rceil}(\pi_{\lceil L(1-\varepsilon) \rceil} B'_L) 
\leq \exp\big[ -L(1-\varepsilon) 
\big(H(\Psi_{[Q]_\tr} \mid \nu^{\otimes \N})-2\varepsilon\big)\big].
\end{align}
Here, use \eqref{eq:mQtr.qb} in (\ref{eq:BL.prop.1}--\ref{eq:BL.prop.2}), and note that 
$N\big(1-\tfrac{\varepsilon}{2}\big)m_Q \sim (1+2\varepsilon)\big(1-\tfrac{\varepsilon}{2}\big)L 
\geq (1+\varepsilon) L$ and $N \tfrac{\varepsilon}{2} m_Q \sim (1+2\varepsilon)\tfrac{\varepsilon}{2} 
L < \varepsilon L$ as $L\to\infty$.

For $I \subset \{1,\dots,L\}$, $x \in E^L$ and $y \in E^{|I|}$, write $\mathsf{ins}_I(x; y) \in E^{L+|I|}$ 
for the word of length $L+|I|$ consisting of the letters from $y$ at index positions in $I$ and the letters 
from $x$ at index positions not in $I$, with the order of letters preserved within $x$ and within $y$ 
(the word $y$ is inserted in $x$ at the positions in $I$). Put 
\begin{align}
B_L = \pi_L\Big( \big\{ \mathsf{ins}_I(x; y) \colon \, 
x \in B'_L, I\subset \{1,\dots,L\}, |I| \leq \varepsilon L, y \in E^{|I|} \big\} \Big).
\end{align}
Then $\pi_L\Psi_Q(B_L) \geq \frac12$ by construction. Furthermore, for fixed $I\subset \{1,\dots,L\}$ 
with $|I|=k \le \varepsilon L$, 
\begin{align} 
\nu^{\otimes L} \Big( \pi_L\big( \big\{ \mathsf{ins}_I(x; y)\colon \, x \in B'_L, y \in E^k \big\} \big)\Big) 
= \nu^{\otimes L} \big( \pi_{L-k}(B'_L) \big) 
\leq \nu^{\otimes \lceil L(1-\varepsilon \rceil]}(\pi_{\lceil L(1-\varepsilon) \rceil} B'_L),
\end{align}
and hence 
\begin{align} 
\nu^{\otimes L}(B_L) & \leq [\varepsilon L] {L \choose [\varepsilon L]} 
\exp\big[ -L(1-\varepsilon) \big(H(\Psi_{[Q]_\tr} \mid \nu^{\otimes \N})-2\varepsilon\big)\big] \notag \\
& = \exp\big[ - L(1-\varepsilon) \big(b_L(\varepsilon) + H(\Psi_{[Q]_\tr} \mid \nu^{\otimes \N})\big) \big]
\end{align}
with $b_L(\varepsilon) = - \frac{1}{(1-\varepsilon) L}(\log [\varepsilon L] + \log {L \choose [\varepsilon L]}) 
- 2\varepsilon$, which satisfies $\lim_{\varepsilon \downarrow 0} b_L(\varepsilon)= b(\varepsilon)$.

It remains to prove \eqref{eq:HPsiQtrlim.disc}. Since $\wlim_{\tr\to\infty} \Psi_{[Q]_\tr} = \Psi_Q$, we 
have $\liminf_{\tr\to\infty} H(\Psi_{[Q]_\tr} \mid \nu^{\otimes \N}) \geq H(\Psi_Q \mid \nu^{\otimes \N})$, 
while the reverse inequality $ \limsup_{\tr\to\infty} H(\Psi_{[Q]_\tr} \mid \nu^{\otimes \N}) \leq 
H(\Psi_Q \mid \nu^{\otimes \N})$ follows from (\ref{eq:mQtr.qb}--\ref{eq:HPsiQtr.qb}) and the 
fact that $\lim_{\tr\to\infty} \E_Q[( |Y^{(1)}|-\tr )_+] = m_Q$ by dominated convergence.
\smallskip

For non-ergodic $Q$, decompose as in \cite[Eqs.(6.1)--(6.3)]{BiGrdHo10}, use the above argument 
on each of the ergodic components, and use the fact that specific relative entropy is affine.
\end{proof}


\section{Existence of specific relative entropy}
\label{contrelentr}

In this section we prove \eqref{eq:SREwrtWM}. For technical reasons, we consider the 
two-sided scenario. The argument is standard, but the fact that time is continuous requires
us to take care.  

\begin{proof}
Let $\Omega = \tilde{C}(\R)$ be the set of continuous functions $\omega\colon\, \R \to \R$ 
with $\omega(0)=0$, which is a Polish space e.g.\ via the metric $d(\omega, \omega') = 
\int_\R e^{-|t|} \big(|\omega(t)-\omega'(t)| \wedge 1\big) dt$. The shifts on $\Omega$ are 
$\theta^t \omega(\cdot) = \omega(\cdot+t)-\omega(t)$. A probability measure $\Psi$ on 
$\Omega$ has stationary increments when $\Psi = \Psi \circ (\theta^t)^{-1}$ for all $t \in \R$. 
For an interval $I \subset \R$ denote $\mathcal{F}_I = \sigma(\omega(t)-\omega(s)\colon\, 
s,t \in I)$. $\Psi_I$ denotes $\Psi$ restricted to $\mathcal{F}_I$. Write $\WM$ for the Wiener 
measure on $\Omega$, i.e., the law of a (two-sided) Brownian motion.

Let $\Psi \in \mathcal{P}(\Omega)$ with stationary increments be given and assume that 
$h(\Psi_{[0,T]} \mid \WM_{[0,T]}) < \infty$ for all $T>0$. To verify \eqref{eq:SREwrtWM}, 
we imitate well-known arguments from the discrete-time setup (see e.g.\ 
Ellis~\cite[Section IX.2]{El85}). 

For $I_1$, $I_2$ disjoint intervals in $\R$, denote by $\kappa^\Psi_{I_1, I_2}\colon\, 
\Omega \times \mathcal{F}_{I_2} \to [0,1]$ a regular version of the conditional law of (the 
increments of) $\Psi$ on $I_2$, given the increments in $I_1$, i.e., for fixed $\omega$, 
$\kappa^\Psi_{I_1, I_2}(\omega, \cdot)$ is a probability measure on $\mathcal{F}_{I_2}$, 
for fixed $A \in \mathcal{F}_{I_2}$, $\kappa^\Psi_{I_1, I_2}(\cdot, A)$ is an 
$\mathcal{F}_{I_1}$-measurable function, and $\kappa^\Psi_{I_1, I_2}(\omega, A)$ is a 
version of $\E_{\Psi}[\1_A | \mathcal{F}_{I_1}]$. When $I_1=\emptyset$, 
$\kappa^\Psi_{\emptyset, I_2}(\omega, A) = \Psi_{I_2}(A)$. Similarly, define 
$\kappa^\WM_{I_1, I_2}$ (which is simply $\kappa^\WM_{I_1, I_2}(\omega, A) = \WM_{I_2}(A)$ 
by the independence of the Brownian increments).

Put 
\begin{align}
a_{I_1,I_2} = \int_\Omega \Psi(d\omega_1)
\int_\Omega  \kappa^\Psi_{I_1, I_2}(\omega_1, d\omega_2) \,
\log\left[ \frac{d \kappa^\Psi_{I_1, I_2}(\omega_1, \cdot)}
{d \kappa^\WM_{I_1, I_2}(\omega_1, \cdot)}(\omega_2)\right],   
\end{align}
the expected relative entropy of the conditional distribution under $\Psi$ on $\mathcal{F}_{I_2}$ 
given $\mathcal{F}_{I_1}$ w.r.t.\ Wiener measure on $\mathcal{F}_{I_2}$). We have $a_{I_1,I_2} 
< \infty$ for bounded intervals, because of the assumption of finite relative entropy of $\Psi$ w.r.t.\ 
$\WM$ on compact time intervals. By stationarity, $a_{I_1,I_2}=a_{t+I_1,t+I_2}$ for any $t$.

Let $I_1' \subset I_1$, note that $\kappa^\Psi_{I_1, I_2}(\omega, \cdot) \ll \kappa^\Psi_{I_1', I_2}
(\omega, \cdot)$ for $\Psi$-a.e.\ $\omega$, and $\kappa^\WM_{I_1, I_2}(\omega, \cdot) = 
\kappa^\WM_{I_1', I_2}(\omega, \cdot) = \WM_{I_2}(\cdot)$. By the consistency property of 
conditional distributions, we have 
\begin{align} 
a_{I_1',I_2} = \int_\Omega \Psi(d\omega_1) \int_\Omega 
\kappa^\Psi_{I_1, I_2}(\omega_1, d\omega_2) 
\log \left[\frac{d \kappa^\Psi_{I_1', I_2}(\omega_1, \cdot)}
{d \kappa^\WM_{I_1', I_2}(\omega_1, \cdot)}(\omega_2)\right].
\end{align}
Indeed,
\begin{equation}
\int_\Omega \Psi(d\omega_1) \int_\Omega \kappa^\Psi_{I_1, I_2}(\omega_1, d\omega_2) 
f(\omega_1,\omega_2) 
= \int_\Omega \Psi(d\omega_1) \int_\Omega \kappa^\Psi_{I'_1, I_2}(\omega_1, d\omega_2) 
f(\omega_1,\omega_2)
\end{equation}
for any function $f(\omega_1,\omega_2)$ that is $\mathcal{F}_{I_1'} \otimes 
\mathcal{F}_{\R}$-measurable. Hence 
\begin{align} 
&a_{I_1,I_2} - a_{I_1',I_2} \\
& = \int_\Omega \Psi(d\omega_1) \int_\Omega 
\kappa^\Psi_{I_1, I_2}(\omega_1, d\omega_2)  
\bigg( \log \left[\frac{d \kappa^\Psi_{I_1, I_2}(\omega_1, \cdot)}
{d \kappa^\WM_{I_1, I_2}(\omega_1, \cdot)}(\omega_2)\right]  
- \log \left[\frac{d \kappa^\Psi_{I_1', I_2}(\omega_1, \cdot)}
{d \kappa^\WM_{I_1', I_2}(\omega_1, \cdot)}(\omega_2)\right] \bigg) \notag \\
\label{eq:hdeccont1}
& = \int_\Omega \Psi(d\omega_1) 
\int_\Omega \kappa^\Psi_{I_1, I_2}(\omega_1, d\omega_2) \, 
\log \left[\frac{d \kappa^\Psi_{I_1, I_2}(\omega_1, \cdot)}
{d \kappa^\Psi_{I_1', I_2}(\omega_1, \cdot)}(\omega_2)\right]  \geq 0
\end{align}
because the inner integral is $h( \kappa^\Psi_{I_1, I_2}(\omega_1, \cdot) \mid \kappa^\Psi_{I_1', I_2}
(\omega_1, \cdot)) \geq 0$. Choosing $I_1'=\emptyset$, \eqref{eq:hdeccont1}, we get $a_{I_1, I_2} 
\geq a_{\emptyset, I_2} = h( \Psi_{I_2}  \mid \WM_{I_2})$.

Observe 
\begin{align} 
\frac{d \Psi_{(0,s+t]}}{d \WM_{(0,s+t]}}(\omega) 
= \frac{d \Psi_{(0,t]}}{d \WM_{(0,t]}}(\omega) 
\, \frac{d \kappa^\Psi_{(0,t],(t,s+t]}(\omega, \cdot)}
{d \kappa^\WM_{(0,t],(t,s+t]}(\omega, \cdot)}(\omega) \quad \Psi_{(0,s+t]}-\text{a.s.},
\end{align}
take logarithms and integrate w.r.t.\ $\Psi$ (using consistency of conditional expectation 
on the right-hand side), to obtain 
\begin{align} 
h\big( \Psi_{(0,s+t]}  \mid \WM_{(0,s+t]} \big) = h\big( \Psi_{(0,t]}  \mid \WM_{(0,t]} \big) 
+ a_{(0,t], (t,s+t]}
\geq h\big( \Psi_{(0,t]}  \mid \WM_{(0,t]} \big) + h\big( \Psi_{(0,s]}  \mid \WM_{(0,s]} \big). 
\end{align}
Thus, the function $(0,\infty) \ni t \mapsto h( \Psi_{(0,t]}  \mid \WM_{(0,t]})$ is super-additive, 
and \eqref{eq:SREwrtWM} follows from Fekete's lemma. 
\end{proof}

Under $\kappa^\Psi_{(-\infty,0], (0,h]}$, the coordinate process will be a Brownian 
motion with a (possibly complicated) drift process $U_t = \int_0^t u_s\, ds$, where
$(u_t)_{t \geq 0}$ can be chosen adapted, and
\begin{equation}
\E_\Psi\big[ h(\kappa^\Psi_{(-\infty,0], (0,h]} \mid \WM_{(0,h]}) \big] = \E_\Psi\big[{\textstyle \int_0^h} u_s^2 \,ds \big]
\end{equation} 
(see F\"ollmer~\cite{Foe86}).


\end{document}